\documentclass[12pt]{amsart}
\usepackage{
   a4wide, 
  amsmath,
  amsfonts,
  amssymb,
  graphicx, 
  times,
  color,
  hyphenat, 
  mathabx, 
 mathtools,
  stmaryrd, datetime, bbm,
}
\usepackage[cal=boondoxo]{mathalfa}

\usepackage{tikz,etex,bbm,xspace}
\usepackage[colorlinks=true, pdfstartview=FitV, linkcolor=blue, citecolor=blue, urlcolor=blue]{hyperref}%
\newcommand{\fdot}[2][]{\node[xshift=-.27cm, yshift=-.27cm, anchor = south west] at (#2) {\tikz[very thick]{\node  [fill=white, draw=black,circle,inner sep=2pt]  at (0,0){};\node[anchor = south west] at (.025,.025){\small $#1$}; }} }

\input xy
\xyoption{all}
\DeclareMathOperator{\nh}{NH}
\DeclareMathOperator{\bnh}{BNH}
\DeclareMathOperator{\BA}{BA}
\DeclareMathOperator{\nc}{NC}
\DeclareMathOperator{\spn}{span}
\DeclareMathOperator{\mat}{Mat}

\newcommand\s{{\mathfrak{s}}}

\newcommand\h{{\mathcal{h}}}
\newcommand\E{{\sf{E}}}
\newcommand\F{{\sf{F}}}

\newcommand\Q{{\sf{Q}}}
\newcommand{\bV}{\raisebox{0.03cm}{\mbox{\footnotesize$\textstyle{\bigwedge}$}}}
\newcommand{\g}{{\mathfrak{g}}}

\newcommand{\slt}{{\mathfrak{sl}_{2}}}

\newcommand{\und}[1]{{\underline{#1}}}

\newcommand{\cP}{\mathcal{P}}
\newcommand{\brak}[1]{\langle #1\rangle}
\newcommand{\pp}[1]{(\!( #1 )\!)}
\newcommand{\n}{\noindent}

\def\E{{\sf{E}}}
\def\F{{\sf{F}}}

\DeclareMathOperator{\amod}{\mathrm{-}mod}
\DeclareMathOperator{\smod}{\mathrm{-}smod}

\DeclareMathOperator{\lfmods}{\mathrm{-}smod_{lf}}

\DeclareMathOperator{\fgpmod}{\mathrm{-}pmod_{fg}}

\DeclareMathOperator{\prmods}{\mathrm{-}psmod_{lfg}}

\DeclareMathOperator{\fdmod}{\mathrm{-}mod_{fd}}

\DeclareMathOperator{\End}{End}

\DeclareMathOperator{\gdim}{gdim}
\DeclareMathOperator{\sdim}{sdim}

\DeclareMathOperator{\im}{im}

\newtheorem{thm}{Theorem}[section]
\newtheorem{lem}[thm]{Lemma}
\newtheorem{cor}[thm]{Corollary}
\newtheorem{prop}[thm]{Proposition}
\newtheorem{exe}[thm]{Example}

\theoremstyle{definition}
\newtheorem{defn}[thm]{Definition}

\newtheorem{rem}[thm]{Remark}

\definecolor{myred}{rgb}{0.9,0,0}
\definecolor{mygreen}{rgb}{0,0.7,0}
\definecolor{myblue}{rgb}{0,0,0.7}
\definecolor{orchid}{RGB}{143,40,194}

\newcommand{\tM}{\mathfrak{M}}

\newcommand{\un}{\mathbbm{1}}

\newcommand{\bN}{\mathbb{N}}
\newcommand{\bZ}{\mathbb{Z}}
\newcommand{\bQ}{\mathbb{Q}}

\newcommand{\bC}{\mathbb{C}}

\newcommand{\cC}{\mathcal{C}}
\newcommand{\cD}{\mathcal{D}}

\newcommand{\cS}{\mathcal{S}}

\DeclareMathOperator{\Hom}{Hom}
\DeclareMathOperator{\HOM}{HOM}
\DeclareMathOperator{\END}{END}

\DeclareMathOperator{\id}{Id}

\DeclareMathOperator{\grk}{grk}

\DeclareMathOperator{\Ind}{Ind}
\DeclareMathOperator{\Res}{Res}

\newcommand{\xra}[1]{\xrightarrow{#1}}

\catcode`\@=11
\long\def\@makecaption#1#2{%
    \vskip 10pt
    \setbox\@tempboxa\hbox{%
\small{#1: }\ignorespaces #2}%
    \ifdim \wd\@tempboxa >\captionwidth {%
        \rightskip=\@captionmargin\leftskip=\@captionmargin
        \unhbox\@tempboxa\par}%
      \else
        \hbox to\hsize{\hfil\box\@tempboxa\hfil}%
    \fi}
\newdimen\@captionmargin\@captionmargin=2\parindent
\newdimen\captionwidth\captionwidth=\hsize
\catcode`\@=12
\newcommand{\boxAn}{\tikz[anchor = center, xscale = 1.2, very thick]{
	    	\draw (0,-.5)-- (0,.5);
	 	\draw (.25,-.5)-- (.25,.5);
	 	\draw (.75,-.5)-- (.75,.5);
	 	\draw (1,-.5)-- (1,.5);
		\node (rect) at (-.125,0) [anchor = west, draw,fill=white,minimum width=1.45cm,minimum height=.5cm] {$n$};
		\node at (.525,.5) {\small $\dots$}; \node at (.525,-.5) {\small $\dots$};
	} }

\newcommand{\boxAnu}{\tikz[anchor = center, xscale = 1.2, very thick]{
	    	\draw (0,-.5)-- (0,.5);
	 	\draw (.25,-.5)-- (.25,.5);
	 	\draw (.75,-.5)-- (.75,.5);
	 	\draw (1,-.5)-- (1,.5);
		\node (rect) at (-.125,0) [anchor = west, draw,fill=white,minimum width=1.45cm,minimum height=.5cm] {$n$};
		\node at (.525,.5) {\small $\dots$}; \node at (.525,-.5) {\small $\dots$};
	} }

\newcommand{\boxAnm}{\tikz[anchor = center, xscale = 1.2, very thick]{
	    	\draw (0,-.5)-- (0,.5);
	 	\draw (.25,-.5)-- (.25,.5);
	 	\draw (.75,-.5)-- (.75,.5);
	 	\draw (1,-.25)-- (1,.5);
	 	\draw +(1,-.25) .. controls (1,-.5) ..  + (1.2,-.5);
		\node (rect) at (-.125,0) [anchor = west, draw,fill=white,minimum width=1.45cm,minimum height=.5cm] {$n$};
		\node at (.525,.5) {\small $\dots$}; \node at (.525,-.5) {\small $\dots$};
	} }

\newcommand{\boxmAn}{\tikz[anchor = center, xscale = 1.2, very thick]{
	    	\draw (0,-.5)-- (0,.5);
	 	\draw (.25,-.5)-- (.25,.5);
	 	\draw (.75,-.5)-- (.75,.5);
	 	\draw (1,-.5)-- (1,.25);
	 	\draw +(1,.25) .. controls (1,.5) ..  + (1.2,.5);
		\node (rect) at (-.125,0) [anchor = west, draw,fill=white,minimum width=1.45cm,minimum height=.5cm] {$n$};
		\node at (.525,.5) {\small $\dots$}; \node at (.525,-.5) {\small $\dots$};
	} }

\newcommand{\boxAnp}{\tikz[anchor = center, xscale = 1.2, very thick]{
	    	\draw (0,-.5)-- (0,.5);
	 	\draw (.25,-.5)-- (.25,.5);
	 	\draw (.75,-.5)-- (.75,.5);
	 	\draw (1,-.5)-- (1,.5);
	 	\draw (1.25,-.5)-- (1.25,.5);
		\node (rect) at (-.125,0) [anchor = west, draw,fill=white,minimum width=1.75cm,minimum height=.5cm] {$n+1$};
		\node at (.525,.5) {\small $\dots$}; \node at (.525,-.5) {\small $\dots$};
	} }

\newcommand{\boxnAnp}{\tikz[anchor = center, xscale = 1.2, very thick]{
	    	\draw (0,-.5)-- (0,.5);
	 	\draw (.25,-.5)-- (.25,.5);
	 	\draw (.75,-.5)-- (.75,.5);
	 	\draw (1,-.5)-- (1,.5);
	 	\draw (1.25,-.5) -- (1.25,.25);
	 	\draw +(1.25,.25) .. controls (1.25,.5) ..  + (1.45,.5);
		\node (rect) at (-.125,0) [anchor = west, draw,fill=white,minimum width=1.75cm,minimum height=.5cm] {$n+1$};
		\node at (.525,.5) {\small $\dots$}; \node at (.525,-.5) {\small $\dots$};
	} }

\newcommand{\boxAnpn}{\tikz[anchor = center, xscale = 1.2, very thick]{
	    	\draw (0,-.5)-- (0,.5);
	 	\draw (.25,-.5)-- (.25,.5);
	 	\draw (.75,-.5)-- (.75,.5);
	 	\draw (1,-.5)-- (1,.5);
	 	\draw (1.25,-.25) -- (1.25,.5);
	 	\draw +(1.25,-.25) .. controls (1.25,-.5) ..  + (1.45,-.5);
		\node (rect) at (-.125,0) [anchor = west, draw,fill=white,minimum width=1.75cm,minimum height=.5cm] {$n+1$};
		\node at (.525,.5) {\small $\dots$}; \node at (.525,-.5) {\small $\dots$};
	} }

\newcommand{\boxnAnpn}{\tikz[anchor = center, xscale = 1.2, very thick]{
	    	\draw (0,-.5)-- (0,.5);
	 	\draw (.25,-.5)-- (.25,.5);
	 	\draw (.75,-.5)-- (.75,.5);
	 	\draw (1,-.5)-- (1,.5);
	 	\draw (1.25,-.25) -- (1.25,.25);
	 	\draw +(1.25,-.25) .. controls (1.25,-.5) ..  + (1.45,-.5);
	 	\draw +(1.25,.25) .. controls (1.25,.5) ..  + (1.45,.5);
		\node (rect) at (-.125,0) [anchor = west, draw,fill=white,minimum width=1.75cm,minimum height=.5cm] {$n+1$};
		\node at (.525,.5) {\small $\dots$}; \node at (.525,-.5) {\small $\dots$};
	} }

\title{On 2-Verma modules for quantum $\slt$}
\author{Gr\'egoire Naisse}
\address{Institut de Recherche en Math\'ematique et Physique\\
Universit\'e Catholique de Louvain\\ 
Chemin du Cyclotron 2\\ 
1348 Louvain-la-Neuve\\ 
Belgium}
\email{gregoire.naisse@uclouvain.be}
\author{Pedro Vaz}
\address{Institut de Recherche en Math\'ematique et Physique\\
Universit\'e Catholique de Louvain\\ 
Chemin du Cyclotron 2\\ 
1348 Louvain-la-Neuve\\ 
Belgium}
\email{pedro.vaz@uclouvain.be}
\begin{document}
 \usetikzlibrary{decorations.pathreplacing,backgrounds,decorations.markings}
\tikzset{wei/.style={draw=red,double=red!40!white,double distance=1.5pt,thin}}
\tikzset{bdot/.style={fill,circle,color=blue,inner sep=3pt,outer sep=0}}
%
\newdimen\captionwidth\captionwidth=\hsize
%
%
\begin{abstract}
  In this paper we study the superalgebra $A_n$, introduced by the authors in previous
  work on categorification of Verma modules for quantum $\slt$. 
  The superalgebra $A_n$ is akin to the nilHecke algebra, and shares similar properties.
  In particular, we prove a uniqueness result about 2-Verma modules on $\Bbbk$-linear 2-categories.  
\end{abstract}
\maketitle

{
\tableofcontents
}
%
%
\pagestyle{myheadings}
\markboth{\em\small Gr\'egoire Naisse and Pedro Vaz}{\em\small On 2-Verma modules for quantum $\slt$}
%
%
%
\section{Introduction}\label{sec:intro}
%
%

In previous work~\cite{naissevaz1} the authors gave a categorification of Verma modules
for quantum $\slt$ 
using cohomology rings of infinite-dimensional, complex Grassmannians and their Koszul duals.
The categorical $\slt$-action therein was built from a setup using pullback maps induced by a
geometric correspondence with an infinite one-step partial flag variety, as
pioneered in the finite case by Chuang--Rouquier~\cite{CR} and Frenkel--Khovanov--Stroppel~\cite{fks}. 
Moreover, in~\cite{naissevaz1}, a superalgebra $A_n$ akin to the nilHecke algebra $\nh_n$ was introduced, 
which governs part of the higher structure in this Verma categorification. 
In fact, the superalgebra $A_n$ is $\bZ$-graded and contains $\nh_n$ as a 
graded subsuperalgebra concentrated 
in even parity.   
It also admits a diagrammatic presentation, as $\nh_n$, but contains additionally a set of anticommuting generators,
not present in $\nh_n$.  
The latter allows the introduction of an extra grading, which is a key ingredient in the categorification of Verma modules, and turn $A_n$ into a $\bZ\times\bZ$-graded superalgebra.

\smallskip 

The superalgebra $A_n$ was obtained as an endomorphism superring of the functors
$\F^n=\F\circ\dotsm\circ\F$
and $\E^n=\E\circ\dotsm\circ\E$ realizing the categorical $\slt$-action.
In contrast, in this paper we describe another way to obtain $A_n$,
namely from an action of the symmetric group
$S_n$ on a supercommutative ring $R$. 
That is, in the same way as the nilHecke algebra: if $R^{S_n}$ denotes the subsuperring of $S_n$-invariants, then
$A_n$ is the endomorphism superring of $R$ as an $R^{S_n}$-supermodule, and therefore is isomorphic to  
an algebra of matrices
with coefficients in $R^{S_n}$ (see~\cite[\S3.2]{L1}).  
Additionally, we give an explicit description of $R^{S_n}$ in terms of Schur polynomials in $R$.
These polynomials are indexed by pairs consisting of a partition and a strict partition, and 
contain commuting and anticommuting variables. 

We would also like to mention that the superring $R^{S_n}$ was studied independently by Appel, Egilmez, Hogancamp, and Lauda
in~\cite{AEHL}, where it is given a combinatorial description, different than ours.

\smallskip

As usual, induction and restriction functors give rise to
functors $\F_n : A_n\smod \rightarrow A_{n+1} \smod$ and $\E_n  : A_{n+1}\smod \rightarrow A_{n} \smod$ on the category of supermodules of 
$A=\oplus_{n \geq 0}A_n$. The latter define a categorical $\slt$-action and are connected
through an exact sequence
\[
0 \rightarrow \F_{n-1}\E_{n-1} \rightarrow \E_{n}\F_{n} \rightarrow
q^{-2n} \lambda \Q_{n+1}\oplus q^{2n}\lambda^{-1} \Pi \Q_{n+1} \rightarrow 0, 
\]
where $\Q_n$ is a functor of tensoring with a polynomial ring in one variable,
$\Pi$ is a parity shift functor,
and $q^{a}\lambda^{b}$ denotes a shift in the bigrading by $(a,b)$. 
The exact sequence above does not split and this is a crucial difference from
the finite-dimensional case of~\cite{CR} and~\cite{fks}. 
With the $\slt$-action above the (suitably defined) Grothendieck group of $A$ is isomorphic to a Verma module for
quantum $\slt$.

\smallskip

Moreover, recall that $\nh_n$ has certain so-called cyclotomic quotients for all $N\in\bN_{0}$,
which are Morita equivalent to cohomology rings of finite-dimensional
Grassmannians,
and play an important role in the categorification of finite-dimensional
representations of $\slt$ (see~\cite{CR} and~\cite{L3}). 
These can be recovered from $A_n$ by using certain differentials $d_N$ on it. 

 \smallskip

Let us mention that in~\cite{AEHL} it was introduced another differential on $A_n$,
whose homology is a quotient of $\nh_n$ which is Morita equivalent to the $GL(N)$-equivariant
cohomology of Grassmannians,
and also categorify these representations. 

\smallskip 

Next, assembling the data of the categorification of a Verma module in a 2-category, we give an axiomatic
definition of a 2-Verma module. 
Following the techniques in~\cite[\S5]{R1} (see also~\cite[\S5]{CR}) we state a uniqueness result
for 2-Verma modules  
whenever they are subcategories of the strict 2-categories of all bigraded, $\Bbbk$-linear, supercategories,
with the 1-morphisms being functors and where the 2-morphisms are grading preserving natural
transformations.

\begin{rem}
We give a diagrammatic presentation of the superalgebra $A_n$ which is slightly different
from the one given in~\cite{naissevaz1}: in this paper ``white dots'' corresponding to odd elements
are placed in the regions of diagrams rather than on the strands. 
This allows writing elaborate relations involving these generators in a compact form.
\end{rem}

\begin{rem}
This paper introduces some of the notions and techniques that are used in a general context, 
done in the sequel~\cite{naissevaz3}, where
versions of KLR algebras are constructed which allow categorification of Verma modules
for all symmetrizable quantum Kac--Moody algebras.
\end{rem}

%
%

\subsection*{Acknowledgments}
G.N. is a Research Fellow of the Fonds de la Recherche Scientifique - FNRS, under Grant no.~1.A310.16.
P.V. was supported by the Fonds de la Recherche Scientifique - FNRS under Grant no.~J.0135.16.

The authors would like to thank the referee for carefully reading the manuscript, providing valuable
comments, finding quite a few mistakes, and  for the constructive and
helpful report.

%
%
\section{The superalgebra $A_n$}\label{sec:algA}


\subsection{Reminders on super structures}

Recall that a \emph{superring} is a $\bZ/2\bZ$-graded ring. Let $A=A_0\oplus A_1$ be a superring. 
We will call the $\bZ/2\bZ$-grading the \emph{parity} and use the notation $p(a)$ to indicate the parity
of a homogeneous element $a\in A$. 
Elements with parity 0 are called \emph{even}, elements of parity 1 are called \emph{odd}. 
Whenever we refer to an element of $A$ as even or odd, we will always be assuming that it
is homogeneous.
A \emph{subsuperring} of $A$ is a subring which is itself a superring, that is, the canonical inclusion
preserves the parity.  
The \emph{supercenter} $Z_s(A)$ of $A$ is the set of all elements of $A$ which supercommute with all 
elements of $A$, that is $Z_s(A)=\{ x\in A \vert xa=(-1)^{p(x)p(a)}ax\text{ for all }a\in A\}$. 
A left $A$-\emph{supermodule} $M$ is a $\bZ/2\bZ$-graded module over $A$ such that
$A_iM_j\subseteq M_{i+i}$ ($i,j\in\bZ/2\bZ$).

If $A$ has additional gradings, then we say $A$ is a \emph{graded superring} (or multigraded superring).  
In this context we can speak of (multi)graded supermodules.
The notion of (graded, multigraded) \emph{superalgebra} and related structures are defined in the same way.

The (graded) supermodules over a (graded) superring $A$, together with (degree and) parity preserving
morphisms of supermodules form an abelian (graded) supercategory $A\smod$. In this supercategory,
we write $\Pi : A\smod \rightarrow A\smod$ for the parity shift functor (i.e. the action of $\bZ/2\bZ$).

\subsection{Supercommutative polynomials, symmetric group action and Demazure operators}
\label{ssec:algan}

Let $\und{x}_n=(x_1,\dotsc,x_n)$ be even variables and
$\und{\omega}_n=(\omega_1,\dotsc,\omega_n)$ be odd variables, and 
form the commutative superring $R=\bZ[\und{x}_n] \otimes \bV^\bullet(\und{\omega}_{n})$, where
$\bV^\bullet(\und{\omega}_{n})$ is the exterior ring in the variables $\und{\omega}_n$ and coefficients in $\bZ$.  
Introduce a $\bZ\times\bZ$-grading in $R$, declaring that $\deg(x_i)=(2,0)$ and $\deg(\omega_i)=(-2i,2)$.
The first grading is referred to as the $q$-grading and the second as the $\lambda$-grading  
(see~\cite[\S 9.1]{naissevaz1} for details).

Let $S_n$ be the symmetric group on $n$ letters, which we view as being generated as a Coxeter group
with generators $s_i$. These correspond to the simple transpositions  $(i\   i{+}1)$, and we use these two descriptions interchangeably throughout. 
As explained in~\cite[\S 9.3]{naissevaz1}, $S_n$ acts (from the left) on $R$ as follows,
\begin{equation}\label{eq:SnActsonR}
\begin{split}
  s_i(x_j) &= x_{s_i(j)} , 
  \\
s_i(\omega_j) &= \omega_{j} + \delta_{i,j}(x_i-x_{i+1})\omega_{i+1} .
\end{split}
\end{equation}

\begin{prop}\label{prop:Snaction}
The assignement in~\eqref{eq:SnActsonR} gives $R$ the structure of an $S_n$-module. 
\end{prop}

This action respects the bigrading as well as the parity, as one easily checks. 

Using the $S_n$-action above, we introduce the \emph{Demazure operators} $\partial_i$ on $R$
for all $1\leq i \leq n-1$ in the usual way, as  
$$\partial_i(f) = \frac{f-s_i(f)}{x_i - x_{i+1}}.$$
The operator $\partial_i$ is an even operator, and it is homogeneous of
bidegree $\deg(\partial_i)=(-2,0)$. 

From the formula for $\partial_i$ above one sees that 
$s_i\partial_i(f)=\partial_i(f)$ and $\partial_i(s_if)=-\partial_i(f)$ for all $i$,
so $\partial_i$ is in fact an operator from $R$ to the subring $R^{s_i}\subset R$
of invariants under the transposition $(i\ i{+}1)$.

The following is proved in~\cite[\S 9.2]{naissevaz1}, and can be viewed as a direct consequence of Proposition~\ref{prop:Snaction} together with the definition of $\partial_i$.

\begin{lem}\label{lem:demazure}
The action of the Demazure operators on $R$ satisfies the Leibniz rule, 
\[
\partial_i(fg)=\partial_i(f)g+s_i(f)\partial_i(g), 
\]
for all $f,g\in R$ and for $1\leq i\leq n-1$,
and the relations 
\begin{gather*}\allowdisplaybreaks
\partial_i^2(f) = 0,\mspace{30mu}
\partial_i\partial_{i+1}\partial_i(f) =
\partial_{i+1}\partial_i\partial_{i+1}(f) ,
\\[1.5ex] 
\partial_i\partial_j(f) =\partial_j\partial_i(f)\mspace{25mu}\text{for }\ \vert i-j\vert > 1,
\\[1.5ex] 
x_i\partial_i(f)  - \partial_i(x_{i+1}f)  = f,
\mspace{30mu}
\partial_i(x_{i}f) - x_{i+1}\partial_i(f)  = f, 
\end{gather*}
and 
\begin{align*} 
\partial_i(\omega_k f) &=  \omega_k\partial_i(f)\mspace{25mu}\text{for }\ k\neq i,
  \\[1.5ex]
  \partial_i\bigl( (\omega_i-  x_{i+1}\omega_{i+1})f \bigr)  &= (\omega_i - x_{i+1}\omega_{i+1})\partial_i(f),
\end{align*} 
for all  $f \in R$ and $1\leq i\leq n-1$. 
\end{lem}

For a reduced expression $\vartheta=s_{i_1}\dotsm s_{i_r}$ in terms of simple transpositions
we put
\[
\partial_\vartheta = \partial_{i_1}\dotsm \partial_{i_r}. 
\]
From Lemma~\ref{lem:demazure} it follows that this is independent of the choice of the reduced decomposition. 
As in the polynomial case~\cite{manivel} the action of the Demazure operators on $R$ also satisfies
\[
\partial_{\vartheta}\partial_{\vartheta'} =
\begin{cases}
  \partial_{\vartheta\vartheta'}, & \text{if }\ell(\vartheta\vartheta')=\ell(\vartheta)+\ell(\vartheta'),
  \\
 0, & \text{else,} 
\end{cases}
\]
where $\ell$ denotes the length function on $S_n$.

\begin{defn}
  We define $A_n$ to be the bigraded $\bZ$-superalgebra of operators on $R$ generated by the Demazure operators
  $\partial_i$
for $1\leq i\leq n-1$, together with multiplication by elements of $R$. 
We put $A_0=\bZ$ and define 
\[
A = \bigoplus\limits_{n\in\bN_{0}} A_n .
\]
\end{defn}

Recall that the \emph{nilHecke} algebra $\nh_n$ is the $\bZ$-algebra generated by
$T_1,\dotsm ,T_{n-1}$ and $x_1,\dotsm ,x_n$, with relations 
\begin{gather}
  \label{eq:nhrel}
    T_i^2 = 0, \mspace{40mu} T_iT_j =T_jT_i \text{\ \ if }\vert i-j\vert>1, \mspace{40mu} T_iT_{i+1}T_i = T_{i+1}T_iT_{i+1}, 
\\[1ex]  \label{eq:nhrel2}
    x_ix_j = x_jx_i,
\\[1ex]
  \label{eq:nhre3l}
    T_ix_j = x_jT_i \text{\ \ if }j-i\neq 0,1,\mspace{40mu}  T_ix_{i}-x_{i+1}T_i =1, \mspace{40mu} T_ix_{i+1}-x_iT_i =-1 .
        \end{gather}
This is a graded algebra with $\deg_q(x_i)=2$ and $\deg_q(T_i)=-2$.
We extend this grading to a $\bZ\times\bZ$-grading whose first, i.e. $q$-grading, is the one from before,
and whose second, i.e. $\lambda$-grading, is trivial. 

\medskip

In~\cite[\S 9]{naissevaz1} we gave a presentation of $A_n$ by generators and relations. In our convention, the notation $A \rtimes B$ for two (graded) $R$-(super)algebras means we take the free product (i.e. the algebra generated by $R$-linear combinations of words with characters in $A$ and $B$ together with multiplication given by concatenation and reduction), which we quotient by some specified relations.

\begin{prop}\label{prop:smashAnNH}
  The superalgebra $A_n$ is isomorphic to the bigraded super\-al\-gebra 
\[ A_n \cong \nh_n  \rtimes \,\bV^\bullet(\und{\omega}_{n}) \]
with the relations
$x_i\omega_j = \omega_jx_i$\ for all $i,j$, and 
\[
   T_i\omega_j = \omega_jT_i \text{\ \ if }i\neq j,
   \mspace{40mu} T_i(\omega_i-x_{i+1}\omega_{i+1}) = (\omega_i-x_{i+1}\omega_{i+1})T_i .
\]
\end{prop}
For a reduced decomposition $\vartheta=s_{i_1}\dotsm s_{i_k}\in S_n$  we put $T_\vartheta=T_{i_1}\dotsm T_{i_k}$, which by~\eqref{eq:nhrel} is a well-defined element of $\nh_n$.

There is a canonical inclusion $\nh_n\hookrightarrow A_n$, given by the inclusion $\nh_n \hookrightarrow  \nh_n \rtimes 1 \subset \nh_n  \rtimes \,\bV^\bullet(\und{\omega}_{n})$, and therefore $\nh_n$ is a graded subsuperalgebra of $A_n$ concentrated in even parity and with
trivial $\lambda$-grading.  

\medskip

The following is~\cite[Proposition 9.1]{naissevaz1}.
\begin{prop}\label{prop:firstbasis} 
  The superalgebra $A_n$ is a free $\bZ$-module with the \emph{sets}
  \begin{gather}\nonumber
    \{ x_1^{k_1}\dotsm x_n^{k_n}\omega_1^{\ell_1}\dotsm\omega_n^{\ell_n}T_\vartheta
    \colon  k_i\in\bN_0, \ell_i\in \{0,1\}, \vartheta\in S_n\},
  \\ \nonumber
  \{ x_1^{k_1}\dotsm x_n^{k_n}T_\vartheta\omega_1^{\ell_1}\dotsm\omega_n^{\ell_n}
  \colon  k_i\in\bN_0, \ell_i\in \{0,1\}, \vartheta\in S_n\},
  \intertext{and}\nonumber
  \{T_\vartheta x_1^{k_1}\dotsm x_n^{k_n}\omega_1^{\ell_1}\dotsm\omega_n^{\ell_n}
  \colon  k_i\in\bN_0, \ell_i\in \{0,1\}, \vartheta\in S_n\},
  \end{gather}
  being basis. 
\end{prop}

Let $\nc_n\subset A_n$ denote the \emph{nilCoxeter algebra}, which is the subalgebra
generated by the $T_i$'s. From Proposition~\ref{prop:firstbasis} it follows at once that
we have a decomposition $A_n \cong R\otimes\nc_n$ as a $\bZ$-(super)module.

We introduce the notations $[n] = q^{n-1} + q^{n-3} \dots + q^{1-n}$ and $[n]! = [n][n-1]\dotsm[1]$. Let $\grk_{\bZ}(A_n) \in \bZ\llbracket q, q^{-1}, \lambda, \lambda^{-1} \rrbracket[\pi]/(\pi^2-1)$ denote  the graded rank of $A_n$ viewed as a $\bZ/2\bZ \times \bZ \times \bZ$-graded $\bZ$-module, with $\pi$ being the parity grading (i.e. the $\bZ/2\bZ$-degree). Then as a direct consequence of Proposition~\ref{prop:firstbasis}  we obtain the following:

\begin{cor}\label{cor:grrankAn}
  The graded rank of $A_n$ is
  \[
\grk_{\bZ}(A_n) = q^{-n(n-1)/2}[n]!  \prod_{j=1}^n\frac{1+\pi\lambda^2q^{-2j}}{1-q^2},
\]  
where we interpret the fractions as power sums. 
\end{cor}

\subsubsection{Tight monomials}\label{ssec:tightmonom}

\begin{defn}
  We say a monomial in $A_n$ is \emph{tight} if it can be written as a word involving only
symbols from the alphabet
$\{T_1,\dotsc ,T_{n-1},x_1,\dotsc ,x_n,\omega_1\}$.
\end{defn}
Or in other words, a monomial is tight if it can be written without using any $\omega_i$ for $i > 1$.
Note that for $i > 1$ we have
\begin{align*}
\omega_{i} &= x_{i-1} T_{i-1} \omega_{i} - T_{i-1}   x_{i}\omega_{i} , 
\intertext{and}
T_{i-1} \omega_{i} &= - T_{i-1} \omega_{i-1} T_{i-1}, 
\end{align*}
and so, $\omega_i$ can be written as a combination of tight monomials
by recursion.

\medskip 

We now introduce another basis, expressed in terms of tight monomials, which is in some sense more natural from the point of view of categorification, and which will be used in~\S\ref{sec:cataction}.
For each $\vartheta \in S_n$ we choose a left-adjusted reduced expression $ \vartheta = s_{i_r} \dotsm s_{i_1}$, where \emph{left-adjusted} means that $i_r + \dots + i_1$ is minimal.

Recall that, given a reduced expression, we can obtain all other reduced expressions of the same permutation by applying a sequence of moves $s_is_{i+1}s_i \leftrightarrow s_{i+1}s_is_{i+1}$ and $s_i s_j \leftrightarrow s_js_i$ for $|i-j| > 1$. Moreover, given two reduced expressions $s_{i_r} \dotsm s_{i_1}$ and $s_{i_r'} \dotsm s_{i_1'}$ such that the second one can be obtained from the first by a single move $s_is_{i+1}s_i \mapsto s_{i+1}s_is_{i+1}$ or $s_i s_j \mapsto s_js_i$ then
\[
\min_{t \in \{0, \dots, r\}} s_{i_t}\dotsm s_{i_1}(k) \le \min_{t \in \{0, \dots, r\}} s_{i'_t}\dotsm s_{i'_1}(k),
\]
for all $k \in \{1,\dots, n\}$.
 Hence, a reduced expression $s_{i_r} \dotsm s_{i_1}$ is left-adjusted if and only if
\begin{equation} \label{eq:leftadjusted}
\min_{t \in \{0, \dots, r\}} s_{i_t}\dotsm s_{i_1}(k) \le \min_{t \in \{0, \dots, r\}} s_{i'_t}\dotsm s_{i'_1}(k),
\end{equation}
for all $k\in \{1,\dots, n\}$ and all other reduced expression $s_{i'_r} \dotsm s_{i'_1}$ of the same permutation. In this condition, we write $\min_\vartheta(k) = \min_{t \in \{0, \dots, r\}} s_{i_t}\dotsm s_{i_1}(k)$. Note that a left-adjusted reduced expression always exists and is unique up to distant permutations (i.e. moves $s_i s_j \leftrightarrow s_js_i$ for $|i-j| > 1$). In particular, we can obtain a left-adjusted reduced expression for any permutation by taking its representative in the coset decomposition
\begin{equation}\label{eq:cosetdecomp}
S_{n} = \bigsqcup_{a=1}^{n} S_{n-1}s_{n-1} \dotsm s_a,
\end{equation}
applied recursively. 

\begin{exe}
The permutation $(1\ 3\ 2\  4)$ of $\{1, 2, 3, 4\}$ admits as left-adjusted reduced expression the word $s_1s_2s_1s_3s_2$ which comes from the summand $S_2 s_3 s_2$ in the first step of the recursive decomposition. Note that $s_1s_2s_3s_1s_2$ is also left-adjusted while  $s_2s_1s_2s_3s_2$ and $s_2s_1s_3s_2s_3$ are not.
\end{exe}

Suppose $s_{i_r} \dotsm s_{i_1}$ is a left-adjusted reduced expression of $\vartheta$. Then we can choose for each $k \in \{1, \dots, n\}$ an index $t_k$ such that
\[
s_{i_{t_k}}\dotsm s_{i_1}(k) =  \mbox{$\min_\vartheta(k)$}.
\]
Clearly this choice is not necessarily unique and we can have $t_{k} = t_{k'}$ for $k \neq k'$. 
However, it defines a partial order on the set $\{1, \dots, n\}$ where we say $k \prec k'$ if $t_{k} \le t_{k'}$. We extend this order arbitrarily and we write $<_t$ for it. There is a bijective map $s : \{1, \dots, n\} \rightarrow \{1, \dots, n\}$ which sends $k < k'$ to $s(k) <_t s(k')$, so that $t_{s(k)} \le t_{s(k')}$.
For $k \in\{1, \dots, n+1\}$, we put
\[
\vartheta^k = s_{i_{t_{s(k)}}}\dotsm s_{i_{t_{s(k-1)}}},
\]
where it is understood that $t_{s(0)} = 0$ and $t_{s(n+1)} = r$. It defines a partition of the reduced expression of $\vartheta$. Moreover, it is constructed such that
\[
\vartheta^k \cdot \vartheta^{1}(s(k)) =  \mbox{$\min_\vartheta(s(k))$}.
\]

\begin{exe}Consider again $\vartheta = s_1s_2s_1s_3s_2$ with $i_1=2, i_2 =3, i_3 = 1, i_4=2, i_5=1$. We can choose $t_1 = 0, t_2 = 0,t_3 = 3$ and $t_4 = 5$ (we could also have chosen $t_1=1$ or $t_1=2$ and also $t_3 = 4$, changing what follows).  Then we can have $s(1) = 1$ (or $2$), $s(2) = 2$ (or $1$), $s(3) = 3$ and $s(4) = 4$, with $\vartheta^1 = 1, \vartheta^2 = 1$, $\vartheta^3 = s_1s_3s_2$ and $\vartheta^4 = s_1s_2$. 
\end{exe}

\smallskip

Now we define the element 
\begin{equation}\label{eq:thetaa}
\theta_a = T_{a-1} \dotsm T_1 \omega_1 T_1 \dotsm T_{a-1},
\end{equation}
and we consider the set
\begin{equation}\label{eq:tightbasis}
\left\{
x_1^{k_1} \dotsm x_n^{k_n} T_{ \vartheta^{n+1}}   \theta_{\min_\vartheta(s(n))}^{\ell_{\vartheta(s(n))}}
T_{ \vartheta^{n}} 
\dotsm   \theta_{\min_\vartheta(s(2))}^{\ell_{\vartheta(s(2))}}T_{ \vartheta^{2}}  \theta_{\min_\vartheta(s(1))}^{\ell_{\vartheta(s(1))}}
 T_{\vartheta^1}
 \colon  k_i\in\bN_0, \ell_i\in \{0,1\}, \vartheta\in S_n
\right\},
\end{equation}
where it is understood that $\theta_a^0 = 1$ and $\vartheta = \vartheta^{n+1}\dotsm\vartheta^1$ is as above.
As we will see in Lemma 3.1 ahead, the set in~\eqref{eq:tightbasis} forms a basis of $A_n$ for all involved choices. 

\begin{exe}We take $n=3$. Then \eqref{eq:tightbasis} can be given by the following elements (we use \eqref{eq:cosetdecomp}):
\begin{align*}
&p(\und x_3) \omega_1^{\ell_1} \theta_2^{\ell_2} \theta_3^{\ell_3},  \\
&p(\und x_3) \omega_1^{\ell_1} T_1 \omega_1^{\ell_2} \theta_3^{\ell_3}, \\
&p(\und x_3) \omega_1^{\ell_1} \theta_2^{\ell_2} T_2 \theta_2^{\ell_3}, \\
&p(\und x_3)  \omega_1^{\ell_1} \theta_2^{\ell_2} T_2  T_1 \omega_1^{\ell_3}, \\
&p(\und x_3)  \omega_1^{\ell_1} T_1  \omega_1^{\ell_2} T_2  \theta_2^{\ell_3}, \\
&p(\und x_3)  \omega_1^{\ell_1}T_1  \omega_1^{\ell_2 }T_2  T_1  \omega_1^{\ell_3}
\end{align*}
where $p(\und x_3)$ is a polynomial in the variables $x_1, x_2, x_3$, and $(\ell_1, \ell_2, \ell_3) \in \{0,1\}^3$. 
\end{exe}

\medskip

Using this basis, we give another construction of the superalgebra $A_n$.

\begin{defn}
Let $A'_n$ be the bigraded $\bZ$-superalgebra given by 
\[
A'_n = \nh_n \rtimes \bV^\bullet(\omega_1)
\]
with relations
\begin{gather}
\nonumber
  x_i \omega_1 = \omega_1 x_i, \mspace{50mu}  T_j \omega_1 = \omega_1 T_j,
\\[1ex] 
\label{eq:omega1}
 T_1 \omega_1 T_1 \omega_1 = - \omega_1 T_1 \omega_1 T_1,
\end{gather}
for $1 \le i \le n$ and $2 \le j \le n$, and where $\nh_n $ is concentrated in parity $0$ and $\lambda$-degree 0, and $\omega_1$ has parity $1$ and degree $(-2,2)$.
\end{defn}

Of course, one can view the monomials in (\ref{eq:tightbasis}) as elements of $A'_n$.

\begin{prop}\label{prop:tightAnp}
The monomials in (\ref{eq:tightbasis}) span $A'_n$.
\end{prop}

\begin{proof}
  By the nilHecke relations~\eqref{eq:nhre3l} and because $\omega_1$ commutes with the $x_i$'s, $A'_n$ is
  generated by monomials with all $x_i$'s to the front (this can be proven by induction on the number of $T_i$'s).
  Therefore, $A'_n$ is spanned by elements of the form
\[
x_1^{k_1} \dotsm x_n^{k_n} T_{\vartheta_r} \omega_1 T_{\vartheta_{r-1}} \omega_ 1 \dotsm
T_{\vartheta_{2}} \omega_1 T_{\vartheta_1},
\]
where $\vartheta_i$'s are reduced expressions.
If $\vartheta_{i+t}\dotsm\vartheta_{i}(1) = 1$, then
$\omega_1T_{\vartheta_{i+t}}  \dotsm  T_{\vartheta_{i+1}} \omega_1 T_{\vartheta_{i}} \omega_1  = 0$ 
by~(\ref{eq:omega1}) and the braid relations~\eqref{eq:nhrel}.
Hence, we can assume $r \le n+1$ and left-adjusted reduced expressions together with the $\theta_i$'s span $A'_n$.
The proof follows by observing that left-adjusted reduced expressions are unique up to
permutation of distant crossings, and by observing the $\theta_i$'s anticommutes up to adding elements with fewer number of $T_i$'s,
because of~(\ref{eq:omega1}) and the braid relations.
\end{proof}

\begin{prop}\label{prop:isoAnAnp}
There is an isomorphism of bigraded superalgebras $A'_n \cong A_n$, sending $x_i \mapsto x_i$, $T_i \mapsto T_i$ and $\omega_1 \mapsto \omega_1$.
\end{prop}

\begin{proof}
We prove $ T_1 \omega_1 T_1 \omega_1 = - \omega_1 T_1 \omega_1 T_1 \in A_n$.
First note that $T_1 \omega_2 \omega_1 = - \omega_1 \omega_2 T_1$ by  the last relation of Lemma~\ref{lem:demazure}. Then we compute
\begin{align*}
T_1 \omega_1 T_1 \omega_1 &=  T_1 \omega_1 \omega_2 T_1 x_2 - T_1 \omega_1 \omega_2 x_2 T_1,  \\
\omega_1 T_1 \omega_1 T_1 &= x_2 T_1 \omega_2 \omega_1 T_1 - T_1 x_2 \omega_2 \omega_1 T_1,
\end{align*}
using the last relation of Lemma~\ref{lem:demazure} again. 
By the nilHecke relations~\eqref{eq:nhrel} we get 
\begin{align*}
- T_1 \omega_1 \omega_2 T_1 x_2 &= T_1 \omega_1 \omega_2 x_1 T_1 - T_1 \omega_1 \omega_2, \\
- x_2 T_1 \omega_2 \omega_1 T_1 &= T_1 x_1 \omega_2 \omega_1 T_1 - \omega_2 \omega_1 T_1.
\end{align*}
Thus, we have a surjective homomorphism $A'_n \twoheadrightarrow A_n$. Since $A_n$ is free and $A'_n$ is generated by the same elements as $A_n$ by Proposition~\ref{prop:tightAnp},  we conclude thanks to this surjective morphism that $A'_n$ is also free of the same rank.
In particular we have constructed an isomorphism. 
\end{proof}

We define recursively $\phi_i \in A'_n$ as follows. We let $\phi_1 = \omega_1$ and
\[ 
\phi_{i+1} = T_i \phi_i  T_i x_{i+1} - x_i T_i \phi_i T_i
=   x_{i+1} T_i \phi_i  T_i -T_i \phi_i T_i  x_i ,
\] 
for all $1\leq i\leq n$.
A direct consequence of the preceding proposition is the following.

\begin{prop}\label{prop:omegainAp}
  The isomorphism $A'_n\to A_n$ in Proposition~\ref{prop:isoAnAnp} sends $\phi_i$ to $\omega_i$. 
\end{prop}

Note this gives the explicit inverse of the surjection $A'_n \twoheadrightarrow A_n$. As a matter of facts, it is a lengthy, but straightforward computation to check that we have 
\[
\phi_i \phi_j = - \phi_j \phi_i, \qquad \phi_i x_j = x_j \phi_i, 
\]
and
\[
T_j \phi_i = \phi_i T_j + \delta_{ij} (T_i \phi_{i+1} x_{i+1} - x_{i+1} \phi_{i+1}  T_i),
\]
in $A'_n$  for all $1 \le i,j \le n$. This gives an alternative proof of Proposition~\ref{prop:isoAnAnp}, and thus of Proposition~\ref{prop:tightAnp}.

\subsection{The subsuperring of $S_n$-invariants and the supercenter of $A_n$}
\label{ssec:Sninvariants}

Let $R^{S_n}\subset R$ be the subsuperring of $S_n$-invariants.
Using $\partial_i^2=0$, we can view $(R,\partial_i)$
as being a chain complex with the homological grading being one-half of the $q$-grading. 
As explained in~\cite[\S 2.1.1]{EKL} for the case of the so-called odd polynomial rings,
the chain complex $(R,\partial_i)$ is contractible for each $i$
and therefore $\ker(\partial_i)=\im(\partial_i)$.
Transporting the arguments therein to our case, it is easy to see that   
\[
R^{S_n} = \bigcap\limits_{i=1}^n\ker(\partial_i) = \bigcap\limits_{i=1}^n\im(\partial_i) .
\]
This superring was defined the same way in~\cite[\S3.1.1]{AEHL}. 
It is straightforward to check that the supercenter of $A_n$ coincides with $R^{S_n}$. 
\begin{prop}\label{prop:isocenter}
There is an isomorphism of graded superrings
\[
Z_s(A_n) \cong R^{S_n}.
\]
\end{prop}

\begin{rem}
The supercenter of $A_n$ is bigger than the center of $A_n$, for example, $\omega_n$ is in
$Z_s(A_n)$ but it is not central since it does not commute with for example $\omega_{n-1}$.
More generally, the center of $A_n$ is the subsuperring of $R^{s_n}$
consisting of all the elements of even parity. 
\end{rem}

We are now going to describe $R^{S_n}$.
It is clear that it contains the ring of symmetric polynomials $\bZ[\und{x}_n]^{S_n}$
as a subsuperring concentrated in even parity and $\lambda$-degree zero.

\begin{defn}
A \emph{superpartition of type $(n\vert k)$} is a pair $(\alpha,\beta)$, where
$\alpha = \{ (\alpha_1,\dotsc,\alpha_n)\in\bN_0^n \colon \alpha_1\geq\alpha_2\geq\dotsm\geq\alpha_n \geq 0\}$
is a partition and
$\beta = \{ (\beta_1,\dotsc,\beta_k)\in\bN_0^k \colon 0\leq \beta_1<\beta_2<\dotsm<\beta_k\leq n\ ,k\leq n\}$
is a (bounded) strict partition\footnote{Note that $\alpha$ and $\beta$ are ordered oppositely.}. 
We say that a strict partition like $\beta$ above has \emph{$k$ parts} if $\beta_1> 0$ and denote
by $\cP_n$ and respectively by $\cP_k^{str}$ the sets of partitions with $n$ parts and the set of
strict partitions with $k$ parts.  
\end{defn}

\begin{rem}
  Superpartitions are known in the literature, and can be used for example to study Jack superpolynomials~\cite[\S2.2]{DLM}.
  The only difference to our setting is that we restrict the size of the strict partition and require $k\leq n$. 
\end{rem}

Let $\vartheta_0 \in S_n$ be the longest element. 

\begin{defn}
For  a superpartition $(\alpha,\beta)$ we define the \emph{Schur superpolynomials\footnote{Similar terminology can be found in the literature, like  \emph{supersymmetric polynomials} or \emph{Schur superfunctions} but they are not related to ours. As far as we can tell, our symmetric polynomials on $R$ are new.} on $R$} as 
\[
\cS_{\alpha,\beta}(\und{x}_n,\und{\omega}_n) =
\partial_{\vartheta_0}\bigl( \und{x}_n^{\delta+\alpha}\omega_{\beta} \bigr) .
\]
where $\und{x}_n^{\delta+\alpha}=x_1^{n-1+\alpha_1}x_2^{n-2+\alpha_2}\dotsm x_n^{\alpha_n}$ and 
$\omega_\beta=\omega_{\beta_1}\dotsm\omega_{\beta_k}$. 
\end{defn}

For $\beta=0$ we recover the usual Schur polynomials on $\bZ[\und{x}_n]$. 
If we denote by $>$ the lexicographic order on strict partitions, then
the polynomials $\cS_{\alpha,\beta}(\und{x}_n,\und{\omega}_n)$ can be written as   
\[
\cS_{\alpha,\beta}(\und{x}_n,\und{\omega}_n) =
\cS_{\alpha,0}(\und{x}_n,\und{\omega}_n)\omega_{\beta} + \sum_{\mu > \beta}c_{\mu}(\und{x}_n)\omega_{\mu}, 
\]
where the coefficients $c_{\mu}(\und{x}_n)$ lie in $\bZ[\und{x}_n]$.
This follows immediately from the definition of $\cS_{\alpha,\beta}(\und{x}_n,\und{\omega}_n)$ and the Leibniz rule
for the Demazures.  

From now on we write $\cS_{\alpha,\beta}$ instead of 
$\cS_{\alpha,\beta}(\und{x}_n,\und{\omega}_n)$ for the sake of simplicity, 
if no confusion can arise. 

\smallskip

It is striking to look for multiplication rules on Schur polynomials on $R$
that generalize the well-known multiplication rules for ordinary Schur polynomials. 
For the particular case of $\alpha$ being the zero partition we can give  a precise formula
(see Proposition~\ref{prop:multip-ai} below).

Let us first look at the case of $\beta=(i)$.  
\begin{lem}\label{lem:S0i}
  We have 
  $\cS_{0,i} = \sum\limits_{\ell \geq i} (-1)^{\ell+i} \h_{\ell-i}(\ell,n)\omega_\ell$,
  where $\mathcal{h}_j(\ell,n)$ is the $j$-th complete homogeneous symmetric polynomial
  in the variables $(x_\ell,\dotsc, x_n)$.  
\end{lem}
\begin{proof}
  Set $\cS_{0,i} = \partial_{\vartheta_0}\bigl( \und{x}_n^{\delta}\omega_{i} \bigr)
  = \sum_{\ell\geq i}c_\ell\omega_\ell$,
  with $c_\ell\in\bZ[\und{x}_n]$. 
  Since  $\sum_{\ell\geq i}c_\ell\omega_\ell \in R^{S_n}$, we have for $i\leq j \leq n-1$  
\begin{equation*}
0 = \partial_j\biggl( \sum_{\ell\geq i}c_\ell\omega_\ell \biggr)     
= -(s_jc_j)\omega_{j+1} + \sum_{\ell\geq i}(\partial_jc_\ell)\omega_\ell .
\end{equation*}
This implies $c_j=\partial_jc_{j+1}$, since $s_j\partial_jf=\partial_jf$ for every $f\in R$.
Hence,
\begin{equation}\label{eq:s0i-cjdemazures}
  c_j = \partial_j\partial_{j+1}\dotsm\partial_{n-1}c_n.
\end{equation}
Now, using the following presentation for $\partial_{\vartheta_0}$,
\[
\partial_{\vartheta_0}=\partial_1(\partial_2\partial_1)(\partial_3\partial_2\partial_1)\dotsm
(\partial_{n-1}\dotsm\partial_2\partial_1), 
\]
one can see that
\begin{equation}\label{eq:s0i-cn}
\cS_{0,i} 
= \sum_{r=i}^{n-1}c_r\omega_r +
(-1)^{n-i}\Bigl( \partial_1(\partial_2\partial_1)\dotsm
  (\partial_{n-2}\dotsm\partial_2\partial_1)s_{n-1}s_{n-2}\dotsm s_i\partial_{i-1}\dotsm\partial_2\partial_1(\und{x}_n^\delta)   \Bigr)\omega_n .
 \end{equation} 

 Using repeatedly 
 \[
 \partial_j\partial_{j-1}\dotsm\partial_1 (x_1^{m-1}x_2^{m-2}\dotsm x_j^{m-j}x_{j+1}^{m-j-1})
 =x_1^{m-2}x_2^{m-3}\dotsm x_j^{m-j-1}x_{j+1}^{m-j-2}
 \]
on~\eqref{eq:s0i-cn} gives 
\[
c_n = (-1)^{n-i}x_n^{n-i} .
\]
Finally, using the formula for $c_j$ in~\eqref{eq:s0i-cjdemazures} we get
\begin{equation*}
c_j=(-1)^{j+i}\h_{j-i}(j,n) , 
\end{equation*}
as claimed.
\end{proof}

\smallskip 

For strict partitions $\beta$, $\beta'$ define the product
$\beta\beta'$ as the unique strict partition that can be formed from the
set $\{\beta\}\cup\{\beta'\}$.
Set $\epsilon_{\beta,\beta'}$ as 
the length of the minimal permutation taking the ordered set $\{\beta\}\cup\{\beta'\}$
to $\beta\beta'$.

\begin{exe}
For $\beta=(1,3,4)$ and $\beta'=(2,6)$ we have
\[
\{\beta\}\cup\{\beta'\}=\{1,3,4,2,6\},\mspace{15mu} \beta\beta'=(1,2,3,4,6) , 
\]
and 
$\epsilon_{(1,3,4),(2,6)}=2$. 
\end{exe}

\begin{prop}\label{prop:multip-ai}
  The superpolynomials $\cS_{0,\beta}$
  satisfy the following multiplication rule:
  \[
  \cS_{0,\beta}\cS_{0,\beta'} =   (-1)^{\epsilon_{\beta,\beta'}}\cS_{0,\beta\beta'} .
  \]
\end{prop}
\begin{proof}
  We consider first the case of $\beta=(i)$, $\beta'=(j)$ with $i\neq j$. 
  First note that 
  $\partial_{\vartheta_0}\bigl(\und{x}_n^\delta \h_{k-j}(k,n)\bigr)$ is zero unless $k=j$,
  and that $\partial_{\vartheta_0}(\omega_i\omega_j)=0$. 
  Hence, 
  \begin{align*}
  \cS_{0,i}\cS_{0,j} &=
\partial_{\vartheta_0}\bigl( \und{x}_n^{\delta}\omega_i  \cS_{0,j} \bigr)
\\[0.8ex]
&= \sum\limits_{k\geq j}(-1)^{n-j}\partial_{\vartheta_0}\bigl(\und{x}_n^\delta \h_{k-j}(k,n)\omega_i\omega_j\bigr) 
\\
&= \partial_{\vartheta_0}\bigl(\und{x}_n^\delta\omega_i\omega_j\bigr)
+
\sum\limits_{k > j}(-1)^{n-j}\und{x}_n^\delta \h_{k-j}(k,n)\partial_{\vartheta_0}(\omega_i\omega_j)
\\
&=
\partial_{\vartheta_0}\bigl(\und{x}_n^\delta\omega_i\omega_j\bigr)
\\[0.8ex]
&= (-1)^{\epsilon_{i,j}} \cS_{0,ij} .
\end{align*}  
  The same reasoning proves that for general $\beta\in \cP_k^{str}$,
one has $\cS_{0,\beta}\cS_{0,\ell}= (-1)^{\epsilon_{\beta,\ell}}\cS_{0,\beta \ell}$. 
The claim now follows by induction on the length of $\beta'$, since 
$\cS_{0,\beta}\cS_{0,ij}=\cS_{0,\beta}\cS_{0,i}\cS_{0,j}$ for $i<j$. 
\end{proof}

\begin{prop}\label{prop:basis-aii}
  The $\bZ$-module $(R^{S_n})_{(\bullet,2k)}\subset R^{S_n}$ consisting of all elements
  with $\lambda$-degree equal to $2k$ has a $\bZ[\und{x}_n]^{S_n}$-basis given by
  \[
\{ \cS_{0,\nu} \colon \nu \in\cP_k^{str} \}. 
  \]
\end{prop}
\begin{proof}
Denote by $<$ the lexicographic order on $\cP_k^{str}$ and    
 let $z_{\nu}=b_{\nu}\omega_{\nu}+\sum_{{\mu} > \nu} b_{\mu}\omega_{\mu}\in R^{S_n}$,
  with $b_{\rho}\in\bZ[\und{x}_n]$, be homogeneous of $\lambda$-degree $2k$.
  It is not hard to see that $b_{\nu}\in\bZ[\und{x}]^{S_n}$.
  Moreover, we have 
  \begin{equation}\label{eq:zsym}
z_{\nu} - b_{\nu}\cS_{0,\nu} = \sum_{ \mu> \nu }b'_{\mu}\omega_{\mu}\in R^{S_n}
\end{equation}
for some $b'_{\mu}\in \bZ[\und{x}_n]$. 
The claim follows by recursive application this argument
to the right-hand-side of~\eqref{eq:zsym}. 
\end{proof}

\smallskip

\begin{cor}~\label{prop:basis-aii}
There is an isomorphism of bigraded superrings 
  \[
R^{S_n} \cong \bZ[\und{x}_n]^{S_n}\otimes \bV^\bullet(\cS_{0,1},\dotsc, \cS_{0,n}) .
\]
\end{cor}
\begin{proof}
This follows directly from Proposition~\ref{prop:multip-ai} and Proposition~\ref{prop:basis-aii}.
\end{proof}

Corollary~\ref{prop:basis-aii} together with the well-known fact that elementary symmetric polynomials generate $\bZ[\und{x}_n]^{S_n}$ shows that  
\[
\grk_{\bZ}(R^{S_n}) = \prod_{j=1}^n\frac{1+\pi\lambda^2q^{-2j}}{1-q^{2j} }= q^{-n(n-1)/2}\frac{1}{[n]!}  \prod_{j=1}^n\frac{1+\pi\lambda^2q^{-2j}}{1-q^2},
\]  
with the same notations as in~\ref{cor:grrankAn}.

\begin{prop}
The Schur superpolynomials form a homogeneous $\bZ$-linear basis of $R^{S_n}$. 
\end{prop}
\begin{proof}
We first note that the Schur superpolynomials of $R$  
are linearly independent over $\bZ$, since   
\begin{equation}\label{eq:schurdecomp}
\cS_{\alpha,\beta}(\und{x}_n,\und{\omega}_n) =
\cS_{\alpha}(\und{x}_n)\omega_{\beta} + \text{higher terms} ,
\end{equation}
and the classical fact that the $\cS_{\alpha}(\und{x}_n)$'s are linearly independent over $\bZ$.
Using~\eqref{eq:schurdecomp} we compute   
\[
\grk_\bZ
\bigl(
\spn_{\bZ}\bigl\{ \cS_{\alpha,\beta}(\und{x}_n,\und{\omega}_n) \colon\alpha\in\cP_n,\beta \in \cP^{str}_k = k \bigr\}
\bigr)
=
\pi^k\lambda^{2k}q^{-n-1}[n-1]\grk_{\bZ}\bigl(\bZ[\und{x}_n]^{S_n}\bigr),
\]
which matches the graded rank of $(R^{S_n})_{2k}$ over $\bZ$, the latter computed from Proposition~\ref{prop:basis-aii}.
This shows the Schur superpolynomials span $R^{S_n}$ over $\bZ$. 
\end{proof}

\subsection{Labeled $\omega_k$'s} \label{ssec:labeledomegas}

 Define recursively for $a \in \bN_0$
\begin{align*}
\omega_k^a &= \omega_{k-1}^{a-1} - x_k \omega_k^{a-1},
\end{align*}
with  $\omega_k^0 = \omega_k$ and $\omega_0 = 0$.
The degrees of these elements are
\[
\deg(\omega_k^a) = (2(a-k), 2).
\]

\begin{exe}
 We have
\begin{align*}
\omega_4^2 &= \omega_2 - (x_3 + x_4)\omega_3 + x_4^2 \omega_4,\\
\omega_2^3 &= (x_1^2 + x_1x_2 + x_2^2)\omega_1 - x_2^3 \omega_2 ,
\end{align*}
with degrees $(-4,2)$ and $(2,2)$, respectively.
\end{exe}

\begin{lem}\label{lem:labomega}
The following holds in $A_n$: 
\begin{equation*}
\omega_k^a = \sum_{\ell = 1}^{k} (-1)^{a+k+\ell} \mathcal{h}_{a+\ell-k}(\ell,k)  \omega_{\ell} .
\end{equation*}
\end{lem}

\begin{proof}
Clearly, the equation holds for $k = 1$, where $\mathcal{h}_a(1,1) = x_1^a$. 
For $k > 1$ we have that
\[
\omega_k^a = (-1)^a x_k^a \omega_k + \sum_{b=0}^{a-1} (-1)^b x_k^{b} \omega_{k-1}^{a-1-b} . 
\]
Thus, by induction,
\begin{align*}
\omega_k^a &= (-1)^a x_k^a \omega_k + \sum_{b=0}^{a-1} (-1)^{a+k}  x_k^{b}  \sum_{\ell=1}^{k-1} (-1)^\ell \mathcal{h}_{a-b-k+\ell}(\ell, k-1) \omega_{\ell} \\
&=  (-1)^a  x_k^a \omega_k +  \sum_{\ell=1}^{k-1} (-1)^{a+k+\ell} \mathcal{h}_{a-k+\ell}(\ell,k) \omega_{\ell} ,
\end{align*}
which finishes the proof.
\end{proof}

Lemma~\ref{lem:labomega} together with Lemma~\ref{lem:S0i} implies that $\omega_n^{i}  = \cS_{0,n-i}   \in R^{S_n}$ for all $ 0 \le i <n $. Moreover, the labeled $\omega_k^a$'s interact with the $T_i$'s in the same way as the unlabeled
$\omega_k$'s, which formally translates to the following:

\begin{prop} \label{prop:relomegaa}
  In $A_n$ we have
\begin{align*}
T_i \omega_k^a &= \omega_k^a T_i , \qquad i \ne k, \\
T_i (\omega_i^a -x_{i+1}\omega_{i+1}^a) &= (\omega_i^a -x_{i+1}\omega_{i+1}^a) T_i,
\end{align*}
for all $a \in \bN_0$ and all $k,i \in \{1, \dots, n\}$.
\end{prop}

\begin{proof}
  The proof follows by induction on $a$, together with the fact that a symmetric polynomial in $x_i,x_{i+1}$
  commutes with $T_i$, and the relation
  $T_i(\omega_i-x_{i+1}\omega_{i+1}) = (\omega_i-x_{i+1}\omega_{i+1})T_i$. 
\end{proof}

From the above we can rephrase Proposition~\ref{prop:basis-aii} in terms of the labeled $\omega_k$'s. 
\begin{cor}\label{cor:leftomegas}
There is an isomorphism of $\bZ$-algebras 
$ 
R^{S_n} \cong \bZ[\und{x}_n]^{S_n}\otimes \bV^\bullet(\omega_n^0, \omega_n^1,\dotsc, \omega_n^{n-1})$.
\end{cor}

It can also be useful to
invert the formula in Lemma~\ref{lem:S0i} and write the $\omega_k$'s as a 
$\bZ[\und{x}_n]$-linear combination of the $\omega_n^a$'s (or equivalently, of the $\cS_{0,i}$'s).
\begin{prop}\label{prop:invert-S0i}
  We have 
\begin{equation*}
\omega_k = \sum_{\ell = 0}^{n-k}  \mathcal{e}_\ell(k+1,n) \omega_n^{n-k-\ell},
\end{equation*}
where $\mathcal{e}_\ell(k+1,n)$ is the $\ell$-th elementary symmetric polynomial in the variables $(x_{k+1}, \dots, x_n)$.
\end{prop}

\begin{proof}
One can prove
\[
\omega_k^a = \sum_{\ell = 0}^{n-k}  \mathcal{e}_\ell(k+1,n) \omega_n^{n-k-\ell+a}
\]
by backwards induction on $k$, beginning with $k = n$ and descending to zero using the equality
\[
\omega_k^a = \omega_{k+1}^{a+1} + x_{k+1} \omega_{k+1}^a.
\]
We leave the details to the reader. 
\end{proof}

\subsection{The $R^{S_n}$-supermodule structure of $R$}\label{ssec:supermodule}
We introduce some notation and recall some results from the polynomial case.
Recall that the Schubert polynomial $\s_\vartheta$ for $\vartheta\in S_n$ is defined by  
\[
\s_{\vartheta}(\und{x}_n) = \partial_{\vartheta^{-1}\vartheta_0}(x_1^{n-1}x_2^{n-2}\dotsm x_{n-1}) , 
\]
with $\vartheta_0\in S_n$ denoting again the longest element. 
It is well-known that both $\{\s_\vartheta (\und{x}_n)\}_{\vartheta\in S_n}$
and $\{x_1^{a_1}\dotsm x_n^{a_n}\}_{0\leq a_i\leq n-i}$
give integral basis of $\bZ[\und{x}_n]$ as a left (and as a right) $\bZ[\und{x}_n]^{S_n}$-module.   
Let
\[
U_n = \spn_{\bZ}\{ \s_\vartheta (\und{x}_n) \colon \vartheta\in S_n \}
  = \spn_{\bZ}\{ x_1^{a_1}\dotsm x_n^{a_n} \colon 0\leq a_i\leq n-i  \} .
  \]

\begin{prop}\label{prop:RfreeRSn}
  $R$ is a free left and right $R^{S_n}$-supermodule of graded rank $q^{\frac{n(n-1)}{2}}[n]!$
  with homogeneous basis given by the Schubert polynomials
  $\{\s_\vartheta (\und{x}_n)\}_{\vartheta\in S_n}$.
  The set $\{x_1^{a_1}\dotsm x_n^{a_n}\}_{0\leq a_i\leq n-i}$ is also a basis. 
\end{prop}

\begin{proof}
  We show that the multiplication
\[
R^{S_n} \otimes U_n \to R ,
\]
is an isomorphism of $\bZ$-modules.
In order to do so we need to show that any $f\in R$ can be expressed in the form 
$f=\sum f_i u_i$, with $f_i\in R^{S_n}$ and $u_i\in U_n$. 
Having in mind that the result holds in the polynomial case, we 
only need to prove the case when $f$ has non-zero $\lambda$-degree. 
It is enough to show that any
$p(\und{x}_n)\omega_{\mu}=p(\und{x}_n)\omega_{\mu_1}\omega_{\mu_2}\dotsm \omega_{\mu_k}$ 
can be written as a $\bZ[\und{x}_n]$-linear combination of the
$\cS_{0,\beta}$'s, since then we can apply the result of the polynomial case and write every such
coefficient as a $\bZ[\und{x}]^{S_n}$-linear combination of Schubert polynomials.  
To prove this claim we note that, 
since $\cS_{0,\mu}=\omega_\mu + \sum\limits_{\nu >\mu}c_{\nu}\omega_{\nu}$
with $c_\mu\in\bZ[\und{x}]$ (cf. \eqref{eq:schurdecomp}), we have 
\[
p(\und{x}_n)\omega_{\mu} - p(\und{x}_n)\cS_{0,\mu} = \sum\limits_{\nu >\mu}c'_{\nu}\omega_{\nu},  
\]
with $c'_\mu\in\bZ[\und{x}]$. 
The claim follows by recursive application of this procedure to the
$c_\nu\omega_\nu$'s. 
This shows the multiplication map above is surjective. 
Since both sides are free $\bZ$-modules, equality between their graded ranks proves it is injective.  
\end{proof}

Let $\mat_{q^{n(n-1)/2}[n]!}(R^{S_n})$ be the algebra of matrices of size
$q^{n(n-1)/2}[n]!$ with coefficients in $R^{S_n}$.
Note that we have $\grk_{\bZ}\bigl(\mat_{q^{n(n-1)/2}[n]!}(R^{S_n})\bigr) = ([n]!)^2\grk_{\bZ}(R^{S^n})=\grk_{\bZ}(A_n)$.

\begin{cor}\label{cor:AnisoMat}
The action of $A_n$ on $R$ induces an isomorphism 
\[
\psi\colon A_n \xra{\ \cong\ } \End_{R^{S_n}}(R) \cong \mat_{q^{n(n-1)/2}[n]!}(R^{S_n})  
\]
of bigraded superalgebras. 
\end{cor}

\begin{proof}
  Since $A_n$ acts on $R$ by linearly independent operators and $R$ is free over $R^{S_n}$
  (see Proposition~\ref{prop:RfreeRSn}), the map $\psi$ is injective.
  The surjectivity follows directly from the decomposition $A_n \cong R\otimes\nc_n$ and the definition of the Schubert polynomial (recalling that they are defined by the action of the Demazure operators).
\end{proof}

\subsubsection{Idempotents in $A_n$}\label{ssec:idempots}

Every homogeneous idempotent in $A_n$ is the image of an idempotent in $\nh_n$ under the inclusion map introduced
right after Proposition~\ref{prop:smashAnNH}, 
which is immediate by degree reasons. 
Denote by $x$ the image of $x\in\nh_n$ in $A_n$.
Let $\und{x}_n^\delta=x_1^{n-1}x_2^{n-1}\dotsm x_{n-1}$ and form the
idempotent $e_n= T_{\vartheta_0} \und{x}_n^\delta$ 
in the nilHecke algebra, where $\vartheta_0 \in S_n$ is the longest element.  
This element is an idempotent because
\[
T_{\vartheta_0} \und{x}_n^\delta T_{\vartheta_0} \und{x}_n^\delta =
T_{\vartheta_0}(\und{x}_n^\delta)T_{\vartheta_0} \und{x}_n^\delta +
\vartheta_0(\und{x}_n^\delta)T^2_{\vartheta_0} \und{x}_n^\delta =
T_{\vartheta_0} \und{x}_n^\delta,
\]
since
$T_{\vartheta_0}(\und{x}_n^\delta)=1$ and $T^2_{\vartheta_0}=0$.  

Then  
$A_ne_n$ is a left $A_n$-supermodule isomorphic to $R$, which is, up to isomorphism and grading shifts,
the unique left, bigraded projective indecomposable supermodule (it is not hard to check that
the map sending $1\in R$ to $e_n\in A_ne_n$ is an isomorphism of left $A_n$-supermodules).
We denote by $P_{(n)}$ the left supermodule $A_ne_n$ with the $q$-degree shifted down by $n(n-1)/2$.  
As a consequence of Corollary~\ref{cor:AnisoMat} we have that
as a supermodule over itself, $A_n$ decomposes into $[n]!$ copies of $R$, giving a direct sum
decomposition of bigraded, left, $A_n$-supermodules.

\begin{prop}\label{prop:AnPnDecomp}
As bigraded left-supermodule over itself, $A_n$ decomposes as
\[
A_n\cong \bigoplus_{[n]!}P_{(n)} .
\]
\end{prop}

Here, for a Laurent polynomial $f=\sum f_iq^i\in\bN[q^{\pm 1}]$ and a bigraded supermodule $M$,
$\oplus_fM$ denotes the direct sum $\oplus_i q^i M^{\oplus f_i}$, and where  
$q^r$ is a shift up by $r$ units in the $q$-degree.

There is an involutive anti-automorphism
$\tau$ of $A_n$
fixing the generators of $A_n$, which is homogeneous of degree 0.
We write $u^\tau$ instead of $\tau(u)$ for $u\in A_n$.

Letting $_{(n)}P$ be the right,
bigraded projective indecomposable supermodule $_{(n)}P=e_n^\tau A_n$ with the $q$-degree shifted down by $n(n-1)/2n$
we get a similar decomposition of bigraded right $A_n$-supermodules
\[
A_n\cong \bigoplus_{[n]!}{}_{(n)}P .
\]

The following corollary to Proposition~\ref{prop:RfreeRSn} is immediate. 
\begin{cor}\label{cor:AeA}
  We have $A_ne_nA_n=A_n$. 
\end{cor}

\subsection{Affine cellular structure}

Affine cellular algebras have been introduced in~\cite{Koenig-Xi} as a generalization of
cellular algebras to include some natural infinite-dimensional algebras, as for example algebras
over polynomial rings. 
In particular, the nilHecke algebra is affine cellular, but not cellular~\cite[\S4]{KLM}. 
Graded affine cellularity is a slight extension of affine cellularity to the graded world and was introduced
in~\cite{KLM}. 
Our aim is to (mildly) generalize these notions further. 
We extend this notion to the case of a superalgebra and introduce the notion of
bigraded affine cellular superalgebra. 
We prove below that the case of $A_n$ is bigraded affine cellular by checking that the arguments
given in~\cite[\S4]{KLM} for the case of the nilHecke algebra extend directly to our case.

(Affine) cellular algebras come equipped with a so-called cellular structure.
This additional datum has important consequences for the representation theory of such algebras, e.g. 
a theory of standard modules. 
We hope that bigraded cellularity can be useful in the study of infinite-dimensional superalgebras. 
As in the case of $\nh_n$ whose graded affine cellularity is established without too much effort, and is then used to prove graded affine cellularity of KLR algebras of finite types, we hope that graded affine cellularity of $A_n$
hints at a similar feature of the extended KLR algebras of~\cite{naissevaz3}. 
Our notion of cellularity works in the finite-dimensional case as well,
similar as affine cellularity generalizes cellularity.

We give below the main definitions involved, adapted from~\cite[\S4]{KLM} to the bigraded case. 
We say that an \emph{affine superalgebra} is a quotient of a supercommutative algebra
(i.e.  a polynomial algebra tensored with an exterior algebra). 
\begin{defn}
  Let $C$ be a graded, unital, $\Bbbk$-superalgebra equipped with a $\Bbbk$-anti-involution $\tau$, where $\Bbbk$ is a noetherian domain.
  A two-sided ideal $J$ in $C$ is called an \emph{affine cell ideal} if the following conditions are satisfied:
  \begin{enumerate}
  \item $\tau(J)=J$,
  \item there exists an affine $\Bbbk$-superalgebra $B$ with a parity preserving $\Bbbk$-involution $\sigma$ and a free
    $\Bbbk$-supermodule $V$ of finite rank such that $\Delta:=V \otimes_{\Bbbk} B$ has a $(C,B)$-superbimodule structure, 
    with the right $B$-supermodule structure induced by the regular right $B$-supermodule structure on $B$, 
  \item let $\Delta' := B \otimes_\Bbbk V$ be the $(B,C)$-superbimodule with left $B$-supermodule structure 
    induced by the regular left $B$-supermodule structure on $B$ and right $C$-supermodule structure defined by 
\begin{equation*}
(b\otimes v)c = \rm(s)(\tau(c)(v \otimes b)),
\end{equation*} 
where 
$\rm{s}:V\otimes_\Bbbk B\to B\otimes_\Bbbk V,\ v\otimes b\to b\otimes v$; 
then there is a $(C,C)$-superbimodule isomorphism $\alpha: J \to \Delta \otimes_B\Delta'$,
such that the following diagram commutes:
$$\xymatrix{ J \ar^-{\alpha}[rr]  \ar^{\tau}[d]&& \Delta \otimes_B\Delta' \ar^{v \otimes b \otimes b' \otimes w \mapsto w \otimes \sigma(b') \otimes \sigma(b) \otimes v }[d] \\ J \ar^-{\alpha}[rr]&&\Delta \otimes_B\Delta'.}$$
  \end{enumerate}

\end{defn}
\begin{defn}
  The superalgebra $C$ is called \emph{bigraded affine cellular} if there is a $\Bbbk$-supermodule decomposition
  \[
  C= J_1' \oplus J_2' \oplus \cdots \oplus J_n'\] 
with $\tau(J_\ell')=J_\ell'$ for $1 \leq \ell \leq n$, such that, setting $J_m:= \bigoplus_{\ell=1}^m J_\ell'$, we obtain an ideal filtration
$$0=J_0 \subset J_1 \subset J_2 \subset \cdots \subset J_n=C$$
so that each $J_m/J_{m-1}$ is an affine cell ideal of $C/J_{m-1}$. 
\end{defn}

The following establishes that the superalgebra $A_n$ is bigraded affine cellular.
Let $\{z_\gamma\}_{\gamma \in\Gamma}$ be any $\bZ$-basis of $R^{S_n}$. 
Recall $\und{x}_n^\delta$ and $\tau$ from~\S\ref{ssec:idempots}. 
\begin{prop}
The superalgebra $A_n$ has an affine cellular basis given by one cell $J$ generated by the set
\[
\{ T_u z_\gamma \und{x}_n^\delta e_n T_v^\tau \vert\ u,v\in S_n, \gamma\in\Gamma \} .
\]  
\end{prop}
\begin{proof}
This is the same proof as ~\cite[Theorem 4.8]{KLM}. 
  Let $J=\spn_\bZ\{ T_u z_\gamma \und{x}_n^\delta e_n T_v^\tau \vert\ u,v\in S_n, \gamma\in\Gamma \}$.
  By comparing (bi)graded dimensions one checks that $J\cong A_n$.
  The results in~\S\ref{ssec:Sninvariants} and~\S\ref{ssec:supermodule} (in particular, 
  Corollaries~\ref{cor:AnisoMat} and~\ref{cor:AeA})
  above imply that $J$ is in fact, a bigraded affine cell ideal. 
\end{proof}

%
%

\subsection{Diagrammatic presentation of the superalgebra $A_n$} \label{sec:diag}

Next, we aim to give the superalgebra $A_n$ a diagrammatic presentation
very much in the spirit of~\cite{KL1}. 

\begin{rem}
A graphical calculus for the superalgebra $A_n$ was already introduced by the authors in~\cite[\S 9.1]{naissevaz1}.
However, the one we give here is slightly different:
white dots from~\cite[\S 9.1]{naissevaz1}  
are now placed in the regions of the diagrams rather than on the strands, and can be labeled by non-negative integers.
In this presentation we call them \emph{floating dots}.
\end{rem}

\smallskip

We use the usual diagrams for the generators of $\nh_n$: 
        \[
        \tikz[very thick,xscale=1.5,baseline={([yshift=.8ex]current bounding box.center)}]{
          \draw (-.5,-.5) node[below] {\small $1$} -- (-.5,.5); 
          \draw (0,-.5) node[below] {\small $2$} -- (0,.5); 
          \draw (1.5,-.5) node[below] {\small  $\phantom{1} n  \phantom{1}$} -- (1.5,.5); 
          \node at (.75,0){$\cdots$};
        }  =\ 1 \in A_n , 
\mspace{60mu}
        \tikz[very thick,xscale=1.5,baseline={([yshift=.8ex]current bounding box.center)}]{
          \draw (-.5,-.5) node[below] {\small $1$}  -- (-.5,.5); 
          \draw (.5,-.5) node[below] {\small $i{-}1$}  -- (.5,.5); 
          \draw (1,-.5) node[below] {\small $i$}  -- (1,.5)
          node [midway,fill=black,circle,inner sep=2pt]{}; ;
          \draw (1.5,-.5) node[below] {\small $i{+}1$}  -- (1.5,.5); 
          \draw (2.5,-.5) node[below] {\small  $\phantom{1} n  \phantom{1}$}  -- (2.5,.5); 
          \node at (0,0){$\cdots$};
          \node at (2,0){$\cdots$};
        } =\ x_i \in A_n ,
        \]
\[
       \tikz[very thick,xscale=1.5,baseline={([yshift=.8ex]current bounding box.center)}]{
          \draw (-1,-.5) node[below] {\small $1$} -- (-1,.5);
          \node at (-.5,0){$\cdots$};
          \draw (0,-.5) node[below] {\small $i{-}1$} -- (0,.5);
	\draw  +(0.5,-.5)node[below] {\small $i$}  .. controls (0.5,0) and (1, 0) ..  +(1,0.5);
	 \draw  +(1,-.5) node[below] {\small $i{+}1$}  .. controls (1,0) and (0.5, 0) ..  +(0.5,0.5);
          \draw (1.5,-.5) node[below] {\small $i{+}2$} -- (1.5,.5);
          \draw (2.5,-.5) node[below] {\small  $\phantom{1} n  \phantom{1}$ } -- (2.5,.5);
          \node at (2,0){$\cdots$};
       }
       =\ T_i\in A_n .
\]
and we picture the $\omega_i$'s as \emph{floating dots}:
\[
        \tikz[very thick,xscale=1.5,baseline={([yshift=.8ex]current bounding box.center)}]{
          \draw (-.5,-.5) node[below] {\small $1$}  -- (-.5,.5); 
          \draw (.5,-.5) node[below] {\small $i$}  -- (.5,.5) ;
	   \fdot[a]{.7,0};
          \draw (1,-.5) node[below] {\small $i{+}1$}  -- (1,.5) ;
          \draw (2,-.5) node[below] {\small  $\phantom{1} n  \phantom{1}$}  -- (2,.5);
          \node at (1.5,0){$\cdots$};
          \node at (0,0){$\cdots$};
        } =\ \omega_i^a \in A_n .
      \]
      Here $a \in \bN_0$ is a non-negative integer and we write $\omega_i^0 = \omega_i$ as a floating dot without
      label, by convention. 

\smallskip

Multiplication is given by 
stacking diagrams on top of each other where,
in our conventions, $ab$ means stacking the diagram for $a$ atop the one for $b$ (thus, we read diagrams from bottom to top).

\smallskip 

Relations in $A_n$ acquire a graphical interpretation, as shown in~\eqref{eq:bdots} to~\eqref{eq:omegashift} below, together with the usual height move for distant crossings and dots. Relations~\eqref{eq:crossings}-\eqref{eq:omegashift} have to be understood as being local. By this we mean that they are embedded is some bigger diagrams where everything matches except a small disk which is represented in the pictures below. Relation~\eqref{eq:bdots} means floating dots anticommute,
and in particular a diagram containing two floating dots with the same label in the same region is zero.
\begin{align}
	\tikz[very thick,baseline={([yshift=-.5ex]current bounding box.center)}]{
		\fdot[a]{-.5,-0.5};
		\fdot[b]{1.5,.5};
      	 	   \node at (.5,0){$\cdots$};
  	}
	&\quad=\quad -\quad
	\tikz[very thick,baseline={([yshift=-.5ex]current bounding box.center)}]{
      	 	   \node at (.5,0){$\cdots$};
		\fdot[a]{-.5,0.5};
		\fdot[b]{1.5,-.5};
  	}
	\label{eq:bdots}
\end{align}
\begin{align}
	\tikz[very thick,xscale=2,yscale=1.5,baseline={([yshift=-.5ex]current bounding box.center)}]{
		\draw  +(0,0) .. controls (0,0.25) and (0.5, 0.25) ..  +(0.5,0.5);
		\draw  +(0.5,0) .. controls (0.5,0.25) and (0, 0.25) ..  +(0,0.5);
		\draw  +(0,0.5) .. controls (0,0.75) and (0.5, 0.75) ..  +(0.5,1);
		\draw  +(0.5,0.5) .. controls (0.5,0.75) and (0, 0.75) ..  +(0,1);
	 }
	&\quad=\quad 0  
	&
	\tikz[very thick,xscale=2,yscale=2.15,baseline={([yshift=-.5ex]current bounding box.center)}]{
		\draw  +(0,0) .. controls (0,0.175) and (1, 0.35) ..  +(1,0.7);
		\draw  +(0.5,0) .. controls (0.5,0.175) and (0, 0.175) ..  +(0,0.35);
		\draw  +(1,0) .. controls (1,0.35) and (0, 0.525) ..  +(0,0.7);
		\draw  +(0,0.35) .. controls (0,0.525) and (0.5, 0.525) ..  +(0.5,0.7);
	 }
	&\quad=\quad 
	 \tikz[very thick,xscale=-2,yscale=2.15,baseline={([yshift=-.5ex]current bounding box.center)}]{
		\draw  +(0,0) .. controls (0,0.175) and (1, 0.35) ..  +(1,0.7);
		\draw  +(0.5,0) .. controls (0.5,0.175) and (0, 0.175) ..  +(0,0.35);
		\draw  +(1,0) .. controls (1,0.35) and (0, 0.525) ..  +(0,0.7);
		\draw  +(0,0.35) .. controls (0,0.525) and (0.5, 0.525) ..  +(0.5,0.7);
	}
	\label{eq:crossings}
\end{align}
\begin{align}
	\tikz[very thick,xscale=2,yscale=2,baseline={([yshift=-.5ex]current bounding box.center)}]{
		\draw  +(0,0) .. controls (0,0.25) and (0.5, 0.25) ..  +(0.5,0.5);
		\draw  +(0.5,0) .. controls (0.5,0.25) and (0, 0.25) ..  +(0,0.5) node [near end,fill=black,circle,inner sep=2pt]{};
	}
	&\quad=\quad
	\tikz[very thick,xscale=-2,yscale=-2,baseline={([yshift=-.5ex]current bounding box.center)}]{
		\draw  +(0,0) .. controls (0,0.25) and (0.5, 0.25) ..  +(0.5,0.5);
		\draw  +(0.5,0) .. controls (0.5,0.25) and (0, 0.25) ..  +(0,0.5) node [near end,fill=black,circle,inner sep=2pt]{};
	}
	\quad + \quad
	\tikz[very thick,scale=2,baseline={([yshift=-.5ex]current bounding box.center)}]{
	          \draw (0,0)-- (0,.5);
	          \draw (.5,0)-- (.5,.5);
	}  \label{eq:relnh1}
	\\[1ex]
	\tikz[very thick,xscale=2,yscale=-2,baseline={([yshift=-.5ex]current bounding box.center)}]{
		\draw  +(0,0) .. controls (0,0.25) and (0.5, 0.25) ..  +(0.5,0.5);
		\draw  +(0.5,0) .. controls (0.5,0.25) and (0, 0.25) ..  +(0,0.5) node [near end,fill=black,circle,inner sep=2pt]{};
	} 
	&\quad=\quad
	\tikz[very thick,xscale=-2,yscale=2,baseline={([yshift=-.5ex]current bounding box.center)}]{
		\draw  +(0,0) .. controls (0,0.25) and (0.5, 0.25) ..  +(0.5,0.5);
		\draw  +(0.5,0) .. controls (0.5,0.25) and (0, 0.25) ..  +(0,0.5) node [near end,fill=black,circle,inner sep=2pt]{};
	}
	\quad + \quad
	\tikz[very thick,scale=2,baseline={([yshift=-.5ex]current bounding box.center)}]{
	          \draw (0,0)-- (0,.5);
	          \draw (.5,0)-- (.5,.5);
	}  \label{eq:relnh2}
\end{align}
\begin{align}
	\tikz[very thick,xscale=2,yscale=3,baseline={([yshift=-.5ex]current bounding box.center)}]{
		\draw  +(0,0) .. controls (0,0.25) and (0.5, 0.25) ..  +(0.5,0.5);
		\draw  +(0.5,0) .. controls (0.5,0.25) and (0, 0.25) ..  +(0,0.5);
		  \fdot[a]{.2,.075};
	}
	\quad - \quad
	\tikz[very thick,xscale=2,yscale=-3,baseline={([yshift=-.5ex]current bounding box.center)}]{
		\draw  +(0,0) .. controls (0,0.25) and (0.5, 0.25) ..  +(0.5,0.5) node [near end,fill=black,circle,inner sep=2pt]{};
		\draw  +(0.5,0) .. controls (0.5,0.25) and (0, 0.25) ..  +(0,0.5);
		  \fdot[a]{.45,.25};
	}
	&\quad = \quad
	\tikz[very thick,xscale=2,yscale=-3,baseline={([yshift=-.5ex]current bounding box.center)}]{
		\draw  +(0,0) .. controls (0,0.25) and (0.5, 0.25) ..  +(0.5,0.5);
		\draw  +(0.5,0) .. controls (0.5,0.25) and (0, 0.25) ..  +(0,0.5);
		  \fdot[a]{.25,.125};
	}
	\quad - \quad
	\tikz[very thick,xscale=2,yscale=-3,baseline={([yshift=-.5ex]current bounding box.center)}]{
		\draw  +(0,0) .. controls (0,0.25) and (0.5, 0.25) ..  +(0.5,0.5);
		\draw  +(0.5,0) .. controls (0.5,0.25) and (0, 0.25) ..  +(0,0.5) node [near start,fill=black,circle,inner sep=2pt]{};
		  \fdot[a]{.45,.25};
	} \label{eq:relomega2} 
\end{align}
\begin{align}\label{eq:omegashift}
\tikz[very thick,baseline={([yshift=-.5ex]current bounding box.center)}]{
	          \draw (0,-.5)-- (0,.5);
		  \fdot[a]{.5,0};
	} 
	&\quad = \quad
	\tikz[very thick,baseline={([yshift=-.5ex]current bounding box.center)}]{
	          \draw (0,-.5)-- (0,.5);
		\fdot[a{-}1]{-1,0};
	}
	\quad - \quad
	\tikz[very thick,baseline={([yshift=-.5ex]current bounding box.center)}]{
	          \draw (0,-.5)-- (0,.5)node [midway,fill=black,circle,inner sep=2pt]{}; 
		 \fdot[a{-}1]{.5,0};
	}  
\end{align}
By Propositions~\ref{prop:smashAnNH} and~\ref{prop:relomegaa}, relations~\eqref{eq:bdots} to~\eqref{eq:omegashift} form a complete set of relations for $A_n$.
As explained in~\S\ref{ssec:algan} (and~\S\ref{ssec:labeledomegas} for the labeled floating dots), the superalgebra $A_n$ is bigraded with degrees given by
\begin{align*}
\deg(x_i) &= (2,0), &   \deg(T_i) &= (-2,0), &   \deg(\omega_i^a) &= (2(a-i),2).
\end{align*}

The symmetry $\tau$ from \S\ref{ssec:idempots} consists in reflecting a
diagram around the horizontal axis.
Note however that a reflection around a 
vertical axis is not homogeneous with respect to the $q$-grading: the $q$-degree of a floating dot depends on the number of strands at its left, and the reflection can change this number.

\begin{rem}
In fact, one can give a topological definition for those diagrams, where they are actually taken up to isotopy  which does not create critical points and preserves  the relative height of floating dots. This is possible because of the relations defining $A_n$, which allow the permutations of distant crossings, dots and floating dots together (except the commutation between floating dots).
\end{rem}

\begin{rem}
In the definition of the $\omega_i^a$'s from~\S\ref{ssec:labeledomegas}, we put $\omega_0^a = 0$. This translates into the diagrammatic framework to kill all diagrams with floating dots in the leftmost region:
\begin{align*}\allowdisplaybreaks
        \tikz[very thick,xscale=1.5,baseline={([yshift=.8ex]current bounding box.center)}]{
          \draw (-.5,-.5) node[below] {\small $1$} -- (-.5,.5); 
          \draw (0,-.5) node[below] {\small $2$} -- (0,.5); 
          \draw (1.5,-.5) node[below] {\small $\phantom{1}n\phantom{1}$} -- (1.5,.5); 
          \node at (.75,0){$\cdots$};
	 \fdot[a]{-1,0};
        }\quad &= \quad 0.
\end{align*}
\end{rem}

The equations presented in the beginning of~\S\ref{ssec:tightmonom} can be generalized for labeled $\omega_i$'s and turned into diagrams to give the following consequence, which will be used in the sequel:
\begin{lem}\label{lem:wdtranslation}
In $A_n$ we have the local relation
\begin{align}\allowdisplaybreaks
	\tikz[very thick,baseline={([yshift=-.5ex]current bounding box.center)}]{
	          \draw (0,-.75)-- (0,.75);
	          \draw (1,-.75)-- (1,.75);
	  	 \fdot[a]{1.35,0};
	} 
	&\quad = \quad
	\tikz[very thick,xscale=2,yscale=1.5,baseline={([yshift=-.5ex]current bounding box.center)}]{
		\draw  +(0,0) .. controls (0,0.25) and (0.5, 0.25) ..  +(0.5,0.5);
		\draw  +(0.5,0) .. controls (0.5,0.25) and (0, 0.25) ..  +(0,0.5);
		\draw  +(0,0.5) .. controls (0,0.75) and (0.5, 0.75) ..  +(0.5,1) node [near end,fill=black,circle,inner sep=2pt]{};
		\draw  +(0.5,0.5) .. controls (0.5,0.75) and (0, 0.75) ..  +(0,1);
	  	 \fdot[a]{.2,.45};
	 } 
	\quad-\quad
	\tikz[very thick,xscale=2,yscale=1.5,baseline={([yshift=-.5ex]current bounding box.center)}]{
		\draw  +(0,0) .. controls (0,0.25) and (0.5, 0.25) ..  +(0.5,0.5)node [near start,fill=black,circle,inner sep=2pt]{};
		\draw  +(0.5,0) .. controls (0.5,0.25) and (0, 0.25) ..  +(0,0.5);
		\draw  +(0,0.5) .. controls (0,0.75) and (0.5, 0.75) ..  +(0.5,1);
		\draw  +(0.5,0.5) .. controls (0.5,0.75) and (0, 0.75) ..  +(0,1);
	  	 \fdot[a]{.2,.45};
	 } 
\end{align}
\end{lem}

\begin{proof}
First we compute
\begin{align*}
\tikz[very thick,xscale=2,yscale=1.5,baseline={([yshift=-.5ex]current bounding box.center)}]{
		\draw  +(0,0) .. controls (0,0.25) and (0.5, 0.25) ..  +(0.5,0.5);
		\draw  +(0.5,0) .. controls (0.5,0.25) and (0, 0.25) ..  +(0,0.5);
		\draw  +(0,0.5) .. controls (0,0.75) and (0.5, 0.75) ..  +(0.5,1);
		\draw  +(0.5,0.5) .. controls (0.5,0.75) and (0, 0.75) ..  +(0,1);
	  	 \fdot[a]{.2,.45};
	 } 
\quad &\overset{\mspace{15mu} \eqref{eq:relomega2}\mspace{15mu}  }{=} \quad
	\tikz[very thick,xscale=2,yscale=1.5,baseline={([yshift=-.5ex]current bounding box.center)}]{
		\draw  +(0,0) .. controls (0,0.25) and (0.5, 0.25) ..  +(0.5,0.5);
		\draw  +(0.5,0) .. controls (0.5,0.25) and (0, 0.25) ..  +(0,0.5);
		\draw  +(0,0.5) .. controls (0,0.75) and (0.5, 0.75) ..  +(0.5,1);
		\draw  +(0.5,0.5) .. controls (0.5,0.75) and (0, 0.75) ..  +(0,1)node [pos=0,fill=black,circle,inner sep=2pt]{};
	  	 \fdot[a]{.55,.75};
	 } 
\quad + \quad
	\tikz[very thick,xscale=2,yscale=1.5,baseline={([yshift=-1ex]current bounding box.center)}]{
		\draw  +(0,0) .. controls (0,0.25) and (0.5, 0.25) ..  +(0.5,0.5);
		\draw  +(0.5,0) .. controls (0.5,0.25) and (0, 0.25) ..  +(0,0.5);
		\draw  +(0,0.5) .. controls (0,0.75) and (0.5, 0.75) ..  +(0.5,1);
		\draw  +(0.5,0.5) .. controls (0.5,0.75) and (0, 0.75) ..  +(0,1);
	  	 \fdot[a]{.25,.9};
	 } 
\quad - \quad
	\tikz[very thick,xscale=2,yscale=1.5,baseline={([yshift=-.5ex]current bounding box.center)}]{
		\draw  +(0,0) .. controls (0,0.25) and (0.5, 0.25) ..  +(0.5,0.5);
		\draw  +(0.5,0) .. controls (0.5,0.25) and (0, 0.25) ..  +(0,0.5);
		\draw  +(0,0.5) .. controls (0,0.75) and (0.5, 0.75) ..  +(0.5,1) node [near end,fill=black,circle,inner sep=2pt]{};
		\draw  +(0.5,0.5) .. controls (0.5,0.75) and (0, 0.75) ..  +(0,1);
	  	 \fdot[a]{.55,.75};
	 }   \\
\quad &\overset{\eqref{eq:relnh1}, \eqref{eq:crossings} }{=}
 - \mspace{2mu}
	\tikz[very thick,xscale=2,yscale=2,baseline={([yshift=-.5ex]current bounding box.center)}]{
		\draw  +(0,0) .. controls (0,0.25) and (0.5, 0.25) ..  +(0.5,0.5);
		\draw  +(0.5,0) .. controls (0.5,0.25) and (0, 0.25) ..  +(0,0.5);
		  \fdot[a]{.45,.25};
	}
\end{align*}
Then we observe
\begin{align*}
\tikz[very thick,baseline={([yshift=-.5ex]current bounding box.center)}]{
	          \draw (0,-.5)-- (0,.5);
	          \draw (1,-.5)-- (1,.5);
	  	  \fdot[a]{1.35,0};
	}
\quad \overset{\eqref{eq:relnh2}}{=} \quad
\tikz[very thick,xscale=2,yscale=2,baseline={([yshift=-.5ex]current bounding box.center)}]{
		\draw  +(0,0) .. controls (0,0.25) and (0.5, 0.25) ..  +(0.5,0.5)node [near start,fill=black,circle,inner sep=2pt]{};
		\draw  +(0.5,0) .. controls (0.5,0.25) and (0, 0.25) ..  +(0,0.5);
		  \fdot[a]{.55,.25};
	}
\quad - \quad 
\tikz[very thick,xscale=2,yscale=2,baseline={([yshift=-.5ex]current bounding box.center)}]{
		\draw  +(0,0) .. controls (0,0.25) and (0.5, 0.25) ..  +(0.5,0.5) node [near end,fill=black,circle,inner sep=2pt]{};
		\draw  +(0.5,0) .. controls (0.5,0.25) and (0, 0.25) ..  +(0,0.5);
		  \fdot[a]{.55,.25};
	}
\end{align*}
and we conclude the proof by combining these two equalities.
\end{proof}

In view of the discussion in \S\ref{ssec:tightmonom}, we say a floating dot is
\emph{tight} if it is unlabeled and placed directly to the right of the leftmost strand.
The element $\theta_a$ defined in (\ref{eq:thetaa}) translates in the diagrammatic framework as a tightened floating dot, i.e.
\[
	\theta_a \quad = \quad
	\tikz[very thick,xscale=2,yscale=1.15,baseline={([yshift=-.4ex]current bounding box.center)}]{
		\draw (2,0) node[below] {\small $a{+}1$} -- (2,0.7);
		\draw (3,0) node[below] {\small $\phantom{1}n\phantom{1}$} -- (3,0.7);
		\draw  +(1.5,0)  node[below] {\small $\phantom{1}a\phantom{1}$}  .. controls (1.5,0.35) and (0, 0.35) ..  +(0,0.7);
		\draw  +(0,0)  node[below] {\small $1$}  .. controls (0,0.35) and (0.5, 0.35) ..  +(0.5,0.7);
		\draw  +(1,0)  node[below] {\small $a{-}1$}  .. controls (1,0.35) and (1.5, 0.35) ..  +(1.5,0.7);
		\draw (2,0.7)-- (2,1.4);
		\draw (3,0.7)-- (3,1.4);
		\draw  +(0,0.7) .. controls (0,1.05) and (1.5, 1.05) ..  +(1.5,1.4);
		\draw  +(0.5,0.7) .. controls (0.5,1.05) and (0, 1.05) ..  +(0,1.4);
		\draw  +(1.5,0.7) .. controls (1.5,1.05) and (1, 1.05) ..  +(1,1.4);
		\node at(1,.7) {\dots}; \node at(.5,.05) {\dots}; \node at(.5,1.35) {\dots};
		\node at(2.5,.7) {\dots};
		 \fdot{.2,.7}; 
	 }
\]
With this in mind, the basis~(\ref{eq:tightbasis}) has a nice diagrammatic description. Recall it is given by
\begin{equation*}
\left\{
x_1^{k_1} \dotsm x_n^{k_n} T_{ \vartheta^{n+1}}   \theta_{\min_\vartheta(s(n))}^{\ell_{\vartheta(s(n))}}
T_{ \vartheta^{n}} 
\dotsm   \theta_{\min_\vartheta(s(2))}^{\ell_{\vartheta(s(2))}}T_{ \vartheta^{2}}  \theta_{\min_\vartheta(s(1))}^{\ell_{\vartheta(s(1))}}
 T_{\vartheta^1}
 \colon  k_i\in\bN_0, \ell_i\in \{0,1\}, \vartheta\in S_n
\right\},
\end{equation*}
where $\vartheta = \vartheta^n \vartheta^{n-1} \dotsm \vartheta^0$ admits a left-adjusted, reduced expression and
is partitioned such that each $i \in \{1, \dots, n\}$ attains its minimal position $\min_\vartheta(i)$ in $\vartheta^{s^{-1}(i)}\dotsm\vartheta^{1}(i)$.

To explain the diagrammatic presentation of the basis from~\eqref{eq:tightbasis}, we first consider elements of $S_n$ using the usual string diagrams. 
In such a diagram, we can index the strands at the bottom or at the top, counting from the left. We will refer to those indexing by saying a strand is the $i$-th strand at the bottom (resp. at the top) if it starts (resp. ends) at the $i$-th position, counting from the left. For example, in the  following string diagram, the first strand at the bottom is the third at the top and the first at the top is the second at the bottom: 
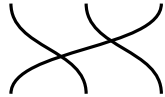
\[
 \tikz[very thick,xscale=-2,yscale=2,baseline={([yshift=0ex]current bounding box.center)}]{
	\draw  +(1,0) .. controls (1,0.3) and (0, 0.3) ..  +(0,0.6);
	\draw  +(0,0) .. controls (0,0.3) and (0.5, 0.3) ..  +(0.5,0.6);
	\draw  +(0.5,0) .. controls (0.5,0.3) and (1, 0.3) ..  +(1,0.6);
 }
\]

Next, choose a left-adjusted, reduced expression $\vartheta$ for each element of $S_n$. This can be thought as choosing any reduced expression and then pulling all strands as much as possible to the left using Reidemeister 3 type moves only (the one on the right in~\eqref{eq:crossings}). Note that the bijection $s: \{1, \dots, n\} \rightarrow \{1, \dots, n\}$ defined in~\S\ref{ssec:tightmonom} tells us in which order (read from bottom to top) each strand attains its leftmost position in the diagram, i.e. the $s(i)$-th strand at the bottom (or $\vartheta(s(i))$-th strand at the top) is the $i$-th strand (from bottom to top) to attains its leftmost position.

\begin{rem}\label{rem:pullstrand}
 In particular, if we consider the left-adjusted reduced expressions constructed through the coset decomposition \eqref{eq:cosetdecomp} of $S_n$, then $s(i)$ is given by looking at the position at the bottom of the $i$-th strand counting from the right at the top (i.e. $(n-i+1)$-th strand counting from the left at the top). Then we have $\vartheta(s(i)) = n-i+1$.
\end{rem}

Clearly, we can view such a string diagram as an element of $A_n$, using the above diagrammatic presentation. We can now put $k_i$ dots at the top of the string diagram  on the $i$-th strand at the top, which still is an element of $A_n$. Last, we put tightened floating dots for $\ell_i  = 1$ into the string diagram by  adding a  
$
\tikz[very thick,baseline={([yshift=-.7ex]current bounding box.center)}]{
		\node  [fill=white, draw=black,circle,inner sep=2pt]  at (0,0){};
	}
$
   symbol to the right of the $i$-th strand at the top, after pulling it to the far left, where it attained already its leftmost position. We do this in order, i.e., we insert these tightened floating dots from bottom to top starting with the smallest $i \in \{ 1, . . . , n \}$ with $\ell_{\vartheta (s(i))}= 1$. In the case of Remark~\ref{rem:pullstrand}, we would do this from $\ell_n$ to $\ell_1$. By construction, we end with a decorated string diagram describing an element of $A_n$. The reader should convince himself/herself that the resulting diagram is actually a basis element from the basis from~\eqref{eq:tightbasis}.

\begin{exe}
Consider the following string diagram obtained by taking a reduced expression of the permutation $(1\ 4\ 3\ 5\ 2) \in S_5:$
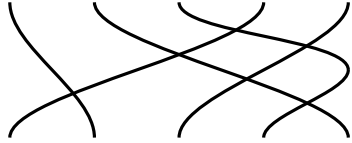
\[
\tikz[very thick,xscale=1.5,yscale=1.2,baseline={([yshift=0ex]current bounding box.center)}]{
        	 \draw +(0,-.75)   .. controls (0,-.25) and (2.25, .25) .. +(2.25,.75);
        	 \draw +(.75,-.75)  ..controls (.75,-.25) and (0,.25) ..  +(0,.75);
        	 \draw +(1.5,-.75)  .. controls (1.5, -.25) and (3,.25) .. +(3,.75);
    	 	\draw  +(2.25,-.75)   .. controls (2.25,-.5) and (3,-.25) ..  +(3,0);
    	 	\draw  +(3, 0)   ..  controls (3,.35) and (1.5,.4) ..  +(1.5,.75);
        	 \draw +(3,-.75) .. controls (3, -.25) and (.75, .25) .. +(.75,.75);
	}
\]
It is not left-adjusted as we can pull  to the left the fourth strand at the bottom by applying two Reidemeister 3 moves.
Doing so we get the following left-adjusted reduced presentation
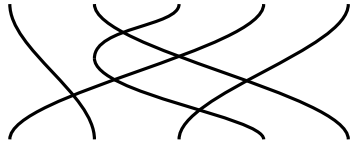
\[
\tikz[very thick,xscale=1.5,yscale=1.2,baseline={([yshift=0ex]current bounding box.center)}]{
        	 \draw +(0,-.75)   .. controls (0,-.25) and (2.25, .25) .. +(2.25,.75);
        	 \draw +(.75,-.75)  ..controls (.75,-.25) and (0,.25) ..  +(0,.75);
        	 \draw +(1.5,-.75)  .. controls (1.5, -.25) and (3,.25) .. +(3,.75);
    	 	\draw  +(2.25,-.75)   .. controls (2.25,-.5) and (.75,-.25) ..  +(.75,.15);
    	 	\draw  +(.75, .15)   ..  controls (.75,.5) and (1.5,.5) ..  +(1.5,.75);
        	 \draw +(3,-.75) .. controls (3, -.25) and (.75, .25) .. +(.75,.75);
	}
\]
Now, if we take for example $k_1 = k_2 = k_3 = 0$, $k_3 = 1$, $k_5 = 2$, $\ell_3 = \ell_4 = 1$ and $\ell_1 = \ell_2=\ell_5 = 0$,  then we get
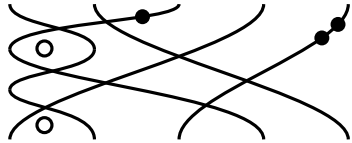
\[
\tikz[very thick,xscale=1.5,yscale=1.2,baseline={([yshift=0ex]current bounding box.center)}]{
        	 \draw +(0,-.75)   .. controls (0,-.25) and (2.25, .25) .. +(2.25,.75);

        	 \draw +(.75,-.75)  ..controls (.75,-.4) and (0,-.4) ..  +(0,-.2);
        	 \draw +(0,-.2)  ..controls (0,0) and (.75,0) ..  +(.75,.25);
        	 \draw +(.75,.25)  ..controls (.75,.5) and (0,.5) ..  +(0,.75);

        	 \draw +(1.5,-.75)  .. controls (1.5, -.25) and (3,.25) .. +(3,.75)  node [pos=.85,fill=black,circle,inner sep=2pt]{} node [pos=.75,fill=black,circle,inner sep=2pt]{};
    	 	\draw  +(2.25,-.75)   .. controls (2.25,-.25) and (0,-.1) ..  +(0,.25);
    	 	\draw  +(0, .25)   ..  controls (0,.55) and (1.5,.55) ..  +(1.5,.75)node [pos=.7,fill=black,circle,inner sep=2pt]{};
        	 \draw +(3,-.75) .. controls (3, -.25) and (.75, .25) .. +(.75,.75);
		\fdot{.3,-.6};  \fdot{.3,.25}; 
	}
\]
where we had to pull the third strand (at the top) to the left in order to put a tight floating dot, the fourth one being already to the left.
\end{exe}

%
%

\subsection{The superalgebra $A_n$ and infinite Grassmannian varieties} \label{sec:geometry}

In this section, we describe the relation between $A_n$ and the geometry of the Grassmannian varieties used in \cite{naissevaz1}. For this section, we consider $A_n$ as a superalgebra over $\bQ$ rather than $\bZ$, and we work with cohomology with rational coefficients.

\smallskip

Recall that the cohomology ring of the Grassmannian $G_{n;N}$ of $n$-planes in $\bC^N$ is generated by the Chern classes  $X_{1,n}, \dots, X_{n,n}$ and $Y_{1,n}, \dots, Y_{N-n,n}$, modded out by the Whitney sum formula.
More explicitly, we have 
\[
H(G_{n;N}) \cong \bQ[X_{1,n}, \dots, X_{n,n}, Y_{1,n}, \dots, Y_{N-n,n} ]/I_{n;N} ,
\]
with $I_{n;N}$ being the ideal generated by all homogeneous terms in the equation
\[
(1 + t X_{1;n} + t^2 X_{2;n} + \dots + t^n X_{n;n})(1+ t Y_{1,n} + \dots + t^{N-n}  Y_{N-n,n}) = 1 .
\]
The same applies for the two-step flag variety
\[
G_{n,n+1;N} = \{ V_n \subset V_{n+1} \subset \bC^N | \dim(V_n) = n, \dim(V_{n+1}) = n+1\} , 
\]
with the Chern classes $X_{1,n}, \dots, X_{n,n}$,  $\xi_{n+1}$ and $Y_{1,n+1}, \dots, Y_{N-n-1,n+1}$, resulting in 
\[
H(G_{n,n+1;N}) \cong \bQ[X_{1,n}, \dots, X_{n,n}, \xi_{n+1}, Y_{1,n+1}, \dots, Y_{N-n-1,n+1} ]/I_{n,n+1;N} , 
\]
where $I_{n,n+1;N}$ is given by
\[
(1 + t X_{1;n} + t^2 X_{2;n} + \dots + t^n X_{n;n})(1+t\xi_{n+1})(1+ t Y_{1,n+1} + \dots + t^{N-n-1} Y_{N-n-1,n+1}) = 1.
\]
The forgetful maps $G_{n,n+1;N} \rightarrow G_{n;N}$ and $G_{n,n+1;N} \rightarrow G_{n+1;N}$ induce an $(H(G_{n;N}), H(G_{n+1;N}))$-bimodule structure on $H(G_{n,n+1;N})$ explicitly given by 
\begin{align*}
X_{i,n+1} &\mapsto X_{i,n} + \xi_{n+1} X_{i-1,n}, &
Y_{i,n} &\mapsto Y_{i,n+1} + \xi_{n+1} X_{i-1,n+1}.
\end{align*}
A nice exposition on the cohomology rings of the finite Grassmannians and iterated flag varieties for the use of categorification is given in~\cite[\S6]{L1}. It is explained there how these rings can be used to construct a categorification of the finite-dimensional representations of quantum $\slt$, using the aforementioned bimodule structure.

\smallskip

The cohomology ring of the infinite Grassmannian $G_{n;\infty}$ of
$n$-planes in $\bC^\infty$ is basically given by the limit
\[
\lim_{N \rightarrow \infty} H(G_{n;N}).
\]
Explicitly, we have 
\[
H(G_{n;\infty}) \cong \bQ[X_{1,n}, \dots, X_{n,n}, Y_{1,n}, \dots ]/I_{n;\infty} ,
\]
with infinitely many $Y_{i,n}$, and where $I_{n;\infty}$ is defined by the homogeneous terms in
\[
(1 + t X_{1;n} + t^2 X_{2;n} + \dots + t^n X_{n;n})(1+ t Y_{1,n} + \dots) = 1 .
\]
Therefore, $H(G_{n;\infty}) \cong \bQ[X_{1,n}, \dots, X_{n,n}]$. 
The same applies for the two-step flag variety
\[
G_{n,n+1;\infty} = \{ V_n \subset V_{n+1} \subset \bC^\infty | \dim(V_n) = n, \dim(V_{n+1}) = n+1\} , 
\]
resulting in 
\[
H(G_{n,n+1;\infty}) \cong \bQ[X_{1,n}, \dots, X_{n,n}, \xi_{n+1}, Y_{1,n+1}, \dots ]/I_{n,n+1;\infty} , 
\]
where $I_{n,n+1;\infty}$ is given by
\[
(1 + t X_{1;n} + t^2 X_{2;n} + \dots + t^n X_{n;n})(1+t\xi_{n+1})(1+ t Y_{1,n+1} + \dots) = 1.
\]
Hence, $H(G_{n,n+1;\infty}) \cong \bQ[X_{1,n}, \dots, X_{n,n}, \xi_{n+1}]$.

\smallskip

Again, there are forgetful maps $G_{n,n+1;\infty} \rightarrow G_{n;\infty}$ and $G_{n,n+1;\infty} \rightarrow G_{n+1;\infty}$, which induce an $(H(G_{n;\infty}), H(G_{n+1;\infty}))$-bimodule structure on $H(G_{n,n+1;\infty})$. It is explicitly described by
\begin{align} \label{eq:grassbimstruct}
X_{i,n+1} &\mapsto X_{i,n} + \xi_{n+1} X_{i-1,n}, &
Y_{i,n} &\mapsto Y_{i,n+1} + \xi_{n+1} X_{i-1,n+1}.
\end{align}
All of this is explained in details in~\cite[\S3]{naissevaz1}.

\smallskip

Then following~\cite[\S3]{naissevaz1}, we define
\begin{align*}
\Omega_n &= H(G_{n;\infty})\otimes \bV^\bullet(s_{1,n}, \dots, s_{n,n}), \\
 \Omega_{n,n+1} &= H(G_{n,n+1;\infty}) \otimes \bV^\bullet(s_{1,n+1}, \dots, s_{n+1,n+1}) ,
\end{align*}
where we think of $\bV^\bullet(s_{1,n}, \dots, s_{n,n})$ as the Koszul dual of $\bQ[X_{1,n}, \dots, X_{n,n}] \cong H(G_{n;\infty})$. By this we mean we can view $\bQ[X_{1,n}, \dots, X_{n,n}]$ as the symmetric algebra of the free vector space generated by $\{X_{1,n}, \dots, X_{n,n}\}$, and then $\bV^\bullet(s_{1,n}, \dots, s_{n,n})$ is the quadratic dual (hence, the Koszul dual) of this symmetric algebra (see~\cite[\S3]{naissevaz1} for details).

We give $\Omega_{n,n+1}$ an $(\Omega_n, \Omega_{n+1})$-bimodule structure by extending the above actions~\eqref{eq:grassbimstruct} via
\begin{align*}
s_{i,n} &\mapsto s_{i,n+1} + \xi_{n+1} s_{i+1, n+1}.
\end{align*}

Recall from Corollary~\ref{cor:leftomegas} that $R^{S_n} \cong \bQ[\und{x}_n]^{S_n}\otimes \bV^\bullet(\omega_n^0, \omega_n^1,\dotsc, \omega_n^{n-1})$. Also recall that $\bQ[\und{x}_n]^{S_n} \cong \bQ[\mathcal{e}_1(\und x_n),\dots, \mathcal{e}_n(\und x_n)]$, so that $R^{S_n}\cong \Omega_n$. Hence $\Omega_n$ is Morita-equivalent to $A_n$ by Corollary~\ref{cor:AnisoMat}. 

\begin{prop}\label{prop:RsnOmega}
There is an isomorphism
\[
R^{S_n} \cong \Omega_n,
\]
given by $\mathcal{e}_{i}(\und x_n) \mapsto X_{i,n}$, $\mathcal{h}_i(\und x_n) \mapsto (-1)^i Y_{i,n}$ and $\omega_n^{n-i} \mapsto s_{i,n}$.
\end{prop}

\begin{proof}
The only thing which requires a proof is $\mathcal{h}_i(\und x_n) \mapsto (-1)^i Y_{i,n}$, which can be checked by verifying the collections of $\mathcal{e}_{i}(\und x_n)$'s and $(-1)^i \mathcal{h}_i(\und x_n)$'s respect an equation of the same kind as the one that defines $I_{n,n+1;\infty}$.
\end{proof}

\begin{rem}
  It is also worthwhile to note that $\omega_k^a$ corresponds exactly to $s_{k-a,k}$ from \cite[\S7.2]{naissevaz1}, even when $a \ge k$. 
  This is related to the formulas used in the reference for the categorification of the
  shifted Verma modules $M(\lambda q^m)$ for $m \ge 0$. 
\end{rem}

It turns out that $H(G_{n,n+1;\infty})$ can also be described in terms of symmetric polynomials
by introducing some extra structure on the nilHecke algebra.
First, let us consider the finite case of $H(G_{n;N})$ and $H(G_{n,n+1;N})$, that can be obtained from $H(G_{n;\infty})$ and $H(G_{n,n+1;\infty})$ by modding out $Y_{> N-n,n} = \{ Y_{i,n} | i > N-n\}$. 
Following a classical approach due to Borel~\cite{borel} (see also~\cite{hiller}), let $C_N = \bQ[z_1, \dots, z_N]/(\mathcal{e}_i(\und z_N))$ be the coinvariant algebra, which comes with an action of the symmetric group $S_N$ by permutation of variables. Write $C_N^n$ and $C_N^{n,n+1}$ for the rings of invariants for the actions  on $C_N$ of $S_n \times S_{N-n} \subset S_N $ and of $S_n \times S_1 \times S_{N-n-1} \subset S_N$, respectively.
Clearly we have
\begin{align*}
C_N^n \cong \bQ[\mathcal{e}_1&(z_1,\dots, z_n), \dots ,\mathcal{e}_n(z_1,\dots, z_n), \\
 & \mathcal{e}_1(z_{n+1},\dots, z_N), \dots, \mathcal{e}_{N-n}(z_{n+1}, \dots, z_N)]/(\mathcal{e}_i(\und z_N)),
\intertext{and}
C_N^{n,n+1}  \cong \bQ[\mathcal{e}_1&(z_1,\dots, z_n), \dots ,\mathcal{e}_n(z_1,\dots, z_n),  \\&z_{n+1}
, \mathcal{e}_1(z_{n+2},\dots, z_N), \dots, \mathcal{e}_{N-n-1}(z_{n+2}, \dots, z_N)]/(\mathcal{e}_i(\und z_N)).
\end{align*}
There is an isomorphism $H(G_{n;N}) \cong C_N^n$ given by
\[
X_{i,n} \mapsto \mathcal{e}_i(z_1, \dots, z_n),\mspace{15mu}Y_{i,n} \mapsto \mathcal{e}_i(z_{n+1}, \dots, z_N),
\]
and an isomorphism $H(G_{n,n+1;N}) \cong C_N^{n,n+1}$ given by
\[
X_{i,n} \mapsto \mathcal{e}_i(z_1, \dots, z_n),\mspace{15mu} \xi_{n+1} \mapsto z_{n+1},\mspace{15mu}
Y_{i,n+1} \mapsto \mathcal{e}_i(z_{n+2}, \dots, z_{N}).
\] 
Indeed, a computation shows that killing $\mathcal{e}_i(\und z_N)$ comes down to killing the same elements as the homogeneous terms of degree $i$ in the equations that define $I_{n;N}$ and $I_{n,n+1;N}$, through the bijection between symmetric polynomials and Chern classes explained above.
Moreover, we recover the $(H(G_{n;N}) , H(G_{+1;N}) )$-bimodule structure of $H(G_{n,n+1;N}) $ in $ C_N^{n,n+1}$ by letting $C_N^n$ and $C_N^{n+1}$ acting by multiplication. This can be checked by verifying the symmetric polynomials respect the following identities
\begin{align*}
\mathcal{e}_i(z_1, \dots, z_{n+1}) &= \mathcal{e}_i(z_1, \dots, z_n) + z_{n+1} \mathcal{e}_{i-1}(z_1, \dots, z_n), \\
\mathcal{e}_i(z_{n+1}, \dots, z_{N}) & = \mathcal{e}_i(z_{n+2}, \dots, z_N)+ z_{n+1} \mathcal{e}_{i-1}(z_{n+2}, \dots, z_N).
\end{align*}

We can write\footnote{This approach is inspired by the lectures given by Webster during the FSMP's Junior Chair~\cite{webster-lectures}.} elements of $C_N^{n,n+1}$ as one strand diagrams on which we can put dots, representing the variable $z_n$, and  ``bubbles'' on both sides representing the actions of $C_N^n$ and $C_N^{n+1}$. For example the equation
\[ 
\mathcal{e}_{i}(z_{n+1}, \dots, z_N) = Y_{i,n}  = Y_{i,n+1} + \xi_{n+1} Y_{i-1,n+1} = \mathcal{e}_i(z_{n+2}, \dots, z_N) + z_{n+1} \mathcal{e}_{i-1}(z_{n+2}, \dots, z_N) ,
\]
becomes
\[
\tikz[very thick,baseline={([yshift=-.5ex]current bounding box.center)}]{
	          \draw (0,-.5)-- (0,.5);
		  \node[draw,circle,minimum size=.75cm,inner sep=1pt] at (-.8,0) { \tiny $Y_i$};
	} 
	\quad = \quad
	\tikz[very thick,baseline={([yshift=-.5ex]current bounding box.center)}]{
	          \draw (0,-.5)-- (0,.5);
		\node[draw,circle,minimum size=.75cm,inner sep=1pt] at (.8,0) {\tiny $Y_i$};
	}
	\quad + \quad
	\tikz[very thick,baseline={([yshift=-.5ex]current bounding box.center)}]{
	          \draw (0,-.5)-- (0,.5)node [midway,fill=black,circle,inner sep=2pt]{}; 
		\node[draw,circle,minimum size=.75cm,inner sep=1pt] at (.8,0) {\tiny $Y_{i-1}$};
	}   
\] 
This generalizes for $n_1 \le \dots \le n_k \le N$, with $H(G_{n_1, \dots, n_k; N}) \cong C_N^{n_1, \dots, n_k}$
through the Young subgroup $S_{n_1} \times S_{n_2-n_1} \times \dots \times S_{N-n_k} \subset S_N$. In addition, we have $C_N^{n,n+1} \otimes_{C_N^{n+1}} C_N^{n+1,n+2} \cong C_N^{n,n+1,n+2}$,
yielding diagrams with $k$ strands for $C_N^{n,n+1, \dots, n+k}$.

For $0 \le i < k$, the Demazure operator $\partial_{n+i}$  acts on $C_N^{n,n+1, \dots, n+k}$ by the divided difference operator
\[
\partial_{n+i}(f) = \frac{f - s_{n+i}(f)}{x_{n+i} - x_{n+i+1}} .
\]
In fact, all bimodule morphims in $\End_{(C_N^n, C_N^{n+k})\operatorname{\mathrm{-}bim}}(C_N^{n,n+1, \dots, n+k})$ are given by multiplications by elements of $C_N^n$ and $C_N^{n+k}$, multiplications by $z_{n}, \dots, z_{n+1}$, and the Demazure operators. 
Therefore we can write $\End_{(C_N^n, C_N^{n+k})\operatorname{\mathrm{-}bim}}(C_N^{n,n+1, \dots, n+k})$ as diagrams with bubbles, dots and crossings respecting the nilHecke relations. This construction motivates the next section.

\subsubsection{Bubbled nilHecke algebras}

\begin{defn}
  The \emph{bubbled nilHecke algebra} is defined as  
  \[
  \bnh_n = \nh_n \rtimes \bigotimes_{p=0}^n \bZ[X_{1,p}, X_{2,p} \dots, X_{i,p}, \dots] \otimes \bZ[Y_{1,p}, Y_{2,p} \dots, Y_{i,p}, \dots] ,
  \]
 with the relations
\begin{align*}
X_{i,p} &= X_{i,p-1} + x_p X_{i-1,p-1}, & Y_{i,p-1} &= Y_{i,p} + x_p Y_{i-1,p}, \\
X_{i,0}u &= u X_{i,0}, & Y_{j,n} u &=u Y_{j,n} ,
\end{align*}
for all $i,j,p \ge 1$,  $u \in \nh_n$ and where $X_{0,p} = Y_{0,p} = 1$.
\end{defn}

The algebra $\bnh_n$ admits a diagrammatic description given by all diagrams
of $\nh_n$ on which we can put labeled bubbles in the regions 
\begin{align*}
        \tikz[very thick,scale=1.2,baseline={([yshift=.8ex]current bounding box.center)}]{
          \draw (-.5,-.5) node[below] {\small $1$}  -- (-.5,.5); 
          \draw (.5,-.5) node[below] {\small $\phantom{1}p\phantom{1}$}  -- (.5,.5) ;
	   \node[draw,circle,minimum size=.75cm,inner sep=1pt] at (1,0) {\tiny $X_i$};
          \draw (1.5,-.5) node[below] {\small $p+1$}  -- (1.5,.5) ;
          \draw (2.5,-.5) node[below] {\small $\phantom{1}n\phantom{1}$}  -- (2.5,.5);
          \node at (2,0){$\cdots$};
          \node at (0,0){$\cdots$};
        } &=\ X_{i,p} \in \bnh_n,
&
        \tikz[very thick,scale=1.2,baseline={([yshift=.8ex]current bounding box.center)}]{
          \draw (-.5,-.5) node[below] {\small $1$}  -- (-.5,.5); 
          \draw (.5,-.5) node[below] {\small $\phantom{1}p\phantom{1}$}  -- (.5,.5) ;
	   \node[draw,circle,minimum size=.75cm,inner sep=1pt] at (1,0) {\tiny $Y_i$};
          \draw (1.5,-.5) node[below] {\small $p+1$}  -- (1.5,.5) ;
          \draw (2.5,-.5) node[below] {\small $\phantom{1}n\phantom{1}$}  -- (2.5,.5);
          \node at (2,0){$\cdots$};
          \node at (0,0){$\cdots$};
        } &=\ Y_{i,p} \in \bnh_n,
 \end{align*}
with the local relations
\begin{align}\label{eq:bubblerelationX}
\tikz[very thick,scale = 1.2,baseline={([yshift=-.5ex]current bounding box.center)}]{
	          \draw (0,-.5)-- (0,.5);
		  \node[draw,circle,minimum size=.75cm,inner sep=1pt] at (.7,0) {\tiny $X_i$};
	} 
	\quad &= \quad
	\tikz[very thick,scale = 1.2,baseline={([yshift=-.5ex]current bounding box.center)}]{
	          \draw (0,-.5)-- (0,.5);
		\node[draw,circle,minimum size=.75cm,inner sep=1pt] at (-.7,0) {\tiny $X_i$};
		\node[draw,white,circle,minimum size=.75cm,inner sep=1pt] at (.7,0) {\tiny $Y_i$};
	}
	\quad + \quad
	\tikz[very thick,scale = 1.2,baseline={([yshift=-.5ex]current bounding box.center)}]{
	          \draw (0,-.5)-- (0,.5)node [midway,fill=black,circle,inner sep=2pt]{}; 
		\node[draw,circle,minimum size=.75cm,inner sep=1pt] at (-.7,0) {\tiny $X_{i-1}$};
		\node[draw,white,circle,minimum size=.75cm,inner sep=1pt] at (.7,0) {\tiny $Y_i$};
	}\\ \label{eq:bubblerelationY}
\tikz[very thick,scale = 1.2,baseline={([yshift=-.5ex]current bounding box.center)}]{
	          \draw (0,-.5)-- (0,.5);
		  \node[draw,circle,minimum size=.75cm,inner sep=1pt] at (-.7,0) {\tiny $Y_i$};
		  \node[draw,white,circle,minimum size=.75cm,inner sep=1pt] at (.7,0) {\tiny $X_i$};
	} 
	\quad   &= \quad
	\tikz[very thick,scale = 1.2,baseline={([yshift=-.5ex]current bounding box.center)}]{
	          \draw (0,-.5)-- (0,.5);
		\node[draw,circle,minimum size=.75cm,inner sep=1pt] at (.7,0) {\tiny $Y_i$};
		\node[draw,white,circle,minimum size=.75cm,inner sep=1pt] at (-.7,0) {\tiny $X_i$};
	}
	\quad + \quad
	\tikz[very thick,scale = 1.2,baseline={([yshift=-.5ex]current bounding box.center)}]{
	          \draw (0,-.5)-- (0,.5)node [midway,fill=black,circle,inner sep=2pt]{}; 
		\node[draw,circle,minimum size=.75cm,inner sep=1pt] at (.7,0) {\tiny $Y_{i-1}$};
		\node[draw,white,circle,minimum size=.75cm,inner sep=1pt] at (-.7,0) {\tiny $X_i$};
	}
\end{align}

\smallskip

From this we can deduce the relations
\begin{align}
	\tikz[very thick,xscale=2,yscale=2,baseline={([yshift=-.5ex]current bounding box.center)}]{
		\draw  +(0,0) .. controls (0,0.25) and (1, 0.75) ..  +(1,1);
		\draw  +(1,0) .. controls (1,0.25) and (0, 0.75) ..  +(0,1);
		\node[draw,circle,minimum size=.75cm,inner sep=1pt] at (.5,.2) {\tiny $Y_i$};
	}
	\quad - \quad
	\tikz[very thick,scale=2,baseline={([yshift=-.5ex]current bounding box.center)}]{
		\draw  +(0,0) .. controls (0,0.25) and (1, 0.75) ..  +(1,1);
		\draw  +(1,0) .. controls (1,0.25) and (0, 0.75) ..  +(0,1) node [near start,fill=black,circle,inner sep=2pt]{};
	  	  \node[draw,circle,minimum size=.75cm,inner sep=1pt] at (.95,.5) {\tiny $Y_{i-1}$};
	  	  \node[draw,circle,minimum size=.75cm,inner sep=1pt,color=white] at (.05,.5) {\tiny $X_{i-1}$}; 
	}
	&\quad = \quad
	\tikz[very thick,scale = 2,baseline={([yshift=-.5ex]current bounding box.center)}]{
	         \draw  +(0,0) .. controls (0,0.25) and (1, 0.75) ..  +(1,1);
		\draw  +(1,0) .. controls (1,0.25) and (0, 0.75) ..  +(0,1);
		\node[draw,circle,minimum size=.75cm,inner sep=1pt] at (.5,.8) {\tiny $Y_i$};
	}
	\quad - \quad
	\tikz[very thick,scale=2,baseline={([yshift=-.5ex]current bounding box.center)}]{
		\draw  +(0,0) .. controls (0,0.25) and (1, 0.75) ..  +(1,1)node [near end,fill=black,circle,inner sep=2pt]{};
		\draw  +(1,0) .. controls (1,0.25) and (0, 0.75) ..  +(0,1);
	  	  \node[draw,circle,minimum size=.75cm,inner sep=1pt] at (.95,.5) {\tiny $Y_{i-1}$};
	  	  \node[draw,circle,minimum size=.75cm,inner sep=1pt,color=white] at (.05,.5) {\tiny $X_{i-1}$}; 
	}  \label{eq:crossbubbleY}
 \\ 
	\tikz[very thick,scale = 2,baseline={([yshift=-.5ex]current bounding box.center)}]{
	         \draw  +(0,0) .. controls (0,0.25) and (1, 0.75) ..  +(1,1);
		\draw  +(1,0) .. controls (1,0.25) and (0, 0.75) ..  +(0,1);
		\node[draw,circle,minimum size=.75cm,inner sep=1pt] at (.5,.2) {\tiny $X_i$};
	}
	\quad - \quad
	\tikz[very thick,scale=2,baseline={([yshift=-.5ex]current bounding box.center)}]{
	         \draw  +(0,0) .. controls (0,0.25) and (1, 0.75) ..  +(1,1) node [near start,fill=black,circle,inner sep=2pt]{};
		\draw  +(1,0) .. controls (1,0.25) and (0, 0.75) ..  +(0,1);
	  	  \node[draw,circle,minimum size=.75cm,inner sep=1pt] at (.05,.5) {\tiny $X_{i-1}$};
	  	  \node[draw,circle,minimum size=.75cm,inner sep=1pt,color=white] at (.95,.5) {\tiny $Y_{i-1}$}; 
	}
	&\quad = \quad
	\tikz[very thick,scale = 2,baseline={([yshift=-.5ex]current bounding box.center)}]{
	         \draw  +(0,0) .. controls (0,0.25) and (1, 0.75) ..  +(1,1);
		\draw  +(1,0) .. controls (1,0.25) and (0, 0.75) ..  +(0,1);
		\node[draw,circle,minimum size=.75cm,inner sep=1pt] at (.5,.8) {\tiny $X_i$};
	}
	\quad - \quad
	\tikz[very thick,scale=2,baseline={([yshift=-.5ex]current bounding box.center)}]{
	         \draw  +(0,0) .. controls (0,0.25) and (1, 0.75) ..  +(1,1);
		\draw  +(1,0) .. controls (1,0.25) and (0, 0.75) ..  +(0,1)node [near end,fill=black,circle,inner sep=2pt]{};
	  	  \node[draw,circle,minimum size=.75cm,inner sep=1pt] at (0.05,.5) {\tiny $X_{i-1}$};
	  	  \node[draw,circle,minimum size=.75cm,inner sep=1pt,color=white] at (.95,.5) {\tiny $Y_{i-1}$}; 
	}  \label{eq:crossbubbleX} 
\end{align}
and bubbles float freely in the regions delimited by the strands, commuting with one another (in opposition to the floating dots of $A_n$).

\begin{defn}
For $M,N \in \bN_0 \cup \{\infty\}$, the \emph{$(M,N)$-bubbled nilHecke algebra}, $(M,N)\text{-}\bnh_n$, is defined as the quotient of $\bnh_n$ by the ideal generated by $X_{>M,0} = \{X_{i,0} | i > M\}$ and $Y_{n,> N-n} = \{Y_{n,i} | i >  N -n\}$.
\end{defn}
This means that we kill all bubbles in the leftmost region with $X$-label greater than $M$ (or none if $M = \infty$) and  all bubbles in the rightmost region with $Y$-label greater than $N-n$.

\smallskip

Let $I_{\infty,\infty}$ denote the ideal in $\bnh_n$ obtained by the homogeneous terms in the equation
\begin{equation}\label{eq:Iinftyinfty}
(1 + t X_{1,p} + t^2 X_{2,p} + \dots )(1 + tY_{1,p} + t^2 Y_{2,p} + \dots) = 1, 
\end{equation}
for all $p \ge 0$. 

\begin{prop}
There is an isomorphism $(0,\infty)\text{-}\bnh_n/I_{\infty,\infty} \cong \nh_n$ induced by the inclusion $\nh_n \subset \bnh_n$. Moreover, it sends 
\begin{align*}
X_{i,p} &\mapsto \mathcal{e}_i(x_1, \dots, x_p), &  Y_{i,p} \mapsto (-1)^i \mathcal \mathcal{h}_i(x_1, \dots, x_p).
\end{align*}
\end{prop}

\begin{proof}
 Using~(\ref{eq:bubblerelationX}) and (\ref{eq:bubblerelationY})  repetitively in $\bnh_n$, we can bring all bubbles to the left, up to adding dots. In particular one can show by induction that
\begin{align*}
X_{i,p} &= \mathcal{e}_i(x_1, \dots, x_p) X_{0,0} + \sum_{r = 1}^{i} \mathcal{e}_{i-r}(x_1, \dots, x_p) X_{r,0}, \\
Y_{i,p} &= (-1)^i \mathcal{h}_i(x_1, \dots, x_p) Y_{0,0} + \sum_{r = 1}^{i} (-1)^{i-r} \mathcal{h}_{i-r}(x_1, \dots, x_p) Y_{r,0}.
\end{align*}
By killing $I_{\infty,\infty}$ together with all $X_{i,0}$ for $i > 0$, we kill all $Y_{i,0}$ with $i > 0$. This observation together with the computations above conclude the proof.
\end{proof}

We recall that the cyclotomic nilHecke algebra $ \nh_n^N$ is the quotient of $\nh_n$ by $(x_1^N)$.

\begin{prop}\label{prop:bnhinftyiso}
There is an isomorphism $(0,N)\text{-}\bnh_n/I_{\infty,\infty} \cong \nh_n^N$ induced by the inclusion $\nh_n \subset \bnh_n$.
\end{prop}

\begin{proof}
First we observe that in $(0,\infty)\text{-}\bnh_n/I_{\infty,\infty}$ we have
\[
(-1)^N x_1^N = Y_{N,1}  = \sum_{r=0}^{n-1}  \mathcal{e}_{r}(x_2, \dots, x_n) Y_{N-r,n},
\]
and also we have $Y_{N-n+1+i,n} = \mathcal{h}_{N-n+1+i}(x_1, \dots, x_n) \in (0,\infty)\text{-}\bnh_n/I_{\infty,\infty}$, which is killed in $\nh_n^N$ \cite[Proposition~2.8]{hoffnunglauda}.
\end{proof}

Note that similarly $H(G_{n,n+1;\infty})$ is isomorphic to $(n,\infty)\text{-}\bnh_1/I_{\infty,\infty}$ and $H(G_{n,n+1;N}) \cong (n,N-n)\text{-}\bnh_1/I_{\infty,\infty}$, if we extend the ground ring to $\bQ$.

\subsubsection{Bubbled $A_n$}\label{ssec:bubbledAn}

We now define a bubbled version of $A_n$, in the same spirit as above.

\begin{defn}
  The \emph{bubbled superalgebra} $\BA_n$ is defined as
  \[
  \BA_n = \bnh_n \rtimes
  \biggl( \bigwedge_{p=0}^n
  \Bigl( \bV^\bullet(\omega_{p}^0, \omega_{p}^1 \dots, \omega_{p}^i, \dots) \wedge  \bV^\bullet(\varpi_{p}^0, \varpi_{p}^1 \dots, \varpi_{p}^i, \dots)
  \Bigr)
  \biggr) ,
  \]
with the relations
\begin{align*}
	 X_{b,i} \omega_j^a &=\omega_j^a   X_{b,i}, & Y_{b,i} \omega_j^a &=\omega_j^a   Y_{b,i},   \\
	x_i\omega_j^a &= \omega_j^ax_i,   & \omega_i^a &= \omega_{i+1}^{a+1} + x_{i+1} \omega_{i+1}^a, \\
   T_i\omega_j^a &= \omega_j^aT_i, \text{\ \ if }i\neq j,
   &T_i(\omega_i^a-x_{i+1}\omega_{i+1}^a) &= (\omega_i^a-x_{i+1}\omega_{i+1}^a)T_i ,
\intertext{and}
	 X_{b,i} \varpi_j^a &=\varpi_j^a   X_{b,i}, & Y_{b,i} \varpi_j^a &=\varpi_j^a   Y_{b,i},   \\
	x_i\varpi_j^a &= \varpi_j^ax_i,   & \varpi_{i+1}^a &= \varpi_{i}^{a+1} + x_{i+1} \varpi_{i}^a, \\
   T_i\varpi_j^a &= \varpi_j^aT_i, \text{\ \ if }i\neq j,
   &T_i(\varpi_i^a-x_{i}\varpi_{i-1}^a) &= (\varpi_i^a-x_{i}\varpi_{i-1}^a)T_i ,
\end{align*}
for all $i, j$ and $a \ge 0$.
\end{defn}

The \emph{$(N,M)$-bubbled superalgebra}, denoted  $(N,M)\text{-}\BA_n$, is obtained by
killing each $X_{> N, 0}$, $Y_{n, >M-n}$, $\omega^{\ge N}_0$ and $\varpi^{\ge M-n}_n$.

Let $J_{\infty,\infty}$ be the two-sided ideal of $\BA_n$ generated by $I_{\infty,\infty}$ and, as before, by
\[
(1 + t \omega_{0}^0 + t^2 \omega_{0}^1 + \dots )(1 + t\varpi_{0}^0 + t^2 \varpi_{0}^1 + \dots) = 1.
\]
Hence, $\omega_0^i$ and  $\varpi_0^i$ are equivalent in the quotient $\BA_n/J_{\infty,\infty}$.

\begin{prop}\label{prop:AninBAn}
There is an isomorphism $A_n \cong (0,\infty)\text{-}\BA_n/J_{\infty,\infty}$ induced by the inclusion $A_n \subset \BA_n$.
\end{prop}

\begin{proof}
The proof follows from the same arguments as in Proposition~\ref{prop:bnhinftyiso} together with the observation that we kill all $\omega_0^a$'s.
\end{proof}

Note also that $\Omega_{n,n+1}\cong (n,\infty)\text{-}\BA_1/J_{\infty,\infty}$ if we extend the ground ring to $\bQ$.

\medskip 

Floating dots have a nice interpretation in terms of bubbles.
Indeed, we can view $\omega_n^a$ as corresponding to $X_{n-a,n}$ under Koszul duality, meaning that we can think of $\omega_n^a$ as the algebraic dual $X_{n-a,n}^*$, as explained below. 
Moreover, $\omega_n^a$ interacts like $Y_{-n+a,n}$ with the $T_i$'s (compare~\eqref{eq:crossbubbleY} with \eqref{eq:relomega2}) and  and bubble slides (\eqref{eq:bubblerelationY} and~\eqref{eq:omegashift}), where $Y_{-n+a,n}$ possesses a supposed negative index.
Of course, the same applies for $\varpi_n^a$, which corresponds to $Y_{n-a,n}$ and interacts like $X_{-n+a,n}$.

\begin{rem}
In particular, we can identify $s_{i,n} \in \Omega_{n,n+1}$ with $\omega_n^{n-i}$, recovering the fact it behaves like $Y_{-i,n}$, as explained in~\cite[\S~3]{naissevaz1}.
\end{rem}

Back to $Y_{i,p}$ in the coinvariant algebra $C_N$, that is the elementary symmetric polynomial $\mathcal{e}_i(z_{p+1}, \dots, z_N) = \mathcal{e}_i(z_{p+2}, \dots, z_N) + z_{p+1} \mathcal{e}_{i-1}(z_{p+2},\dots, z_N)$, and acting on it with $s_p \in S_N$, we get
\begin{align*}
s_p(Y_{i,p}) &= \mathcal{e}_i(z_{p+2}, \dots, z_N) + z_{p} \mathcal{e}_{i-1}(z_{p+2},\dots, z_N) \\
&= Y_{i,p} + (z_p - z_{p+1}) Y_{i-1,p+1},
\end{align*}
and $s_p(Y_{i,k}) = Y_{i,k}$ for $k \neq p$. This is suggestive in view of our action of the symmetric group on $\omega_j \in R$ from \S\ref{ssec:algan}. 
In view of the explanation about quadratic duality from \cite[\S~3.4]{naissevaz1}, we can also give it an interpretation in terms of ``duals of $X_{i,p} \in \bnh$". 

\smallskip

For this we have to consider the polynomial space $P(\bnh_n)$ of $\bnh_n$,
which is given by the subalgebra of $\bnh_n$ generated by the $x_i$'s and all $X_{i,p}$ and $Y_{i,p}$,
and extending the scalars to the field of rational fractions $\bQ(\und x_n) = \bQ(x_1, \dots, x_n)$. Hence, $P(\bnh_n)$ is a $\bQ(\und x_n)$-vector space, admitting as basis $\brak{X_{0,p}, X_{1,p}, \dots, Y_{0,p}, Y_{1,p}, \dots}$ for any $p$.

\smallskip

There is an action of $S_n$ on $P(\bnh_n)$, with $s_p$ exchanging $x_p$ and $x_{p+1}$ and acting invariantly on $X_{i,k}$ and $Y_{i,k}$ whenever $k\neq p$. For $k = p$, we have
\begin{align*}
s_p(X_{i,p}) &= X_{i,p} + (x_{p+1} - x_{p}) X_{i-1,p-1}, \\
s_p(Y_{i,p}) &= Y_{i,p} + (x_p - x_{p+1}) Y_{i-1,p+1}. 
\end{align*}
 This action induce an action on the algebraic dual space $P(\bnh_n)^*$ with $s_p f = s_p \circ f \circ s_p$ for any $f \in P(\bnh_n)^* : P(\bnh_n) \rightarrow \bQ(\und x_n)$. In particular, we compute
\begin{align*}
(s_p X_{i,p}^*)(X_{i+k,p}) &= s_p \circ X_{i,p}^*(s_p(X_{i+k,p})) \\
&= s_p \circ X_{i,p}^*( X_{i+k,p} + (x_{p+1} - x_{p}) X_{i+k-1,p-1}) \\
&= s_p \circ X_{i,p}^* \biggl( X_{i+k,p} +(x_{p+1} - x_p) \sum_{\ell = 0}^{i-1+k} (-1)^\ell x_p^\ell X_{i+k-1-\ell,p} \biggr) \\
&= \begin{cases}
0, &\text{ if } k < 0, \\
1, &\text{ if } k = 0, \\
(-1)^{k-1} x_{p+1}^{k-1}(x_p - x_{p+1}) , &\text{ if } k > 0,
\end{cases}
\end{align*}
for all $k \in \bZ$, and where we used the fact $X_{i,p-1} =  \sum_{\ell = 0}^i (-1)^\ell x_p^\ell X_{i-\ell,p}$,  which can be deduced from \eqref{eq:bubblerelationX}. Then we compute
\begin{align*}
\left(X_{i,p}^* + (x_p - x_{p+1}) X_{i+1,p+1}^*\right)&(X_{i+k,p})  \\
&= X_{i,p}^*(X_{i+k,p})  + (x_p - x_{p+1}) X_{i+1,p+1}^*(X_{i+k,p})  \\
&= X_{i,p}^*(X_{i+k,p})  + (x_p - x_{p+1}) X_{i+1,p+1}^*\biggl(\sum_{\ell = 0}^{i+k} (-1)^\ell x_{p+1}^\ell X_{i+k-\ell,p+1}\biggr) 
 \\
&= \begin{cases}
0, &\text{ if } k < 0, \\
1, &\text{ if } k = 0, \\
(-1)^{k-1}x_{p+1}^{k-1}(x_{p} - x_{p+1}), &\text{ if } k > 0.
\end{cases} 
\end{align*}
Hence, we obtain $ s_p(X_{i,p}^*) = X_{i,p}^* + (x_p - x_{p+1}) X_{i+1,p+1}^*$ and  $s_p(X_{i,k}^*) = X_{i,k}^*$ for $k \ne p$. Then, if we think of $\omega_{p}^{p-i}$ as $X_{i,p}^*$, we recover the action on floating dots,
\begin{align*}
s_p(\omega_{p}^{p-i}) &= \omega_p^{p-i} + (x_p - x_{p+1}) \omega_{p+1}^{p-i},&
s_p(\omega_k^{k-i}) &= \omega_k^{k-i},
\end{align*}
 as expected. The same story applies for $\varpi_p^{p-i}$ and $Y_{i,p}^*$.

\subsubsection{Equivariant cohomology}\label{ssec:equivcoh}
The $GL(N)$-equivariant cohomologies of the finite Grassmannian varieties and their iterated flags are described by
\begin{align*}
H^*_{GL(N)}(G_{n;N}) &\cong \bQ[X_{1,n}, \dots, X_{n,n}, Y_{1,N-n}, Y_{N-n,N-n}], \\
H^*_{GL(N)}(G_{n,n+1;N}) &\cong \bQ[X_{1,n}, \dots, X_{n,n}, \xi_{n+1}, Y_{1,N-n-1}, Y_{N-n-1,N-n-1}].
\end{align*}
See~\cite{fulton} for details (see also \cite[\S3]{L2} for an exposition in the context of higher representation theory).
Therefore, we have
\[
H^*_{GL(N)}(G_{n,n+1;N}) \cong (n,N-n)\text{-}\bnh_1,
\]
 and in general equivariant cohomology of iterated flag varieties can be expressed using $\bnh$.

%
%

%
%
\section{Categorical $\slt$-action}\label{sec:cataction}

This section is dedicated to proving that the algebra $A(m)$, which is a generalization of $A$ defined below, categorify the $U_q(\slt)$ Verma module with highest weight~$\lambda q^m$, denoted by~$M(\lambda q^m)$ (see~\cite[\S 2]{naissevaz1} for details about those modules and conventions for quantum $\slt$ and Vermas). This will be done using a similar approach as the one used by Kang--Kashiwara in~\cite{KK} to prove that cyclotomic quotients of KLR algebras categorify the finite-dimensional, irreducible representation of $U_q(\g)$.

\begin{rem}From now on, we will assume again $A_n$ to be defined over $\bZ$ except when specified otherwise (e.g. when we need to compute Grothendieck groups). Also by module we will mean bigraded supermodule, with homomorphisms preserving the degree. In general, we will tend to drop
all ``super'' prefixes, whenever it is clear from the context (e.g. superbimodule, superfunctor, etc. ).
\end{rem}

For all $m \in \bZ$, we define a (super)algebra $A_n(m)$ in the same way we have defined $A_n$, but with floating dots having minimal label $m+1$ instead of $0$. This means a floating dot in $A_n(m)$ can have a negative label whenever $m < -1$. Thus, in $A_n(m)$, the floating dots with minimal label have $q$-degree given by   $\deg_q(\omega_i^{m+1})=2-2i+2m$. In this context, we draw a floating dot as unlabeled if it has label $m+1$, and a tight floating dot has also label $m+1$. We also write $\tilde \omega_i = \omega_i^{m+1}$.
The analogs of the rings $R$ and $R^{S_n}$ from~\S\ref{ssec:algan} and~\S\ref{ssec:Sninvariants}
are denoted by $R(m)$ and $R^{S_n}(m)$, respectively. 
 We let
$A(m)=\bigoplus_{n\in\bN}A_n(m)$. Clearly, $A_n(m)$ has similar properties as $A_n$, and in particular $A_n(-1) = A_n$. There is also an inclusion $A_n(m) \subset A_n(m')$ for $m \ge m'$, sending $x_i$ to $x_i$, $\omega_i^a$ to $\omega_i^a$ and $T_i$ to $T_i$. %

\subsection{Categorical $\slt$-commutator}\label{ssec:resind-cataction}

The inclusion of algebras $\imath : A_n(m) \hookrightarrow A_{n+1}(m)$ that adds a vertical strand to the right of a diagram gives rise to an induction functor 
\[
\Ind_n^{n+1} : A_n(m)\smod \rightarrow A_{n+1}(m)\smod.
\]
 In terms of bimodules,
 this functor can be viewed as tensoring from the left with the ($A_{n+1}(m)$,$A_n(m)$)-bimodule $A_{n+1}(m)$:
\[
\Ind_n^{n+1}  = A_{n+1}(m) \otimes_{A_n(m)}(-).
\]
Taking its right adjoint gives a restriction functor  
\[
\Res_n^{n+1} : A_{n+1}(m)\smod  \rightarrow  A_{n}(m)\smod,
\]
which consists in seeing an $A_{n+1}(m)$-module $M$ as an $A_n(m)$-module ${_{A_{n}(m)}}M$ with action given by $a \bullet_{A_n(m)}(-) = \imath(a) \bullet_{A_{n+1}(m)}(-)$. In terms of bimodules, it is given by tensoring with the ($A_n(m), A_{n+1}(m)$)-bimodule $A_n(m)$:
\[
\Res_n^{n+1} \cong A_n(m) \otimes_{A_{n+1}(m)} (-),
\]

\begin{rem}
The fact that it is an adjunction comes from the usual observation that
\[
A_n(m) \otimes_{A_{n+1}(m)} (-) = \HOM_{A_{n+1}(m)}(A_{n+1}(m)_{A_n(m)}, -),
\]
together with 
\[
\Hom_{A_{n+1}(m)}\left(A_{n+1}(m) \otimes_{A_n(m)} M, N\right) = \Hom_{A_{n}(m)}\left(M, \HOM_{A_{n+1}(m)}(A_{n+1}(m), N)\right),
\]
for all $A_n(m)$-module $M$ and $A_{n+1}(m)$-module $N$.
\end{rem}

In order to simplify notation we introduce some conventions. We suppose that $m$ is fixed and we write $\otimes_n$ for $\otimes_{A_n(m)}$. We draw $A_n(m)$ as a box
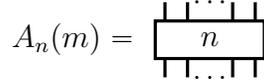
\[
	A_n(m) = 
	\tikz[anchor = center, xscale = 1.2, very thick,baseline={([yshift=-.5ex]current bounding box.center)}]{
	    	\node [anchor = west] at (0,0) {\boxAn};
	}
\]
and we write the tensor product $\otimes_n$ as concatenating boxes on top of each other, reading from right to left.

\begin{exe}
We can write
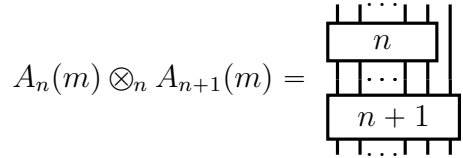
\[
	A_n(m) \otimes_n A_{n+1}(m) =
	\tikz[xscale = 1.2, very thick,baseline={([yshift=-.5ex]current bounding box.center)}]{
		\node [anchor = west] at (0,1) {\boxAn};
		\node [anchor = west] at (0,0) {\boxAnp};
	 	\draw (1.51,.5)-- (1.51,1.5);
	}
\]
\end{exe}

When $A_{n+1}(m)$ is regarded  as an $A_{n}(m)$-left module, as an $A_{n}(m)$-right module or as
an $(A_n(m), A_n(m))$-bimodule, we draw respectively
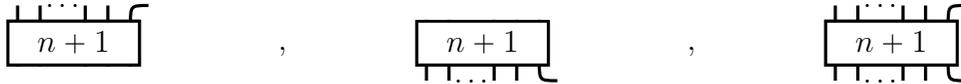
\begin{align*}
	&\tikz[xscale = 1.2, very thick,baseline={([yshift=-.5ex]current bounding box.center)}]{
		\node [anchor = west] at (0,0) {\boxnAnp}; 
			\node (rect) at (.125,-.475) [anchor = west, draw=white,fill=white,minimum width=2cm,minimum height=.25cm] {};
	}&&, &
	&\tikz[xscale = 1.2, very thick,baseline={([yshift=-.5ex]current bounding box.center)}]{
		\node [anchor = west] at (0,0) {\boxAnpn};
			\node (rect) at (.125,.475) [anchor = west, draw=white,fill=white,minimum width=2cm,minimum height=.25cm] {};
	}&&, &
	&\tikz[xscale = 1.2, very thick,baseline={([yshift=-.5ex]current bounding box.center)}]{
		\node [anchor = west] at (0,0) {\boxnAnpn};
	}
\end{align*}
Also, for an $A_n(m)$-module $M$, writing $M \otimes G$ where $G = \bigoplus_{x\in X} x\bZ$ is a free $\bZ\times \bZ \times \bZ/2\bZ$-graded abelian group generated by all (homogeneous) $x$ in some countable set $X$ means taking the direct sum
\[
M \otimes G \cong \bigoplus_{x \in X}  q^{\deg_q(x)} \lambda^{\deg_\lambda(x)}  \Pi^{p(x)}M,
\]
where we recall $\Pi$ is the parity shift, $q^a$ is a $q$-degree shift by $a$, and $\lambda^b$ is a $\lambda$-degree shift by $b$.
In particular, if we write $M \otimes x$ for $x\in A_{n+1}(m)$, it means we take the tensor product with~$\bZ x$, that is $ q^{\deg_q(x)} \lambda^{\deg_\lambda(x)}  \Pi^{p(x)}M$. We also write $M \otimes (x \oplus y) = M \otimes \left(\bZ x \oplus \bZ y \right)$. By abuse of notation, we can assume $\oplus$ is distributive with respect to the composition in $A_{n+1}(m)$.

\begin{exe}
We interpret a formula as in Lemma~\ref{lem:newbasisdecomp_0} below as
\begin{align*}
 M \otimes \bZ[x_{n+1}] T_{w}(y \oplus z) =&  M \otimes \left( \bZ[x_{n+1}]  T_{w} y \oplus \bZ[x_{n+1}] T_w z \right) \\
=& \bigoplus_{i \ge 0}  \bigl( q^{\deg_q(x_{n+1}^i T_w y)} \lambda^{\deg_\lambda(x_{n+1}^i T_w y)} \Pi^{p(y)} M  \\ 
&\mspace{30mu}\oplus 
q^{\deg_q(x_{n+1}^i T_w z)} \lambda^{\deg_\lambda(x_{n+1}^i T_w z)} \Pi^{p(z)} M   \bigr).
\end{align*}
\end{exe}

We will now prove that $\Ind_n^{n+1}$ and $\Res_n^{n+1}$ give rise to a categorical version of the $\slt$-commutator 
\begin{align*}
EF -EF &= \frac{K-K^{-1}}{q-q^{-1}}, & \frac{1}{q-q^{-1}} &= -q(1+q^{2}+ \dots),
\end{align*}
as in~\cite[\S6]{naissevaz1}. This will be presented in the form of a short exact sequence:

\begin{thm}\label{thm:sesAnm}
There is a short exact sequence
\begin{align*}
0 \rightarrow q^{-2} A_{n}(m) \otimes_{n-1} A_{n}(m)  &\rightarrow A_{n+1}(m) \\
 &\rightarrow \left(A_{n}(m) \otimes \bZ[\xi]\right) \oplus \left(q^{2m-4n}\lambda^{2} \Pi A_{n}(m) \otimes \bZ [\xi] \right)\rightarrow 0,
\end{align*}
of $(A_n(m), A_n(m))$-bimodules.
\end{thm}

Informally, we can interpret the short exact sequence in the theorem above in terms of diagrams as  
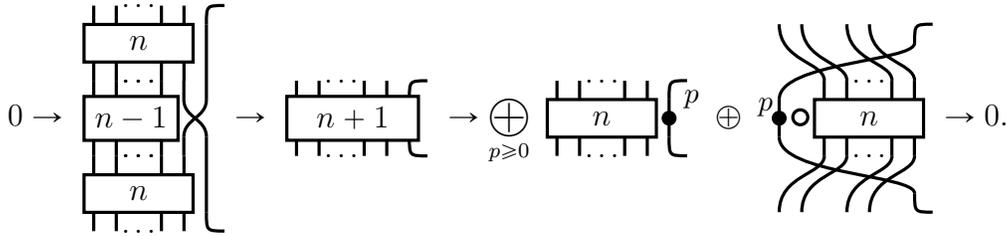
\[
	0 \rightarrow
	\tikz[xscale = 1.2, very thick,baseline={([yshift=-.5ex]current bounding box.center)}]{
		\node [anchor = west] at (-.01,1) {\boxAn};
		\node [anchor = west] at (-.01,0) {
			\tikz[anchor = center, xscale = 1.2, very thick]{
	    		\draw (0,-.5)-- (0,.5);
	 		\draw (.25,-.5)-- (.25,.5);
	 		\draw (.75,-.5)-- (.75,.5);
			\node (rect) at (-.125,0) [anchor = west, draw,fill=white,minimum width=1.25cm,minimum height=.5cm] {$n-1$};
			\node at (.525,.5) {\small $\dots$}; \node at (.525,-.5) {\small $\dots$};
			}
		};
	 	\draw +(1.5,1.25) .. controls (1.5,1.5) ..  + (1.7,1.5);
		\draw (1.5,1.25) --(1.5,.25); \draw (1.5,-1.25) --(1.5,-.25);	 \draw (1.25,.5) --(1.25,.25); \draw (1.25,-.5) --(1.25,-.25);
		\draw (1.25,-.25) .. controls (1.25, 0) and (1.5, 0) .. (1.5, .25); \draw (1.5,-.25) .. controls (1.5, 0) and (1.25, 0) .. (1.25, .25);
	 	\draw +(1.5,-1.25) .. controls (1.5,-1.5) ..  + (1.7,-1.5);
		\node [anchor = west] at (-.01,-1) {\boxAn};
	}\rightarrow
	\tikz[xscale = 1.2, very thick,baseline={([yshift=-.5ex]current bounding box.center)}]{
		\node [anchor = west] at (0,1) {\boxnAnpn};
	}\rightarrow
	\bigoplus_{p \ge 0}
	\tikz[xscale = 1.2, very thick,baseline={([yshift=-.5ex]current bounding box.center)}]{
		\node [anchor = west] at (0,1) {\boxAn};
	 	\draw (1.5,.75)-- (1.5,1.25) node [midway,fill=black,circle,inner sep=2pt]{};  \node at (1.75,1.25) {\small$p$};
	 	\draw +(1.5,1.25) .. controls (1.5,1.5) ..  + (1.7,1.5);
	 	\draw +(1.5,.75) .. controls (1.5,.5) ..  + (1.7,.5);
	} \oplus
	 \tikz[xscale = 1.2, very thick,baseline={([yshift=-.5ex]current bounding box.center)}]{
		\node [anchor = west] at (.24,1) {\boxAnu};
	 	\draw +(1.5,2) .. controls (1.5,2.25) ..  + (1.7,2.25);
	 	\draw +(1.5,0) .. controls (1.5,-.25) ..  + (1.7,-.25);
		\draw + (1.5,2)  .. controls (1.5,1.75) and (0,1.75) ..  +  (0,1.25);
		\draw + (0,0.75)  .. controls (0,.25) and (1.5,.25) ..  +  (1.5,0);
		\draw + (0,2.25)  .. controls (0,1.75) and (.5,1.75) ..  +  (.5,1.5);
		\draw + (.25,2.25)  .. controls (.25,1.75) and (.75,1.75) ..  +  (.75,1.5);
		\draw + (.75,2.25)  .. controls (.75,1.75) and (1.25,1.75) ..  +  (1.25,1.5);
		\draw + (1,2.25)  .. controls (1,1.75) and (1.5,1.75) ..  +  (1.5,1.5);
		\draw + (0,-.25)  .. controls (0,.25) and (.5,.25) ..  +  (.5,.5);
		\draw + (.25,-.25)  .. controls (.25,.25) and (.75,.25) ..  +  (.75,.5);
		\draw + (.75,-.25)  .. controls (.75,.25) and (1.25,.25) ..  +  (1.25,.5);
		\draw + (1,-.25)  .. controls (1,.25) and (1.5,.25) ..  +  (1.5,.5);
		\fdot{.225,1};
	 	\draw (0,1.25)-- (0,.75) node [midway,fill=black,circle,inner sep=2pt]{};  \node at (-.15,1.15) {\small$p$};
	} \rightarrow 0.
\]
The rightmost bimodule in the short exact sequence should not be immediately translated from the diagram as depicted, but rather as the quotient of $A_{n+1}(m)$ corresponding to such diagrams, as explained in the proof of Theorem~\ref{thm:sesAnm} below. It is isomorphic to $q^{2m-4n}\lambda^2\Pi A_n(m) \otimes \bZ[\xi]$, since crossings, dots and floatings dots, when added at the top or bottom, can freely slide to the other side. Indeed, in doing so in $A_{n+1}(m)$ can produce remaining terms that are killed when projected onto the quotient (see the proof of Lemma~\ref{lem:samedecomp} below).

\smallskip

The induction and restriction functors still need to be shifted accordingly in order to get an ``$\slt$-commutator" relation. To this end, we define the functors
\begin{align*}
\F_n &= \Ind_n^{n+1}, &
\E_n &= q^{2n-m} \lambda^{-1} \Res_n^{n+1}, &
\Q_n &=   -\otimes q \Pi \bZ[\xi],
\end{align*}
where $\deg(\xi) = (2,0)$ and $p(\xi) = 0$.
As a direct consequence of Theorem~\ref{thm:sesAnm} we get the following.

\begin{cor}\label{cor:EF-rel}
There is a natural short exact sequence
\[
0 \rightarrow \F_{n-1}\E_{n-1} \rightarrow \E_{n}\F_{n} \rightarrow q^{m-2n}\lambda \Q_{n+1}  \oplus q^{2n-m}\lambda^{-1} \Pi \Q_{n+1}\rightarrow 0.
\]
\end{cor}

\begin{rem}
We can speak of a short exact sequence of functors (i.e. natural short exact sequence) since the category of functors between abelian categories is itself abelian (recall we only consider degrees and parity preserving maps, so that our categories of modules are abelian).
\end{rem}

The remaining of this subsection is dedicated to the proof of Theorem~\ref{thm:sesAnm}.

\begin{lem}\label{lem:newbasisdecomp_0}
As a left $A_{n}(m)$-module, $A_{n+1}(m)$ is free with decomposition given by
\[
 \bigoplus_{a = 1}^{n+1} A_{n}(m) \otimes \bZ[x_{n+1}] T_{n}\dotsm T_{a}(1 \oplus \theta_a),
\]
where $\theta_a = T_{a-1}\dotsm T_1 \tilde\omega_1T_1 \dotsm T_{a-1}$, and it is understood that $T_{n}\dotsm T_{n+1} = 1$ and $ T_1\dotsm T_{0} = 1$.
\end{lem}

 This can be pictured as
\[
	\tikz[ very thick,baseline={([yshift=-.5ex]current bounding box.center)}]{
		\node[anchor = west]  at (0,0) {\boxnAnp};
			\node (rect) at (.125,-.475) [anchor = west, draw=white,fill=white,minimum width=2cm,minimum height=.25cm] {};
		
	}
	\quad \cong \quad  \bigoplus_{a = 1}^{n+1} \bigoplus_{p \ge 0}
	\tikz[xscale = 1.2, very thick,baseline={([yshift=-.5ex]current bounding box.center)}]{
		\node [anchor = west] at (-.01,0) {\boxAn};
		\draw (.25,-.5) -- (.25,-1);
		\draw (.5,-.5)  -- (.5,-1);
	 	\draw +(1.5,.25) .. controls (1.5,.5) ..  + (1.7,.5);
	 	\draw (1.5,-.5) -- (1.5,.25) node [near start,fill=black,circle,inner sep=2pt]{};
		\draw + (1.5,-.5)  .. controls (1.5,-.75) and (1,-.75) ..  +  (1,-1) node[below] {\small $a$} ;
		 \node at (1.7,-.125) {\small$p$};
		\draw + (1,-.5)  .. controls (1,-.75) and (1.25,-.75) ..  +  (1.25,-1);
		\draw + (1.25,-.5)  .. controls (1.25,-.75) and (1.5,-.75) ..  +  (1.5,-1);
	} 
	\oplus
	\tikz[xscale = 1.2, very thick,baseline={([yshift=-.5ex]current bounding box.center)}]{
		\node [anchor = west] at (-.01,0) {\boxAn};
		\fdot{.45,-1.025};
	 	\draw +(1.5,.25) .. controls (1.5,.5) ..  + (1.7,.5);
	 	\draw (1.5,-.5) -- (1.5,.25) node [near start,fill=black,circle,inner sep=2pt]{};
		\draw + (.25,-1)  .. controls (.25,-.75) and (1.5,-.75) ..  +  (1.5,-.5);
		\draw + (.25,-1)  .. controls (.25,-1.25) and (1,-1.25) ..  +  (1,-1.5) node[below] {\small $a$} ;
		\node at (1.7,-.125) {\small$p$};
		\draw + (1,-.5)  .. controls (1,-.75) and (1.25,-1.25) ..  +  (1.25,-1.5);
		\draw + (1.25,-.5)  .. controls (1.25,-.75) and (1.5,-1.25) ..  +  (1.5,-1.5);
		\draw + (.25,-.5)  .. controls (.25,-.75) and (.65,-.75) ..  +  (.65,-1);
		\draw + (.25,-1.5)  .. controls (.25,-1.25) and (.65,-1.25) ..  +  (.65,-1);
		\draw + (.5,-.5)  .. controls (.5,-.75) and (.8,-.75) ..  +  (.8,-1);
		\draw + (.5,-1.5)  .. controls (.5,-1.25) and (.8,-1.25) ..  +  (.8,-1);
	} 
\]

\begin{proof}
Using the third basis from Proposition~\ref{prop:firstbasis} and the right coset decomposition of $S_{n+1}$, i.e. 
\[
S_{n+1} = \bigsqcup_{a=1}^{n+1} S_ns_n \dotsm s_a,
\]
it is not hard to see that $A_{n+1}(m)$ decomposes as left $A_{n}(m)$-module as
\[
A_{n+1}(m) \cong \bigoplus_{a = 1}^{n+1} A_{n}(m) \otimes \bZ[x_{n+1}, \tilde\omega_{n+1}] T_{n}\dotsm T_{a}.
\]
Moreover, we have $ \deg(\tilde\omega_{n+1}) = \deg(T_{n} \dotsm T_1 \tilde\omega_1)$ and thus, since $A_{n+1}(m)$ is locally of finite dimension (i.e. finite-dimensional in each bidegree as a $\bZ$-modules), it is enough to show that 
\begin{equation}
 \bigoplus_{a = 1}^{n+1} \left(A_{n}(m) \otimes \bZ[x_{n+1}] T_{n}\dotsm T_{a}\right) \oplus \left(A_{n}(m)\otimes \bZ[x_{n+1}] T_n \dotsm T_1 \tilde\omega_1 T_1\dotsm T_{n+1-a}\right) \label{eq:newbasis}
 \end{equation}
generates $A_{n+1}(m)$. Indeed, it means we can split everything into free $\bZ$-modules of  finite rank (one for each bidegree), and then we exhibit a generating family having the same number of elements as the rank, thus it is a basis. Then the only thing that requires a proof is that (\ref{eq:newbasis}) generates
\[
\bigoplus_{a = 1}^{n+1} A_{n}(m) \otimes \bZ[x_{n+1}]  \tilde\omega_{n+1} T_{n}\dotsm T_{a}.
\]
To this end, we first observe that applying repetitively Lemma~\ref{lem:wdtranslation} together with the nilHecke relations~(\ref{eq:relnh1}-\ref{eq:relnh2}) allows us to write $\tilde\omega_{n+1}$ as a linear combination in (\ref{eq:newbasis}) so that
\begin{align*}
\bigoplus_{a = 1}^{n+1} A_{n}(m) \otimes \bZ[x_{n+1}]  \tilde\omega_{n+1} T_{n}\dotsm T_{a} \subset & 
\bigoplus_{a = 1}^{n+1}  \bigoplus_{\ell = 1}^{n+1}  A_{n}(m) \otimes \bZ[x_{n+1}] T_{n}\dotsm T_{\ell} T_{n}\dotsm T_{a} \\
&\oplus    A_{n}(m) \otimes \bZ[x_{n+1}]   T_n \dotsm T_1 \tilde\omega_1 T_1\dotsm T_{n+1-\ell} T_{n}\dotsm T_{a}.
\end{align*}
Then, using the braid move~\eqref{eq:crossings}, we get that $T_n \dotsm T_1 \tilde\omega_1 T_1 \dotsm T_{n+1-\ell} T_n \dotsm T_a$ is generated by~\eqref{eq:newbasis}. The same applies for $T_n \dotsm T_\ell T_n \dotsm T_a$ and this concludes the proof.
\end{proof}

\begin{cor}
The family of elements~(\ref{eq:tightbasis}) forms a basis for $A_n(m)$.
\end{cor}

\begin{proof}
This follows by recursion on $n$, applying Lemma~\ref{lem:newbasisdecomp_0} at each step.
\end{proof}

It is not hard to see we can choose any (fixed) position on each strand for placing the dots on the diagrams in the construction of the basis, after inserting the tightened floating dots. In particular, we get the following.

\begin{lem}\label{lem:newbasisdecomp}
As a left $A_{n}(m)$-module, $A_{n+1}(m)$ is free with decomposition given by
\[
 \bigoplus_{a = 1}^{n+1} A_{n}(m) \otimes T_{n}\dotsm T_{a}  (\bZ[x_{a}] \oplus \theta_a^{x_1}),
\]
where $\theta_a^{x_1} = T_{a-1}\dotsm T_1  \bZ[x_1] \tilde\omega_1 T_1 \dotsm T_{a-1}$. 
\end{lem}

This can be written as
\[
	\tikz[ very thick,baseline={([yshift=-.5ex]current bounding box.center)}]{
		\node[anchor = west]  at (0,0) {\boxnAnp};
			\node (rect) at (.125,-.475) [anchor = west, draw=white,fill=white,minimum width=2cm,minimum height=.25cm] {};
		
	}
	\quad\cong \quad  \bigoplus_{a = 1}^{n+1} \bigoplus_{p \ge 0}
	\tikz[xscale = 1.2, very thick,baseline={([yshift=-.5ex]current bounding box.center)}]{
		\node [anchor = west] at (-.01,0) {\boxAn};
		\draw (.25,-.5) -- (.25,-1);
		\draw (.5,-.5)  -- (.5,-1);
	 	\draw +(1.5,.25) .. controls (1.5,.5) ..  + (1.7,.5);
	 	\draw (1.5,-.5) -- (1.5,.25);
		\draw + (1.5,-.5)  .. controls (1.5,-.75) and (1,-.75) ..  +  (1,-1) node[below] {\small $a$}  node [pos=.8,fill=black,circle,inner sep=2pt]{} ;
		\node at (.85,-.75) {\small$p$};
		\draw + (1,-.5)  .. controls (1,-.75) and (1.25,-.75) ..  +  (1.25,-1);
		\draw + (1.25,-.5)  .. controls (1.25,-.75) and (1.5,-.75) ..  +  (1.5,-1);
	} 
	\oplus
	\tikz[xscale = 1.2, very thick,baseline={([yshift=-.5ex]current bounding box.center)}]{
		\node [anchor = west] at (-.01,0) {\boxAn};
		\fdot{.45,-1.025};
	 	\draw +(1.5,.25) .. controls (1.5,.5) ..  + (1.7,.5);
	 	\draw (1.5,-.5) -- (1.5,.25);
		\draw + (.25,-1)  .. controls (.25,-.75) and (1.5,-.75) ..  +  (1.5,-.5);
		\draw + (.25,-1)  .. controls (.25,-1.25) and (1,-1.25) ..  +  (1,-1.5) node[below] {\small $a$}  node [pos=0,fill=black,circle,inner sep=2pt]{} ;
		  \node at (0.05,-.9) {\small$p$};
		\draw + (1,-.5)  .. controls (1,-.75) and (1.25,-1.25) ..  +  (1.25,-1.5);
		\draw + (1.25,-.5)  .. controls (1.25,-.75) and (1.5,-1.25) ..  +  (1.5,-1.5);
		\draw + (.25,-.5)  .. controls (.25,-.75) and (.65,-.75) ..  +  (.65,-1);
		\draw + (.25,-1.5)  .. controls (.25,-1.25) and (.65,-1.25) ..  +  (.65,-1);
		\draw + (.5,-.5)  .. controls (.5,-.75) and (.8,-.75) ..  +  (.8,-1);
		\draw + (.5,-1.5)  .. controls (.5,-1.25) and (.8,-1.25) ..  +  (.8,-1);
	} 
\]

\begin{proof}
Again, it is enough to prove that $ \bigoplus_{a = 1}^{n+1} A_{n}(m) \otimes T_{n}\dotsm T_{a}  (\bZ[x_{a}] \oplus \theta_a^{x_1}),$ yields a generating family, and thus that it generates all $\bZ[x_{n+1}] T_{n}\dotsm T_{a}(1 \oplus \theta_a)$ because of Lemma~\ref{lem:newbasisdecomp_0}.
Therefore, we apply the nilHecke relation~(\ref{eq:relnh2}) $n-a+1$ times on $x_{n+1}^p T_n \dotsm T_a$ so that
\[
x_{n+1}^p T_n \dotsm T_a = x_{n+1}^{p-1} T_n \dotsm T_a x_a + R,
\]
where $R \in  \bigoplus_{b = a+1}^{n+1} A_n(m) \otimes \bZ[x_{n+1}]/(x_{n+1})^p T_{n}\dotsm T_{b}$. Thus, a backward induction on $b$ and an induction on $p$ shows that we can generate all $\bZ[x_{n+1}] T_{n}\dotsm T_{a}1$. A similar reasoning can be applied for $\bZ[x_{n+1}] T_{n}\dotsm T_{a}\theta_a$, with
\[
x_{n+1}^p T_{n}\dotsm T_{1} x_1^q \tilde\omega_1 T_1 \dotsm T_{a-1} =  x_{n+1}^{p-1} T_{n}\dotsm T_{1} x_1^{q+1} \tilde\omega_1 T_1 \dotsm T_{a-1}  + R
\]
where $R \in  \bigoplus_{b = 1}^{n+1} A_n(m) \otimes \bZ[x_{n+1}] T_{n}\dotsm T_{b}$, and this concludes the proof.
\end{proof}

\begin{lem}\label{lem:samedecomp}
As a right $A_{n}(m)$-module, $A_{n+1}(m)$ is free with decomposition 
\[
 \bigoplus_{a = 1}^{n+1}  (\bZ[x_{a}]  \oplus \theta_a^{x_1})  T_{a} \dotsm  T_{n}   \otimes A_{n}(m),
\]
and if $y \in \theta_{n+1}^{x_1}\otimes A_n(m)$ then there exists some $w \notin  A_n(m) \otimes \theta_{n+1}^{x_1}$ such that $y - w \in A_n(m) \otimes \theta_{n+1} ^{x_1}$.
\end{lem}

\begin{proof}
The decomposition follows from the same arguments as above and thus, we only prove the second claim. 

By the Reidemeister~3 move~(\ref{eq:crossings}), we can freely slide crossings over $T_{n}\dotsm T_1  x_1^p \tilde\omega_1 T_1 \dotsm T_{n}$. 
Because of the nilHecke relations~(\ref{eq:relnh1}-\ref{eq:relnh2}), we can slide dots over crossings at the cost of adding diagrams with less crossings, such that dots slide over $T_{n}\dotsm T_1  x_1^p \tilde\omega_1 T_1 \dotsm T_{n}$ at the cost of adding terms not in $A_n(m) \otimes \theta_{n+1} ^{x_1}$. This is best illustrated by the following equation
\begin{align*}
\tikz[very thick,xscale=.75,yscale=1,baseline={([yshift=.4ex]current bounding box.center)}]{
	      \fdot{.4,0}; 
	\draw +(0,-.75) .. controls (0,-.375) and (1,-.375) .. +(1,0) .. controls (1,.375) and (0,.375) .. +(0,.75);
	\draw[xshift=4pt] +(0,-.75) .. controls (0,-.45) and (1,-.45) .. +(1,0) .. controls (1,.45) and (0,.45) .. +(0,.75);
	\draw +(2,-.75) .. controls (2,-.375) and (3,-.375) .. +(3,0) .. controls (3,.375) and (2,.375) .. +(2,.75);
	\draw[xshift=4pt] +(2,-.75) .. controls (2,-.45) and (3,-.45) .. +(3,0) .. controls (3,.45) and (2,.45) .. +(2,.75);
	      \draw  +(3,-.75)   ..  controls (3,-.375) and (0,-.375) ..  +(0,0)  node [pos=1,fill=black,circle,inner sep=2pt]{} ; \node at (-.25,.15) {\small$p$};		
	      \draw  +(3,.75).. controls (3,.375) and (0,.375) ..  +(0,0); 
	\draw +(1,-.75)  node[below] { \small $a$}    .. controls (1,-.375) and (2,-.375) .. +(2,0) node [pos=.25,fill=black,circle,inner sep=2pt]{} .. controls (2,.375) and (1,.375) .. +(1,.75) ; 	
	 }
\ &= \ 
\tikz[very thick,xscale=.75,yscale=1,baseline={([yshift=.4ex]current bounding box.center)}]{
	      \fdot{.4,0}; 
	\draw +(0,-.75) .. controls (0,-.375) and (1,-.375) .. +(1,0) .. controls (1,.375) and (0,.375) .. +(0,.75);
	\draw[xshift=4pt] +(0,-.75) .. controls (0,-.45) and (1,-.45) .. +(1,0) .. controls (1,.45) and (0,.45) .. +(0,.75);
	\draw +(2,-.75) .. controls (2,-.375) and (3,-.375) .. +(3,0) .. controls (3,.375) and (2,.375) .. +(2,.75);
	\draw[xshift=4pt] +(2,-.75) .. controls (2,-.45) and (3,-.45) .. +(3,0) .. controls (3,.45) and (2,.45) .. +(2,.75);
	      \draw  +(3,-.75)   ..  controls (3,-.375) and (0,-.375) ..  +(0,0)  node [pos=1,fill=black,circle,inner sep=2pt]{} ; \node at (-.25,.15) {\small$p$};		
	      \draw  +(3,.75).. controls (3,.375) and (0,.375) ..  +(0,0); 
	\draw +(1,-.75)  node[below] { \small $a$}    .. controls (1,-.375) and (2,-.375) .. +(2,0) .. controls (2,.375) and (1,.375) .. +(1,.75)  node [pos=.75,fill=black,circle,inner sep=2pt]{}; 	
	 }
\ + \ 
\tikz[very thick,xscale=.75,yscale=1,baseline={([yshift=.4ex]current bounding box.center)}]{
	      \fdot{.4,0}; 
	\draw +(0,-.75) .. controls (0,-.375) and (1,-.375) .. +(1,0) .. controls (1,.375) and (0,.375) .. +(0,.75);
	\draw[xshift=4pt] +(0,-.75) .. controls (0,-.45) and (1,-.45) .. +(1,0) .. controls (1,.45) and (0,.45) .. +(0,.75);
	\draw +(2,-.75) .. controls (2,-.375) and (3,-.375) .. +(3,0) .. controls (3,.375) and (2,.375) .. +(2,.75);
	\draw[xshift=4pt] +(2,-.75) .. controls (2,-.45) and (3,-.45) .. +(3,0) .. controls (3,.45) and (2,.45) .. +(2,.75);
	      \draw  +(1,-.75) node[below] { \small $a$}    ..  controls (1,-.375) and (0,-.375) ..  +(0,0)  node [pos=1,fill=black,circle,inner sep=2pt]{} ; \node at (-.25,.15) {\small$p$};		
	      \draw  +(3,.75).. controls (3,.375) and (0,.375) ..  +(0,0); 
	\draw +(3,-.75)     .. controls (3,-.375)  and (1,.375) .. +(1,.75);
	 }
\ - \ 
\tikz[very thick,xscale=.75,yscale=1,baseline={([yshift=.4ex]current bounding box.center)}]{
	      \fdot{.4,0}; 
	\draw +(0,-.75) .. controls (0,-.375) and (1,-.375) .. +(1,0) .. controls (1,.375) and (0,.375) .. +(0,.75);
	\draw[xshift=4pt] +(0,-.75) .. controls (0,-.45) and (1,-.45) .. +(1,0) .. controls (1,.45) and (0,.45) .. +(0,.75);
	\draw +(2,-.75) .. controls (2,-.375) and (3,-.375) .. +(3,0) .. controls (3,.375) and (2,.375) .. +(2,.75);
	\draw[xshift=4pt] +(2,-.75) .. controls (2,-.45) and (3,-.45) .. +(3,0) .. controls (3,.45) and (2,.45) .. +(2,.75);
	      \draw  +(3,-.75)   ..  controls (3,-.375) and (0,-.375) ..  +(0,0)  node [pos=1,fill=black,circle,inner sep=2pt]{} ; \node at (-.25,.15) {\small$p$};		
	      \draw  +(1,.75).. controls (1,.375) and (0,.375) ..  +(0,0); 
	\draw +(1,-.75)  node[below] { \small $a$}    .. controls (1,-.375) and (3,.375) .. +(3,.75);
	 } \\
\ &=\  
\tikz[very thick,xscale=.75,yscale=1,baseline={([yshift=.4ex]current bounding box.center)}]{
	      \fdot{.4,0}; 
	\draw +(0,-.75) .. controls (0,-.375) and (1,-.375) .. +(1,0) .. controls (1,.375) and (0,.375) .. +(0,.75);
	\draw[xshift=4pt] +(0,-.75) .. controls (0,-.45) and (1,-.45) .. +(1,0) .. controls (1,.45) and (0,.45) .. +(0,.75);
	\draw +(2,-.75) .. controls (2,-.375) and (3,-.375) .. +(3,0) .. controls (3,.375) and (2,.375) .. +(2,.75);
	\draw[xshift=4pt] +(2,-.75) .. controls (2,-.45) and (3,-.45) .. +(3,0) .. controls (3,.45) and (2,.45) .. +(2,.75);
	      \draw  +(3,-.75)   ..  controls (3,-.375) and (0,-.375) ..  +(0,0)  node [pos=1,fill=black,circle,inner sep=2pt]{} ; \node at (-.25,.15) {\small$p$};		
	      \draw  +(3,.75).. controls (3,.375) and (0,.375) ..  +(0,0); 
	\draw +(1,-.75)  node[below] { \small $a$}    .. controls (1,-.375) and (2,-.375) .. +(2,0) .. controls (2,.375) and (1,.375) .. +(1,.75)  node [pos=.75,fill=black,circle,inner sep=2pt]{}; 	
	 }
\ +\ 
\tikz[very thick,xscale=.75,yscale=1,baseline={([yshift=.4ex]current bounding box.center)}]{
	      \fdot{.4,0}; 
	\draw +(0,-.75) .. controls (0,-.375) and (.9,-.375) .. +(.9,0) .. controls (.9,.375) and (0,.375) .. +(0,1.25);
	\draw[xshift=4pt] +(0,-.75) .. controls (0,-.45) and (.9,-.45) .. +(.9,0) .. controls (.9,.45) and (0,.45) .. +(0,1.25);
	\draw +(2,-.75) .. controls (2,.05) and (1.1,.05) .. +(1.1,.5) .. controls (1.1,.95) and (2,.95) .. +(2,1.25);
	\draw[xshift=4pt] +(2,-.75) .. controls (2,.125) and (1.1,.125) .. +(1.1,.5) .. controls (1.1,.875) and (2,.875) .. +(2,1.25);
	      \draw  +(1,-.75) node[below] { \small $a$}    ..  controls (1,-.375) and (0,-.375) ..  +(0,0)  node [pos=1,fill=black,circle,inner sep=2pt]{} ; \node at (-.25,.15) {\small$p$};		
	      \draw  +(3,1.25).. controls (3,.375) and (0,.375) ..  +(0,0); 
	\draw +(3,-.75)     .. controls (3,.125)  and (1,.875) .. +(1,1.25);
	 }
\mspace{9mu} - \ 
\tikz[very thick,xscale=.75,yscale=-1,baseline={([yshift=.4ex]current bounding box.center)}]{
	      \fdot{.4,0}; 
	\draw +(0,-.75) .. controls (0,-.375) and (.9,-.375) .. +(.9,0) .. controls (.9,.375) and (0,.375) .. +(0,1.25);
	\draw[xshift=4pt] +(0,-.75) .. controls (0,-.45) and (.9,-.45) .. +(.9,0) .. controls (.9,.45) and (0,.45) .. +(0,1.25);
	\draw +(2,-.75) .. controls (2,.05) and (1.1,.05) .. +(1.1,.5) .. controls (1.1,.95) and (2,.95) .. +(2,1.25);
	\draw[xshift=4pt] +(2,-.75) .. controls (2,.125) and (1.1,.125) .. +(1.1,.5) .. controls (1.1,.875) and (2,.875) .. +(2,1.25);
	      \draw  +(1,-.75)   ..  controls (1,-.375) and (0,-.375) ..  +(0,0)  node [pos=1,fill=black,circle,inner sep=2pt]{} ; \node at (-.25,-.15) {\small$p$};		
	      \draw  +(3,1.25).. controls (3,.375) and (0,.375) ..  +(0,0); 
	\draw +(3,-.75)   .. controls (3,.125)  and (1,.875) .. +(1,1.25)  node[below] { \small $a$}  ;
	 }
\end{align*}
where double strands represent multiple parallel strands (the number of strands depending on $n$ and $a$). We observe that the last two terms are in $A_n(m) \otimes T_{n} \dotsm T_{a} \theta_a^{x_1}$ and in $A_n(m) \otimes T_n \dotsm T_a \bZ[x_a]$, respectively.  Hence, only the first term is in $A_n(m) \otimes \theta_{n+1}^{x_1}$.
Note that this remains true even if there is another element of $A_n(m)$ below.
Indeed, by Lemma~\ref{lem:newbasisdecomp} it will decompose as a combination of elements without adding any $T_n$. This is important for the proof since we want to be able to `locally' slide dots in the projection.
 Finally, since tight floating dots generates (with crossings and dots) all other floatings dots, we only need to show $\tilde \omega_1$ slides over $T_{n}\dotsm T_1 \tilde\omega_1 T_1 \dotsm T_{n}$. For this we observe using~\eqref{eq:relomega2} and~\eqref{eq:relnh1},
\begin{align*}
	\tikz[very thick,xscale=2,yscale=1.5,baseline={([yshift=-.5ex]current bounding box.center)}]{
		\draw  +(0,0) .. controls (0,0.25) and (0.5, 0.25) ..  +(0.5,0.5);
		\draw  +(0.5,0) .. controls (0.5,0.25) and (0, 0.25) ..  +(0,0.5) node [at end,fill=black,circle,inner sep=2pt]{};  \node at (-.075,.6) {\small$p$};
		\draw  +(0,0.5) .. controls (0,0.75) and (0.5, 0.75) ..  +(0.5,1);
		\draw  +(0.5,0.5) .. controls (0.5,0.75) and (0, 0.75) ..  +(0,1);
	  	 \fdot{.25,.5};  \fdot{.25,.1};
	 } 
\ &= \ 
	\tikz[very thick,xscale=2,yscale=1.5,baseline={([yshift=-.5ex]current bounding box.center)}]{
		\draw  +(0,0) .. controls (0,0.25) and (0.5, 0.25) ..  +(0.5,0.5);
		\draw  +(0.5,0) .. controls (0.5,0.25) and (0, 0.25) ..  +(0,0.5) node [at end,fill=black,circle,inner sep=2pt]{}  node [near start,fill=black,circle,inner sep=2pt]{};  \node at (-.075,.6) {\small$p$};
		\draw  +(0,0.5) .. controls (0,0.75) and (0.5, 0.75) ..  +(0.5,1);
		\draw  +(0.5,0.5) .. controls (0.5,0.75) and (0, 0.75) ..  +(0,1);
	  	 \fdot{.25,.5};  \fdot{.5,.25};
	 } 
\ - \ 
	\tikz[very thick,xscale=2,yscale=1.5,baseline={([yshift=-.5ex]current bounding box.center)}]{
		\draw  +(0,0) .. controls (0,0.25) and (0.5, 0.25) ..  +(0.5,0.5) node [at end,fill=black,circle,inner sep=2pt]{};
		\draw  +(0.5,0) .. controls (0.5,0.25) and (0, 0.25) ..  +(0,0.5) node [at end,fill=black,circle,inner sep=2pt]{} ;  \node at (-.075,.6) {\small$p$};
		\draw  +(0,0.5) .. controls (0,0.75) and (0.5, 0.75) ..  +(0.5,1);
		\draw  +(0.5,0.5) .. controls (0.5,0.75) and (0, 0.75) ..  +(0,1);
	  	 \fdot{.25,.5};  \fdot{.5,.25};
	 } 
\ = \ -
	\tikz[very thick,xscale=2,yscale=2,baseline={([yshift=-.75ex]current bounding box.center)}]{
		\draw  +(0,0) .. controls (0,0.25) and (0.5, 0.25) ..  +(0.5,0.5) node [near start,fill=black,circle,inner sep=2pt]{};   \node at (0,.225) {\small$p$};
		\draw  +(0.5,0) .. controls (0.5,0.25) and (0, 0.25) ..  +(0,0.5);
		  \fdot{.5,.25};   \fdot{.25,.1};
	}
\ + \
	\tikz[very thick,xscale=2,yscale=1.5,baseline={([yshift=-.5ex]current bounding box.center)}]{
		\draw  +(0,0) .. controls (0,0.25) and (0.5, 0.25) ..  +(0.5,0.5);
		\draw  +(0.5,0) .. controls (0.5,0.25) and (0, 0.25) ..  +(0,0.5) node [at end,fill=black,circle,inner sep=2pt]{};  \node at (-.225,.6) {\small$p{+}1$};
		\draw  +(0,0.5) .. controls (0,0.75) and (0.5, 0.75) ..  +(0.5,1);
		\draw  +(0.5,0.5) .. controls (0.5,0.75) and (0, 0.75) ..  +(0,1);
	  	 \fdot{.25,.5};  \fdot{.5,.25};
	 } 
\ - \ 
	\tikz[very thick,xscale=2,yscale=1.5,baseline={([yshift=-.5ex]current bounding box.center)}]{
		\draw  +(0,0) .. controls (0,0.25) and (0.5, 0.25) ..  +(0.5,0.5) node [at end,fill=black,circle,inner sep=2pt]{};
		\draw  +(0.5,0) .. controls (0.5,0.25) and (0, 0.25) ..  +(0,0.5) node [at end,fill=black,circle,inner sep=2pt]{} ;  \node at (-.075,.6) {\small$p$};
		\draw  +(0,0.5) .. controls (0,0.75) and (0.5, 0.75) ..  +(0.5,1);
		\draw  +(0.5,0.5) .. controls (0.5,0.75) and (0, 0.75) ..  +(0,1);
	  	 \fdot{.25,.5};  \fdot{.5,.25};
	 } 
\intertext{and the same for}
	\tikz[very thick,xscale=2,yscale=-1.5,baseline={([yshift=-.5ex]current bounding box.center)}]{
		\draw  +(0,0) .. controls (0,0.25) and (0.5, 0.25) ..  +(0.5,0.5);
		\draw  +(0.5,0) .. controls (0.5,0.25) and (0, 0.25) ..  +(0,0.5) node [at end,fill=black,circle,inner sep=2pt]{};  \node at (-.075,.4) {\small$p$};
		\draw  +(0,0.5) .. controls (0,0.75) and (0.5, 0.75) ..  +(0.5,1);
		\draw  +(0.5,0.5) .. controls (0.5,0.75) and (0, 0.75) ..  +(0,1);
	  	 \fdot{.25,.5};  \fdot{.25,.1};
	 } 
\ &= \ -
	\tikz[very thick,xscale=2,yscale=-2,baseline={([yshift=-1ex]current bounding box.center)}]{
		\draw  +(0,0) .. controls (0,0.25) and (0.5, 0.25) ..  +(0.5,0.5) node [near start,fill=black,circle,inner sep=2pt]{};   \node at (0,.275) {\small$p$};
		\draw  +(0.5,0) .. controls (0.5,0.25) and (0, 0.25) ..  +(0,0.5);
		  \fdot{.5,.25};   \fdot{.25,.1};
	}
\ + \
	\tikz[very thick,xscale=2,yscale=-1.5,baseline={([yshift=-.5ex]current bounding box.center)}]{
		\draw  +(0,0) .. controls (0,0.25) and (0.5, 0.25) ..  +(0.5,0.5);
		\draw  +(0.5,0) .. controls (0.5,0.25) and (0, 0.25) ..  +(0,0.5) node [at end,fill=black,circle,inner sep=2pt]{};  \node at (-.225,.4) {\small$p{+}1$};
		\draw  +(0,0.5) .. controls (0,0.75) and (0.5, 0.75) ..  +(0.5,1);
		\draw  +(0.5,0.5) .. controls (0.5,0.75) and (0, 0.75) ..  +(0,1);
	  	 \fdot{.25,.5};  \fdot{.5,.25};
	 } 
\ - \ 
	\tikz[very thick,xscale=2,yscale=-1.5,baseline={([yshift=-.5ex]current bounding box.center)}]{
		\draw  +(0,0) .. controls (0,0.25) and (0.5, 0.25) ..  +(0.5,0.5) node [at end,fill=black,circle,inner sep=2pt]{};
		\draw  +(0.5,0) .. controls (0.5,0.25) and (0, 0.25) ..  +(0,0.5) node [at end,fill=black,circle,inner sep=2pt]{} ;  \node at (-.075,.4) {\small$p$};
		\draw  +(0,0.5) .. controls (0,0.75) and (0.5, 0.75) ..  +(0.5,1);
		\draw  +(0.5,0.5) .. controls (0.5,0.75) and (0, 0.75) ..  +(0,1);
	  	 \fdot{.25,.5};  \fdot{.5,.25};
	 }
\end{align*}
Up to a sign corresponding with the parity of the projection morphism and of the floating dot, the last two terms in both equations coincide. Therefore, it only remains to show the first term in both equations are projected onto the same element (with a sign again). But this is a direct consequence of Lemma~\ref{lem:wdtranslation}, together with the relation $T_1 \tilde \omega_1 T_1 \tilde \omega_1 = - \tilde \omega_1 T_1 \tilde \omega_1 T_1$ as in the proof of Proposition~\ref{prop:isoAnAnp},
\[
	\tikz[very thick,xscale=2,yscale=2,baseline={([yshift=-.5ex]current bounding box.center)}]{
		\draw  +(0,0) .. controls (0,0.25) and (0.5, 0.25) ..  +(0.5,0.5) node [near start,fill=black,circle,inner sep=2pt]{};   \node at (0,.225) {\small$p$};
		\draw  +(0.5,0) .. controls (0.5,0.25) and (0, 0.25) ..  +(0,0.5);
		  \fdot{.5,.25};   \fdot{.25,.1};
	}
\ = \ -
	\tikz[very thick,xscale=2,yscale=1.5,baseline={([yshift=-.5ex]current bounding box.center)}]{
		\draw  +(0,0) .. controls (0,0.25) and (0.5, 0.25) ..  +(0.5,0.5) node [near start,fill=black,circle,inner sep=2pt]{};  \node at (0,.225) {\small$p$};
		\draw  +(0.5,0) .. controls (0.5,0.25) and (0, 0.25) ..  +(0,0.5);
		\draw  +(0,0.5) .. controls (0,0.75) and (0.5, 0.75) ..  +(0.5,1);
		\draw  +(0.5,0.5) .. controls (0.5,0.75) and (0, 0.75) ..  +(0,1);
	  	 \fdot{.25,.5};  \fdot{.25,.1};
	 } 
\ = \ 
	\tikz[very thick,xscale=2,yscale=1.5,baseline={([yshift=-1ex]current bounding box.center)}]{
		\draw  +(0,0) .. controls (0,0.25) and (0.5, 0.25) ..  +(0.5,0.5) node [near start,fill=black,circle,inner sep=2pt]{};  \node at (0,.225) {\small$p$};
		\draw  +(0.5,0) .. controls (0.5,0.25) and (0, 0.25) ..  +(0,0.5);
		\draw  +(0,0.5) .. controls (0,0.75) and (0.5, 0.75) ..  +(0.5,1);
		\draw  +(0.5,0.5) .. controls (0.5,0.75) and (0, 0.75) ..  +(0,1);
	  	 \fdot{.25,.5};  \fdot{.25,.9};
	 } 
\]
 and the observation made above about sliding of dots.
\end{proof}

The second claim of the lemma above means the projection of an element in $A_{n+1}(m)$ onto the summands $A_n(m) \otimes \theta_{n+1} ^{x_1}$ or $\theta_{n+1}^{x_1}\otimes A_n(m)$ when viewed as left or right $A_n(m)$-module gives the same result.

\begin{proof}[Proof of Theorem~\ref{thm:sesAnm}]
We define the morphism $s : q^{-2} A_{n}(m) \otimes_{n-1} A_{n}(m) \rightarrow A_{n+1}(m) $ as the one that adds a crossing to the right, that is
$s(x \otimes_{n-1} y) = xT_n y$, or graphically
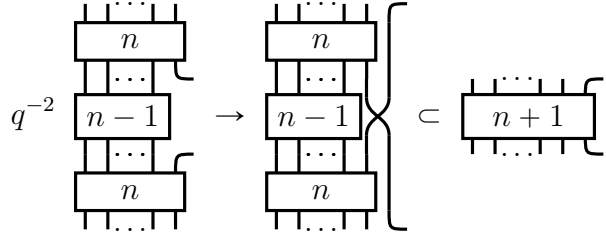
\[
	q^{-2} \ \tikz[xscale = 1.2, very thick,baseline={([yshift=-.5ex]current bounding box.center)}]{
		\node [anchor = west] at (0,1) {\boxAnm};
		\node [anchor = west] at (0,0) {
			\tikz[anchor = center, xscale = 1.2, very thick]{
	    		\draw (0,-.5)-- (0,.5);
	 		\draw (.25,-.5)-- (.25,.5);
	 		\draw (.75,-.5)-- (.75,.5);
			\node (rect) at (-.125,0) [anchor = west, draw,fill=white,minimum width=1.25cm,minimum height=.5cm] {$n-1$};
			\node at (.525,.5) {\small $\dots$}; \node at (.525,-.5) {\small $\dots$};
			}
		};
		\node [anchor = west] at (0,-1) {\boxmAn};
	}\rightarrow
	\tikz[xscale = 1.2, very thick,baseline={([yshift=-.5ex]current bounding box.center)}]{
		\node [anchor = west] at (-.01,1) {\boxAn};
		\node [anchor = west] at (-.01,0) {
			\tikz[anchor = center, xscale = 1.2, very thick]{
	    		\draw (0,-.5)-- (0,.5);
	 		\draw (.25,-.5)-- (.25,.5);
	 		\draw (.75,-.5)-- (.75,.5);
			\node (rect) at (-.125,0) [anchor = west, draw,fill=white,minimum width=1.25cm,minimum height=.5cm] {$n-1$};
			\node at (.525,.5) {\small $\dots$}; \node at (.525,-.5) {\small $\dots$};
			}
		};
	 	\draw +(1.5,1.25) .. controls (1.5,1.5) ..  + (1.7,1.5);
		\draw (1.5,1.25) --(1.5,.25); \draw (1.5,-1.25) --(1.5,-.25);	 \draw (1.25,.5) --(1.25,.25); \draw (1.25,-.5) --(1.25,-.25);
		\draw (1.25,-.25) .. controls (1.25, 0) and (1.5, 0) .. (1.5, .25); \draw (1.5,-.25) .. controls (1.5, 0) and (1.25, 0) .. (1.25, .25);
	 	\draw +(1.5,-1.25) .. controls (1.5,-1.5) ..  + (1.7,-1.5);
		\node [anchor = west] at (-.01,-1) {\boxAn};
	} \subset 
	\tikz[xscale = 1.2, very thick,baseline={([yshift=-.5ex]current bounding box.center)}]{
		\node [anchor = west] at (0,1) {\boxnAnpn};
	}
\]
This is clearly a well-defined morphism of bimodules since $T_n$ commutes with every element of $A_{n-1}(m)$. 
We define the projection morphisms as the projection on the left (equivalently right) $A_n(m)$-submodules of $A_{n+1}(m)$ given by the summands $A_n(m)\otimes \bZ[x_{n+1}]$ and $ A_n(m) \otimes \theta_{n+1}^{x_1}$ from the decomposition in Lemma~\ref{lem:newbasisdecomp} (equivalently Lemma~\ref{lem:samedecomp}).
By the second assertion in  Lemma~\ref{lem:samedecomp}, this yields a well-defined bimodule morphism.

What is left to prove is that we have a short exact sequence of $A_n(m)$-left modules (in fact, it is enough to prove we have a `short exact sequence' of sets since we already know by the above that our maps respect the bimodule structure). Applying Lemma~\ref{lem:newbasisdecomp} we have
\begin{align*}
 A_{n}(m) \otimes_{n-1} A_{n}(m) &\cong \bigoplus_{a=1}^n A_n(m)  \otimes T_{n-1}\dotsm T_a (\bZ[x_a] \oplus \theta_a ^{x_1}),
\end{align*}
and
\begin{align*}
s( A_n(m)  \otimes  T_{n-1}\dotsm T_a (\bZ[x_a] \oplus \theta_a ^{x_1})) &= 
 A_n(m)  \otimes  T_nT_{n-1}\dotsm T_a (\bZ[x_a] \oplus \theta_a ^{x_1}).
\end{align*}
Therefore, each direct summand of $A_{n}(m) \otimes_{n-1} A_{n}(m)$ is send to a different direct summand of $A_{n+1}(m)$, so that $s$ is an injection. Moreover, the only remaining summands in $A_{n+1}(m)$ are $A_n(m) \otimes \bZ[x_{n+1}]$ and $A_n(m) \otimes \theta_{n+1}^{x_1}$, which are exactly   the right part of the sequence. We conclude that we have constructed a short exact sequence.
\end{proof}

By the decomposition in Lemma~\ref{lem:samedecomp}, we know that $A_{n+1}(m)$ is bifree as a $(A_{n+1}(m)$,$A_{n}(m))$-bimodule (i.e. free as left module and free as right module, but not necessarily free as a bimodule). Therefore, we have:

\begin{prop}\label{prop:projtoproj}
The functors $\F_n$ and $\E_{n}$ are exact and send projectives to projectives.
\end{prop}

Applying recursively the short exact sequence from Theorem~\ref{thm:sesAnm} also gives an interesting observation about the graded superdimension of $A_n(m)$ (i.e. graded rank of $A_n(m)$ seen as $\bZ$-module where we specialize the parity to $-1$).

\begin{cor}
We have 
\[
\sdim A_n(m) = \left(F^n m_0, F^n m_0\right)_{\lambda q^{m}},
\]
where $m_0$ is the highest weight vector of the $U_q(\slt)$-Verma module of highest weight $\lambda q^{m}$, and $\left(-,-\right)_{\lambda q^{m}}$ is the universal Shapovalov form~(as in~\cite[\S~2.3]{naissevaz1}).
\end{cor}

\subsection{The categorification theorem}
We now consider $A_n(m)$ with coefficients in $\bQ$.
We write $\bZ_\pi = \bZ[\pi]/(\pi^2-1)$ and $\bQ_\pi = \bZ_\pi \otimes \bQ$. Let $\bZ \pp{q, \lambda}$ be the ring of formal Laurent series in the variables $q$ and $\lambda$, given by the order $0 \prec q \prec \lambda$, as explained in \cite[\S5.1]{naissevaz1} (see \cite{laurent} for a general discussion about rings of formal Laurent series in several variables).
It is given by formal series in variables $q^{\pm 1}$ and $\lambda^{\pm 1}$ for which the sum  of the bidegrees of each term is contained in a cone compatible with the additive order $0 \prec q \prec \lambda$ on $\bZ q \oplus \bZ \lambda$. 
This ensures that we can multiply two formal Laurent series using only finite sums to determine each coefficient. Moreover, there is an inclusion $\bZ(q,\lambda) \hookrightarrow \bZ\pp{q,\lambda}$. For example, we get $\frac{1}{q-q^{-1}} \mapsto -q(1+q^2+\dots)$.

\smallskip

Let $A_n(m)\lfmods \subset A_n(m) \smod$ be the full subcategory of $A_n(m)$-supermodules which are of cone bounded, locally finite dimension over $\bQ$. By \emph{cone bounded} we mean that their graded dimensions, which are elements of  $\bZ\llbracket q, q^{-1}, \lambda, \lambda^{-1} \rrbracket$, are contained in a cone compatible with the additive order $0 \prec q \prec \lambda$ on $\bZ q \oplus \bZ \lambda$ and so, in particular, are elements of $\bZ \pp{q, \lambda}$. 

\smallskip

Recall from~\S\ref{ssec:idempots} that $A_n(m)$ admits a unique indecomposable, graded projective module $P_{(n)}$, up to shift. This projective module is of locally finite dimension contained in a cone compatible with $\prec$. Therefore, following the results from~\cite[\S 5]{naissevaz1}, $A_n(m)\lfmods$ is cone complete and possesses the local Jordan--H\"older property (i.e. it admits all cone bounded, locally finite direct sums and any object admits an essentially unique (infinite) filtration with simple quotients). 
Thus, its (topological) Grothendieck group $\boldsymbol G_0(A_n(m)\lfmods)$ (i.e. the quotient of the usual Grothendieck group by the relations obtained through the infinite filtrations) is a $\bZ_\pi \pp{q, \lambda}$-module freely generated by the unique simple module $S = R(m)/R^{S_n}_+(m)$, where
$R^{S_n}_+(m)\subset R^{S_n}(m)$ is the maximal ideal. 
This simple module admits a projective cover given by $P_{(n)}$. Taking a projective resolution of $S$, it is not hard to see  the Grothendieck group is also generated by the unique indecomposable projective module.
\begin{thm}
  The functors $\F = \bigoplus_{n \ge 0}  \F_n$ and $\E = \bigoplus_{n \ge 0} \E_n$ induce an action of quantum $\slt$ on the topological Grothendieck group $\boldsymbol G_0(A(m)\lfmods)$. 
  With this action there is an isomorphism
\[
\boldsymbol  G_0(A(m)\lfmods) \otimes_{\bZ_\pi} \bQ_\pi/{(\pi + 1)} \cong M(\lambda q^{m})
\]
of $U_q(\slt)$-modules.
  This isomorphism sends classes of projective indecomposables to the canonical basis elements and
  classes of simples to dual canonical elements (see~\cite[\S2]{naissevaz1}).
\end{thm}

\begin{proof}
By Corollary~\ref{cor:EF-rel} and Proposition~\ref{prop:projtoproj}, we know that $\F$ and $\E$ induce an action of $U_q(\slt)$ on $\boldsymbol G_0(A(m)\lfmods)$, which becomes a highest weight module induced by the action of $\F$ on the highest weight vector $[A_0(m)]$. The rest follows by direct comparison between the action of $U_q(\slt)$ on the canonical basis elements of $M(\lambda q^{m})$ (resp. dual canonical) and on the image of the projectives  $[P_{(n)}]$ (resp. the simples) in the Grothendieck group, as in~\cite[\S6]{naissevaz1}, using Proposition~\ref{prop:AnPnDecomp} and Lemma~\ref{lem:newbasisdecomp}.
\end{proof}

\begin{rem}\label{rem:psmodlfg}
Alternatively, the full subcategory $A_n(m)\prmods \subset A_n(m) \lfmods$ of cone bounded, locally finitely generated projective modules is cone complete, locally Krull-Schmidt (see~\cite[\S 5]{naissevaz1} for details about these notions). Therefore, the (topological) split Grothendieck group $\boldsymbol K_0(A_n(m)\prmods )$ is freely generated over $\bZ_\pi \pp{q, \lambda}$ by  $P_{(n)}$. We conclude that $ A(m)\prmods$ also yields a categorification of the $U_q(\slt)$-module $M(\lambda q^{m})$. 
\end{rem}

\subsection{Recovering the irreducible finite-dimensional $U_q(\slt)$-module $V(N)$}\label{sec:differential} For this section, we can again assume $A_n(m)$ to be defined over $\bZ$, except when specified otherwise. We also upgrade the parity grading into a $\bZ$-grading by setting 
\begin{align*}
\deg_h(x_i) &= 0, & \deg_h(T_i) &= 0, & \deg_h(\omega_i^a) = 1.
\end{align*}
We can do this because all relations in $A_n(m)$ preserve the number of floating dots.
We call this degree the \emph{homological degree}, and we write $[1]$ for a shift up by one in it. Note that the parity of an element $x$ is given by $\deg_h(x) \mod 2$.

\subsubsection{Turning $A_n(m)$ into a dg-algebra}

As already explained in~\cite[\S8]{naissevaz1}, for each $N \in \bN$ such that $m+N \ge 0$, we can turn $A_n(m)$ into a dg-algebra by defining the differential $d_N$ via 
\begin{align*}
d_N(x_i) &= 0,& d_N(T_i) &= 0, & d_N(\omega_i^a) = (-1)^{N+a-i} \mathcal{h}_{N+a-i}(\und x_i),
\end{align*}
where we recall that $\mathcal{h}_{k}(\und x_i)$ is the $k$-th complete homogeneous symmetric polynomial in variables $(x_1, \dots, x_i)$, together with the graded Leibniz rule
\[
d_N(xy) = d_N(x) y + (-1)^{\deg_h(x)}xd_N(y).
\]
In particular, we get $d_N(\tilde \omega_1) = d_N(\omega_1^{m+1}) = (-1)^{m+N} x_1^{m+N}$.
Using the isomorphism $A_n \cong (0,\infty)\text{-}\BA_n/J_{\infty,\infty}$, we can express the differential $d_N$ using bubbles, as a 'blow-up':
\[
 \tikz[very thick,scale=1.5,baseline={([yshift=.8ex]current bounding box.center)}]{
          \draw (-.5,-.5) node[below] {\small $1$}  -- (-.5,.5); 
          \draw (.5,-.5) node[below] {\phantom{\small$1$}\small $i$\phantom{\small$1$}}  -- (.5,.5) ;
	   \fdot[a]{1,0};
          \draw (1.5,-.5) node[below] {\small $i+1$}  -- (1.5,.5) ;
          \draw (2.5,-.5) node[below] {\phantom{\small$1$}\small $n$\phantom{\small$1$}}  -- (2.5,.5);
          \node at (2,0){$\cdots$};
          \node at (0,0){$\cdots$};
   }
\quad \xmapsto{\  d_N \ }\quad
 \tikz[very thick,scale=1.5,baseline={([yshift=.8ex]current bounding box.center)}]{
          \draw (-.5,-.5) node[below] {\small $1$}  -- (-.5,.5); 
          \draw (.5,-.5) node[below] {\phantom{\small$1$}\small $i$\phantom{\small$1$}}  -- (.5,.5) ;
	   \node[draw,circle,minimum size=.75cm,inner sep=1pt] at (1,0) {\tiny $Y_{N+a-i}$};
          \draw (1.5,-.5) node[below] {\small $i+1$}  -- (1.5,.5) ;
          \draw (2.5,-.5) node[below] {\phantom{\small$1$}\small $n$\phantom{\small$1$}}  -- (2.5,.5);
          \node at (2,0){$\cdots$};
          \node at (0,0){$\cdots$};
        }
\]
which shows $d_N$ is well-defined since $\omega_i^a$ interacts like $Y_{N+a-i}$ with crossings~\eqref{eq:crossbubbleY} and bubble slides~\eqref{eq:bubblerelationY}, as explained in \S\ref{ssec:bubbledAn}.

\begin{lem}\label{lem:hominzero}
The homology $H(A_n(m), d_N)$ is concentrated in homological degree $0$.
\end{lem}

\begin{proof}
By Proposition~\ref{prop:invert-S0i}, we can construct a basis of $A_n(m)$ where all floating dots involved are concentrated in the right part of the diagrams and have labels between $m+1$ and $m+n$. Therefore, as $\bZ$-modules, we have 
\[
A_n(m) \cong \nh_n \otimes \bV^\bullet(\omega_n^{m+1}, \omega_n^{m+2},\dotsc, \omega_n^{m+n}).
\]
Also we have $d_N(\omega_n^{m+\ell}) \in Z_s(A_n(m)) \cong R^{S_n}(m)$.  Hence we get a decomposition of $\bZ$-dg-modules
\[
(A_n(m), d_N) \cong  (R^{S_n}(m), d_N) \otimes (U_n, 0),
\]
which means 
\[
H(A_n(m), d_N) \cong H (R^{S_n}(m), d_N) \otimes U_n,
\]
 as $\bZ$-modules. Moreover, it is clear that $(R^{S_n}(m), d_N)$ has trivial homology groups in non-zero homological degree and this concludes the proof.
\end{proof}

Having in mind the basis~\eqref{eq:tightbasis} and Lemma~\ref{lem:hominzero}, it is not hard to see that $H(A_n(m), d_N)$ is isomorphic to the cyclotomic quotient of the nilHecke algebra $\nh_n^{m+N} = \nh_n/{(x_1^{m+N})}$. 

\begin{prop}
There is an isomorphism
\[
H(A_n(m), d_N) \cong \nh_n^{m+N}.
\]
Moreover, $(A_n(m), d_N)$ is acyclic for $n > m+N$.
\end{prop}
The second statement follows from  
$d_N(\omega^{m+1}_{m+N+1})=1$. 
Note also that $\deg_{q,\lambda}(d_N) = (2N, -2)$ and $H(A(m),d_N) \cong \nh^{m+N} = \bigoplus_{n \ge 0} \nh_n^{m+N}$.

\medskip

We will now prove that the short exact sequence from Theorem~\ref{thm:sesAnm} can be enhanced in terms of dg-bimodules for $d_N$. First observe that all arguments in \S\ref{ssec:resind-cataction} hold for the homological grading, so that we obtain a short exact sequence
\begin{align*}
0 \rightarrow q^{-2}A_{n}(m) \otimes_{n-1}& A_{n}(m) \rightarrow A_{n+1}(m) \\
 &\rightarrow \left(A_{n}(m) \otimes \bZ[\xi]\right) \oplus  \left(q^{2m-4n}\lambda^2 A_{n}(m)[1] \otimes \bZ [\xi]\right) \rightarrow 0,
\end{align*}
where the parity $\Pi$ in Theorem~\ref{thm:sesAnm} becomes the homological shift $[1]$. The differential on $A_n(m)$ induce a differential on $A_n(m) \otimes_{n-1} A_n(m)$, also written $d_N$,  with the Koszul sign rule
\[
d_N(x \otimes_{n-1} y) = d_N(x) \otimes y + (-1)^{\deg_h(x)} x \otimes d_N(y).
\]
Hence, both the left and the middle part of the short exact sequence can be enhanced into $\left((A_n(m), d_N), (A_n(m), d_N)\right)$-dg-bimodules (we will write this $(A_n(m), d_N)$-bimodules for the sake of compactness). Then, it only remains to define a differential on the right part, compatible with $d_N$ and the projection morphism. This can be achieved by taking a lift of $1 \otimes \xi^j \in   \left(q^{2m-4n}\lambda^2 A_{n}(m)[1] \otimes \bZ [\xi]\right) $, for example $T_{n}\dotsm T_1 x_1^j \tilde \omega_1 T_1 \dotsm T_n$, then applying $d_N$ on it, and finally computing its projection in $ \left(A_{n}(m) \otimes \bZ[\xi]\right) $. Since $d_N(T_{n}\dotsm T_1 x_1^j \tilde \omega_1 T_1 \dotsm T_n) =(-1)^{N+m} T_{n}\dotsm T_1 x_1^{j+N+m}  T_1 \dotsm T_n$, we first compute the projection of $T_{n}\dotsm T_1 x_1^j T_1 \dotsm T_n$ for any $j \ge 0$.

\begin{lem}\label{lem:projvalues}
The projection 
\[
\tikz[xscale = 1.2, very thick,baseline={([yshift=-.5ex]current bounding box.center)}]{
		\node [anchor = west] at (0,1) {\boxnAnpn};
	}\rightarrow
 \bigoplus_{p \in \bN}
	\tikz[xscale = 1.2, very thick,baseline={([yshift=-.5ex]current bounding box.center)}]{
		\node [anchor = west] at (0,1) {\boxAn};
	 	\draw (1.5,.75)-- (1.5,1.25) node [midway,fill=black,circle,inner sep=2pt]{};  \node at (1.75,1.25) {\small$p$};
	 	\draw +(1.5,1.25) .. controls (1.5,1.5) ..  + (1.7,1.5);
	 	\draw +(1.5,.75) .. controls (1.5,.5) ..  + (1.7,.5);
	}
\]
from Theorem~\ref{thm:sesAnm} sends 
\[
\tikz[xscale = 1.2, very thick,baseline={([yshift=-.5ex]current bounding box.center)}]{
		\draw + (.25,2.25) .. controls (.25,1.5) and (.75,1.5) ..  +  (.75,1); \draw + (.25,-.25) .. controls (.25,.5) and (.75,.5) ..  +  (.75,1); 
		\draw + (.5,2.25) .. controls (.5,1.5) and (1,1.5) ..  +  (1,1); \draw + (.5,-.25) .. controls (.5,.5) and (1,.5) ..  +  (1,1); 
		\draw + (1,2.25) .. controls (1,1.5) and (1.5,1.5) ..  +  (1.5,1); \draw + (1,-.25) .. controls (1,.5) and (1.5,.5) ..  +  (1.5,1); 
		\draw + (1.25,2.25) .. controls (1.25,1.5) and (1.75,1.5) ..  +  (1.75,1); \draw + (1.25,-.25) .. controls (1.25,.5) and (1.75,.5) ..  +  (1.75,1); 
		\draw + (1.75,2.25) .. controls (1.75,1.5) and (.25,1.5) ..  +  (.25,1);
		 \draw + (1.75,-.25) .. controls (1.75,.5) and (.25,.5) ..  +  (.25,1)	node [pos=1,fill=black,circle,inner sep=2pt]{};   \node at (.1,1.15) {\small $j$};
		 \node at (1.3,1) {\small $\dots$};  \node at (.8,2.15) {\small $\dots$};  \node at (.8,-.15) {\small $\dots$};
	 }
\mapsto
(-1)^{n} \sum_{p = n }^{j-n} \mathcal{h}_{p - n}(\und x_n, \xi) \mathcal{h}_{j-n-p}(\und x_n),
\]
where we identify $\xi$ with a dot on the rightmost strand.
\end{lem}

\begin{proof}
We can suppose $n \ge 1$. By~(\ref{eq:crossings}) and~(\ref{eq:relnh2}) we obtain 
\begin{align*}
\tikz[xscale = 1.2, very thick,baseline={([yshift=-.5ex]current bounding box.center)}]{
		\draw + (.25,2.25) .. controls (.25,1.5) and (.75,1.5) ..  +  (.75,1); \draw + (.25,-.25) .. controls (.25,.5) and (.75,.5) ..  +  (.75,1); 
		\draw + (.5,2.25) .. controls (.5,1.5) and (1,1.5) ..  +  (1,1); \draw + (.5,-.25) .. controls (.5,.5) and (1,.5) ..  +  (1,1); 
		\draw + (1,2.25) .. controls (1,1.5) and (1.5,1.5) ..  +  (1.5,1); \draw + (1,-.25) .. controls (1,.5) and (1.5,.5) ..  +  (1.5,1); 
		\draw + (1.25,2.25) .. controls (1.25,1.5) and (1.75,1.5) ..  +  (1.75,1); \draw + (1.25,-.25) .. controls (1.25,.5) and (1.75,.5) ..  +  (1.75,1); 
		\draw + (1.75,2.25) .. controls (1.75,1.5) and (.25,1.5) ..  +  (.25,1) node [pos=.15,fill=black,circle,inner sep=2pt]{};  \node at (1.85,2) {$i$};
		 \draw + (1.75,-.25) .. controls (1.75,.5) and (.25,.5) ..  +  (.25,1)	node [pos=1,fill=black,circle,inner sep=2pt]{};   \node at (.1,1.15) {\small $j$};
		 \node at (1.3,1) {\small $\dots$};  \node at (.8,2.15) {\small $\dots$};  \node at (.8,-.15) {\small $\dots$};
	 }
 \quad = \quad
\tikz[xscale = 1.2, very thick,baseline={([yshift=-.5ex]current bounding box.center)}]{
		\draw + (.25,2.25) .. controls (.25,1.5) and (.75,1.5) ..  +  (.75,1); \draw + (.25,-.25) .. controls (.25,.5) and (.75,.5) ..  +  (.75,1); 
		\draw + (.5,2.25) .. controls (.5,1.5) and (1,1.5) ..  +  (1,1); \draw + (.5,-.25) .. controls (.5,.5) and (1,.5) ..  +  (1,1); 
		\draw + (1,2.25) .. controls (1,1.5) and (1.5,1.5) ..  +  (1.5,1); \draw + (1,-.25) .. controls (1,.5) and (1.5,.5) ..  +  (1.5,1); 
		\draw + (1.25,2.25) .. controls (1.25,1.5) and (1.75,1.5) ..  +  (1.75,1); \draw + (1.25,-.25) .. controls (1.25,.5) and (1.75,.5) ..  +  (1.75,1); 
		\draw + (1.75,2.25) .. controls (1.75,1.5) and (.25,1.5) ..  +  (.25,1) node [pos=.15,fill=black,circle,inner sep=2pt]{};  \node at (2.15,2) {\small $i+1$};
		 \draw + (1.75,-.25) .. controls (1.75,.5) and (.25,.5) ..  +  (.25,1)	node [pos=1,fill=black,circle,inner sep=2pt]{}; \node at (-.25,1.15) {$j-1$};
		 \node at (1.3,1) {\small $\dots$};  \node at (.8,2.15) {\small $\dots$};  \node at (.8,-.15) {\small $\dots$};
	 }
\quad + \sum_{\substack{r+s \\ =n-1}} \quad
\tikz[xscale = 1.2, very thick,baseline={([yshift=-.5ex]current bounding box.center)}]{
		\draw + (.25,2.25) .. controls (.25,2) and (.5,2) ..  +  (.5,1.5); \draw + (.25,-.25) .. controls (.25,1) and (.5,1) ..  +  (.5,1.5); 
		\draw + (.75,2.25) .. controls (.75,2) and (1,2) ..  +  (1.,1.5); \draw + (.75,-.25) .. controls (.75,1) and (1,1) ..  +  (1,1.5); 
		\draw + (1.25,2.25) .. controls (1.25,1) and (1,1) ..  +  (1,.5); \draw + (1.25,-.25) .. controls (1.25,0) and (1,0) ..  +  (1,.5); 
		\draw + (1.75,2.25) .. controls (1.75,1) and (1.5,1) ..  +  (1.5,.5); \draw + (1.75,-.25) .. controls (1.75,0) and (1.5,0) ..  +  (1.5,.5); 
		\draw + (2,-.25) .. controls (2,1) and (.25,1) ..  +  (.25,1.5); \draw + (1,2.25) .. controls (1,2) and (.25,2) ..  +  (.25,1.5)	
		node [pos=1,fill=black,circle,inner sep=2pt]{}; \node at (-.25,1.65) {$j-1$};
		\draw + (1,-.25) .. controls (1,0) and (2,0) ..  +  (2,2.25) node [pos=.9,fill=black,circle,inner sep=2pt]{};  \node at (2.15,1.85) {\small $i$};
 		\node at (1.5,2) {\small $\overset{s}{\dots}$}; \node at (.5,0) {\small $\overset{r}{\dots}$}; 
} 
\end{align*}
for all $i \ge 0$ and $j \ge 1$, where the $\overset{r}{\dots}$ means we take $r$ such strands.
We have 
\begin{align*}
\tikz[xscale = 1.2, very thick,baseline={([yshift=-.5ex]current bounding box.center)}]{
		\draw + (.25,2.25) .. controls (.25,1.5) and (.75,1.5) ..  +  (.75,1); \draw + (.25,-.25) .. controls (.25,.5) and (.75,.5) ..  +  (.75,1); 
		\draw + (.5,2.25) .. controls (.5,1.5) and (1,1.5) ..  +  (1,1); \draw + (.5,-.25) .. controls (.5,.5) and (1,.5) ..  +  (1,1); 
		\draw + (1,2.25) .. controls (1,1.5) and (1.5,1.5) ..  +  (1.5,1); \draw + (1,-.25) .. controls (1,.5) and (1.5,.5) ..  +  (1.5,1); 
		\draw + (1.25,2.25) .. controls (1.25,1.5) and (1.75,1.5) ..  +  (1.75,1); \draw + (1.25,-.25) .. controls (1.25,.5) and (1.75,.5) ..  +  (1.75,1); 
		\draw + (1.75,2.25) .. controls (1.75,1.5) and (.25,1.5) ..  +  (.25,1);
		 \draw + (1.75,-.25) .. controls (1.75,.5) and (.25,.5) ..  +  (.25,1)	node [pos=1,fill=black,circle,inner sep=2pt]{};   \node at (.1,1.15) {\small $j$};
		 \node at (1.3,1) {\small $\dots$};  \node at (.8,2.15) {\small $\dots$};  \node at (.8,-.15) {\small $\dots$};
	 }
 \quad = \quad
\sum_{\substack{r+s \\ =n-1}} \quad \sum_{i=0}^{j-1}  \quad
\tikz[xscale = 1.2, very thick,baseline={([yshift=-.5ex]current bounding box.center)}]{
		\draw + (.25,2.25) .. controls (.25,2) and (.5,2) ..  +  (.5,1.5); \draw + (.25,-.25) .. controls (.25,1) and (.5,1) ..  +  (.5,1.5); 
		\draw + (.75,2.25) .. controls (.75,2) and (1,2) ..  +  (1.,1.5); \draw + (.75,-.25) .. controls (.75,1) and (1,1) ..  +  (1,1.5); 
		\draw + (1.25,2.25) .. controls (1.25,1) and (1,1) ..  +  (1,.5); \draw + (1.25,-.25) .. controls (1.25,0) and (1,0) ..  +  (1,.5); 
		\draw + (1.75,2.25) .. controls (1.75,1) and (1.5,1) ..  +  (1.5,.5); \draw + (1.75,-.25) .. controls (1.75,0) and (1.5,0) ..  +  (1.5,.5); 
		\draw + (2,-.25) .. controls (2,1) and (.25,1) ..  +  (.25,1.5); \draw + (1,2.25) .. controls (1,2) and (.25,2) ..  +  (.25,1.5)	
		node [pos=1,fill=black,circle,inner sep=2pt]{}; \node at (-.5,1.65) {$j-i-1$};
		\draw + (1,-.25) .. controls (1,0) and (2,0) ..  +  (2,2.25) node [pos=.9,fill=black,circle,inner sep=2pt]{};  \node at (2.15,1.85) {\small $i$};
 		\node at (1.5,2) {\small $\overset{s}{\dots}$}; \node at (.5,0) {\small $\overset{r}{\dots}$}; 
} 
\end{align*}
By the nilHecke relation~\eqref{eq:relnh2}, we obtain 
\begin{equation}\label{eq:relnhextended}
	\tikz[very thick,scale=2,baseline={([yshift=-.5ex]current bounding box.center)}]{
		\draw  +(0,0) .. controls (0,0.25) and (0.5, 0.25) ..  +(0.5,0.5) node [near end,fill=black,circle,inner sep=2pt]{}; \node at (.5,.25) {\small $i$};
		\draw  +(0.5,0) .. controls (0.5,0.25) and (0, 0.25) ..  +(0,0.5);
	}
	 \ =\ 
	\tikz[very thick,scale=2,baseline={([yshift=-.5ex]current bounding box.center)}]{
	    	\draw  +(0,0) .. controls (0,0.25) and (0.5, 0.25) ..  +(0.5,0.5) node [near start,fill=black,circle,inner sep=2pt]{}; \node at (0,.25) {\small $i$};
		\draw  +(0.5,0) .. controls (0.5,0.25) and (0, 0.25) ..  +(0,0.5);
	} 
	 \  - \sum_{\ell = 0}^{i-1} \ 
	\tikz[very thick,xscale=1.2,baseline={([yshift=-.5ex]current bounding box.center)}]{
	          \draw (.1,-.5)-- (.1,.5) node [midway,fill=black,circle,inner
	          sep=2pt]{};   \node at (.3,.15) {\small $\ell$};
	          \draw (.9,-.5)-- (.9,.5) node [midway,fill=black,circle,inner
	          sep=2pt]{};   \node at (1.65,.15) {\small $i-1-\ell$};
	}  
\end{equation}
so that we get in the image of the projection 
\begin{align*}
\tikz[xscale = 1.2, very thick,baseline={([yshift=-.5ex]current bounding box.center)}]{
		\draw + (.25,2.25) .. controls (.25,1.5) and (.75,1.5) ..  +  (.75,1); \draw + (.25,-.25) .. controls (.25,.5) and (.75,.5) ..  +  (.75,1); 
		\draw + (.5,2.25) .. controls (.5,1.5) and (1,1.5) ..  +  (1,1); \draw + (.5,-.25) .. controls (.5,.5) and (1,.5) ..  +  (1,1); 
		\draw + (1,2.25) .. controls (1,1.5) and (1.5,1.5) ..  +  (1.5,1); \draw + (1,-.25) .. controls (1,.5) and (1.5,.5) ..  +  (1.5,1); 
		\draw + (1.25,2.25) .. controls (1.25,1.5) and (1.75,1.5) ..  +  (1.75,1); \draw + (1.25,-.25) .. controls (1.25,.5) and (1.75,.5) ..  +  (1.75,1); 
		\draw + (1.75,2.25) .. controls (1.75,1.5) and (.25,1.5) ..  +  (.25,1);
		 \draw + (1.75,-.25) .. controls (1.75,.5) and (.25,.5) ..  +  (.25,1)	node [pos=1,fill=black,circle,inner sep=2pt]{};   \node at (.1,1.15) {\small $j$};
		 \node at (1.3,1) {\small $\dots$};  \node at (.8,2.15) {\small $\dots$};  \node at (.8,-.15) {\small $\dots$};
	 }
 \mapsto
-\sum_{\substack{r+s \\ =n-1}} \quad \sum_{i=1}^{j-1}\quad  \sum_{\ell=0}^{i-1} \quad
\tikz[xscale = 1.2, very thick,baseline={([yshift=-.5ex]current bounding box.center)}]{
		\draw + (.25,2.25) .. controls (.25,2) and (.5,2) ..  +  (.5,1.5); \draw + (.25,-.25) .. controls (.25,1) and (.5,1) ..  +  (.5,1.5); 
		\draw + (.75,2.25) .. controls (.75,2) and (1,2) ..  +  (1.,1.5); \draw + (.75,-.25) .. controls (.75,1) and (1,1) ..  +  (1,1.5); 
		\draw + (1.25,2.25) .. controls (1.25,1) and (1,1) ..  +  (1,.5); \draw + (1.25,-.25) .. controls (1.25,0) and (1,0) ..  +  (1,.5); 
		\draw + (1.75,2.25) .. controls (1.75,1) and (1.5,1) ..  +  (1.5,.5); \draw + (1.75,-.25) .. controls (1.75,0) and (1.5,0) ..  +  (1.5,.5); 
		\draw + (1,2.25) .. controls (1,2) and (.25,2) ..  +  (.25,1.5)	
		node [pos=1,fill=black,circle,inner sep=2pt]{}; \node at (-.5,1.65) {$j-i-1$};
		\draw + (1,-.25) .. controls (1,0) and (1.75,0) ..  +  (1.75,.5)	
		node [pos=1,fill=black,circle,inner sep=2pt]{}; \node at (1.85,.25) {$\ell$};
		\draw +  (.25,1.5) .. controls (.25,1) and (1.75,1) ..  + (1.75,.5);
		\draw (2,-.25) --  (2,2.25) node [pos=.9,fill=black,circle,inner sep=2pt]{};  \node at (2.75,1.85) {\small $i-1-\ell$};
 		\node at (1.5,2) {\small $\overset{s}{\dots}$}; \node at (.5,0) {\small $\overset{r}{\dots}$}; 
}
\end{align*}
By a triangular change of variables, i.e. $p = i-1-\ell$ and $q = j-i-1$, this can be rewritten as
\begin{equation} \label{eq:projdiff1}
- \sum_{p=0}^{j-2} \quad  \sum_{\substack{q+\ell \\=j-2-p}} \quad  \sum_{\substack{r+s \\ =n-1}} \quad
\tikz[xscale = 1.2, very thick,baseline={([yshift=-.5ex]current bounding box.center)}]{
		\draw + (.25,2.25) .. controls (.25,2) and (.5,2) ..  +  (.5,1.5); \draw + (.25,-.25) .. controls (.25,1) and (.5,1) ..  +  (.5,1.5); 
		\draw + (.75,2.25) .. controls (.75,2) and (1,2) ..  +  (1.,1.5); \draw + (.75,-.25) .. controls (.75,1) and (1,1) ..  +  (1,1.5); 
		\draw + (1.25,2.25) .. controls (1.25,1) and (1,1) ..  +  (1,.5); \draw + (1.25,-.25) .. controls (1.25,0) and (1,0) ..  +  (1,.5); 
		\draw + (1.75,2.25) .. controls (1.75,1) and (1.5,1) ..  +  (1.5,.5); \draw + (1.75,-.25) .. controls (1.75,0) and (1.5,0) ..  +  (1.5,.5); 
		\draw + (1,2.25) .. controls (1,2) and (.25,2) ..  +  (.25,1.5)	
		node [pos=1,fill=black,circle,inner sep=2pt]{}; \node at (.1,1.65) {$q$};
		\draw + (1,-.25) .. controls (1,0) and (1.75,0) ..  +  (1.75,.5)	
		node [pos=1,fill=black,circle,inner sep=2pt]{}; \node at (1.85,.25) {$\ell$};
		\draw +  (.25,1.5) .. controls (.25,1) and (1.75,1) ..  + (1.75,.5);
		\draw (2,-.25) --  (2,2.25) node [pos=.9,fill=black,circle,inner sep=2pt]{};  \node at (2.25,1.85) {\small $p$};
 		\node at (1.5,2) {\small $\overset{s}{\dots}$}; \node at (.5,0) {\small $\overset{r}{\dots}$}; 
}
\end{equation}
Now we claim that 
\begin{equation} \label{eq:pit}
\pi_t := - \sum_{\substack{q+\ell\\= t}} \quad \sum_{\substack{r+s \\= n-1}} \quad
\tikz[xscale = 1.2, very thick,baseline={([yshift=-.5ex]current bounding box.center)}]{
		\draw + (.25,2.25) .. controls (.25,2) and (.5,2) ..  +  (.5,1.5); \draw + (.25,-.25) .. controls (.25,1) and (.5,1) ..  +  (.5,1.5); 
		\draw + (.75,2.25) .. controls (.75,2) and (1,2) ..  +  (1.,1.5); \draw + (.75,-.25) .. controls (.75,1) and (1,1) ..  +  (1,1.5); 
		\draw + (1.25,2.25) .. controls (1.25,1) and (1,1) ..  +  (1,.5); \draw + (1.25,-.25) .. controls (1.25,0) and (1,0) ..  +  (1,.5); 
		\draw + (1.75,2.25) .. controls (1.75,1) and (1.5,1) ..  +  (1.5,.5); \draw + (1.75,-.25) .. controls (1.75,0) and (1.5,0) ..  +  (1.5,.5); 
		\draw + (1,2.25) .. controls (1,2) and (.25,2) ..  +  (.25,1.5)	
		node [pos=1,fill=black,circle,inner sep=2pt]{}; \node at (.1,1.65) {$q$};
		\draw + (1,-.25) .. controls (1,0) and (1.75,0) ..  +  (1.75,.5)	
		node [pos=1,fill=black,circle,inner sep=2pt]{}; \node at (1.85,.25) {$\ell$};
		\draw +  (.25,1.5) .. controls (.25,1) and (1.75,1) ..  + (1.75,.5);
 		\node at (1.5,2) {\small $\overset{s}{\dots}$}; \node at (.5,0) {\small $\overset{r}{\dots}$}; 
} \quad = (-1)^{n} \sum_{q=n-1}^{t-n+1} \mathcal{h}_{q+1-n}(\und x_n)\mathcal{h}_{t-n+1-q}(\und x_n).
\end{equation}
To prove it, first we observe that $\pi_t$ must be in the supercenter $Z_s(A_n(m))$. Indeed, $T_i$ commutes with 
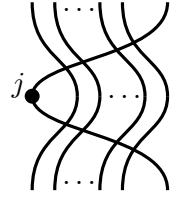
\[
\tikz[xscale = 1.2, very thick,baseline={([yshift=-.5ex]current bounding box.center)}]{
		\draw + (.25,2.25) .. controls (.25,1.5) and (.75,1.5) ..  +  (.75,1); \draw + (.25,-.25) .. controls (.25,.5) and (.75,.5) ..  +  (.75,1); 
		\draw + (.5,2.25) .. controls (.5,1.5) and (1,1.5) ..  +  (1,1); \draw + (.5,-.25) .. controls (.5,.5) and (1,.5) ..  +  (1,1); 
		\draw + (1,2.25) .. controls (1,1.5) and (1.5,1.5) ..  +  (1.5,1); \draw + (1,-.25) .. controls (1,.5) and (1.5,.5) ..  +  (1.5,1); 
		\draw + (1.25,2.25) .. controls (1.25,1.5) and (1.75,1.5) ..  +  (1.75,1); \draw + (1.25,-.25) .. controls (1.25,.5) and (1.75,.5) ..  +  (1.75,1); 
		\draw + (1.75,2.25) .. controls (1.75,1.5) and (.25,1.5) ..  +  (.25,1);
		 \draw + (1.75,-.25) .. controls (1.75,.5) and (.25,.5) ..  +  (.25,1)	node [pos=1,fill=black,circle,inner sep=2pt]{};   \node at (.1,1.15) {\small $j$};
		 \node at (1.3,1) {\small $\dots$};  \node at (.8,2.15) {\small $\dots$};  \node at (.8,-.15) {\small $\dots$};
	 }\
\]
for $1 \le i \le n-1$ in $A_{n+1}(m)$, and thus $\pi_t \in A_n(m)$ has to commute with $T_i$ in $A_n(m)$. Then by~Proposition~\ref{prop:isocenter}, it follows that $\pi_t \in Z_s(A_n(m))$. Therefore, since $\pi_t$ has  $\lambda$-degree 0, we know by Corollary~\ref{cor:leftomegas} that it has to be a symmetric polynomial in $(x_1, \dots, x_n)$, and also it must be zero whenever $t < 2n-2$ (elements with $\lambda$-degree 0 in the supercenter are of non-negative $q$-degree).
From the results in~\S\ref{sec:algA} we know that there is a faithful action of $A_n(m)$ on $\bZ[\und{x}_n] \otimes \bV^\bullet(\und{\omega}_{n})$ and so we only need to check the image of $1$ through 
the action of $\pi_t$. As $\partial_i$ acts by zero on $1$,  we can restrict to $s = 0$ and we have 
\[
\pi_t(1) = - \sum_{q+\ell= t} \left(
\tikz[xscale = 1.2, very thick,baseline={([yshift=-.5ex]current bounding box.center)}]{
		\draw + (.25,2.25) .. controls (.25,1.5) and (.75,1.5) ..  +  (.75,1); \draw + (.25,-.25) .. controls (.25,.5) and (.75,.5) ..  +  (.75,1); 
		\draw + (.5,2.25) .. controls (.5,1.5) and (1,1.5) ..  +  (1,1); \draw + (.5,-.25) .. controls (.5,.5) and (1,.5) ..  +  (1,1); 
		\draw + (1,2.25) .. controls (1,1.5) and (1.5,1.5) ..  +  (1.5,1); \draw + (1,-.25) .. controls (1,.5) and (1.5,.5) ..  +  (1.5,1); 
		\draw + (1.25,2.25) .. controls (1.25,1.5) and (1.75,1.5) ..  +  (1.75,1); \draw + (1.25,-.25) .. controls (1.25,.5) and (1.75,.5) ..  +  (1.75,1); 
		\draw + (1.75,2.25) .. controls (1.75,1.5) and (.25,1.5) ..  +  (.25,1);
		 \draw + (1.75,-.25) .. controls (1.75,.5) and (.25,.5) ..  +  (.25,1)	node [pos=1,fill=black,circle,inner sep=2pt]{}
		node [pos=.15,fill=black,circle,inner sep=2pt]{};   \node at (.1,1.15) {\small $q$};  \node at (1.8,.25) {\small $\ell$};
		 \node at (1.3,1) {\small $\dots$};  \node at (.8,2.15) {\small $\dots$};  \node at (.8,-.15) {\small $\dots$};
	 }\right)(1).
\]
Using again the nilHecke relations~\eqref{eq:relnh1} and~\eqref{eq:relnh2}, together with an generalization of~\eqref{eq:relnhextended} and the fact that $\partial_i$ acts by zero on $1$, we get
\begin{align*}
\left(\tikz[very thick,baseline={([yshift=-.5ex]current bounding box.center)}]{
	          \draw (0,-.5) .. controls (0,0) and (2,0) .. (2,.5) node [pos=0.15,fill=black,circle,inner sep=2pt]{};  \node at (.0,-.05) {\small $i$};
	          \draw (1,-.5) .. controls (1,0) and (0,0) .. (0,.5); 
		 \draw (2,-.5) .. controls (2,0) and (1,0) .. (1,.5);   \node at (.65,.35) {\small $\overset{k}{\dots}$};
	}\ \right)(1)  &=  \mathcal{h}_{i-k}(\und x_{k+1}), &
\left(\ \tikz[very thick,xscale=-1,baseline={([yshift=-.5ex]current bounding box.center)}]{
	           \draw (0,-.5) .. controls (0,0) and (2,0) .. (2,.5) node [pos=0.15,fill=black,circle,inner sep=2pt]{};  \node at (.0,-.05) {\small $i$};
	          \draw (1,-.5) .. controls (1,0) and (0,0) .. (0,.5); 
		 \draw (2,-.5) .. controls (2,0) and (1,0) .. (1,.5);   \node at (.65,.35) {\small $\overset{k}{\dots}$};
	}\right)(1)  &=  \mathcal(-1)^k \mathcal{h}_{i-k}(\und x_{k+1}).
\end{align*}
From this we can see that we must have $\ell \ge n-1$ in order not to get zero. Therefore, we have
\[
\pi_t(1) = (-1)^{n-1} \sum_{\ell \ge n-1}^{t} \left(\tikz[very thick,baseline={([yshift=-.5ex]current bounding box.center)}]{
		\draw (0,-.5) .. controls (0,0) and (2,0) .. (2,.5) node [pos=0.15,fill=black,circle,inner sep=2pt]{};  \node at (-.25,-.05) {\small $t{-}\ell$};
	          \draw (1,-.5) .. controls (1,0) and (0,0) .. (0,.5); 
		 \draw (2,-.5) .. controls (2,0) and (1,0) .. (1,.5);   \node at (.65,.35) {\small $\overset{n-1}{\dots}$};
	}\ \right)(\mathcal{h}_{\ell-n+1}(\und x_{n-1})).
\]
Since $\mathcal{h}_{\ell-n+1}(\und x_{n})$ lies in $Z_s(A_{n}(m))$, we get 
\begin{align*}
\pi_t(1) &= (-1)^{n} \sum_{\ell \ge n-1}^{t}  \mathcal{h}_{\ell-n+1}(\und x_{n}) \left(\tikz[very thick,baseline={([yshift=-.5ex]current bounding box.center)}]{
		\draw (0,-.5) .. controls (0,0) and (2,0) .. (2,.5) node [pos=0.15,fill=black,circle,inner sep=2pt]{};  \node at (-.25,-.05) {\small $t{-}\ell$};
	          \draw (1,-.5) .. controls (1,0) and (0,0) .. (0,.5); 
		 \draw (2,-.5) .. controls (2,0) and (1,0) .. (1,.5);   \node at (.65,.35) {\small $\overset{n-1}{\dots}$};
	}\ \right)(1) \\
&= (-1)^{n} \sum_{\ell \ge n-1}^{t-n+1}  \mathcal{h}_{\ell-n+1}(\und x_{n}) \mathcal{h}_{t-\ell-n+1}(\und x_{n}),
\end{align*}
which prove the claim. Finally, applying~(\ref{eq:pit}) on~(\ref{eq:projdiff1}) gives 
\[
(-1)^{n} \sum_{p=0}^{j-2}\  \sum_{q=n-1}^{j-1-p-n}  \mathcal{h}_{q+1-n} (\und x_n) \mathcal{h}_{j-1-p-n-q}(\und x_n) \xi^p,
\]
which is equivalent to the statement of the lemma, after decomposing
\[
 (-1)^{n} \sum_{p' = n }^{j-n}  \mathcal{h}_{p' - n}(\und x_n, \xi) \mathcal{h}_{j-n-p'}(\und x_n)  =  (-1)^{n} \sum_{p' = n }^{j-n}\sum_{q'=0}^{p'-n} \mathcal{h}_{p'-n-q'}(\und x_n)  \mathcal{h}_{j-n-p'}(\und x_n)  \xi^{q'},
\]
and applying a triangular change of variables $p' = 1+p+q$, $q' = p$ in the sums.
\end{proof}

Now we can define a differential on $\left(A_{n}(m) \otimes \bZ[\xi] \right) \oplus \left( q^{2m-4n}\lambda^2  A_{n}(m)[1] \otimes \bZ [\xi] \right)$, which we denote again by $d_N$. To do so, we only have to define it on each of the independant generators of the direct decomposition into free $A_n(m)$-bimodules. Thus we set
\[
d_N(0\oplus \xi^i) =\biggl( (-1)^{N+m+n} \sum_{p = n}^{N+m+i-n} \mathcal{h}_{p - n}(\und x_n, \xi) \mathcal{h}_{N+m+i-n-p}(\und x_n)\biggr) \oplus 0,
\]
and $d_N(\xi^i \oplus 0) = 0$, for all $i \in \bN_0$. The action of the differential in general is induced by the actions of $(A_n(m),d_N)$. Moreover, by Lemma~\ref{lem:projvalues}, we know $d_N$ commutes with the $(A_n(m), A_n(m))$-bimodule maps from Theorem~\ref{thm:sesAnm}. This means we get a short exact sequence of dg-bimodules 
\begin{align*}
0 \rightarrow (q^{-2}A_{n}(m) \otimes_{n-1}& A_{n}(m), d_N) \rightarrow (A_{n+1}(m),d_N) \\
 &\rightarrow \left(\left(A_{n}(m) \oplus q^{2m-4n}\lambda^2 A_{n}(m)[1] \right) \otimes \bZ[\xi] ,d_N\right) \rightarrow 0.
\end{align*}
By the snake lemma in the abelian category of complexes of graded bimodules, it descends to a long exact sequence
\[
\xymatrix@C-4pc{
 & \dots \ar[dl] \\
H^1\ (A_{n+1}(m),d_N) \ar[r] &   H^1\left(\left(A_{n}(m)\oplus q^{2m-4n}\lambda^2   A_{n}(m) [1] \right)\otimes \bZ[\xi],d_N\right) \ar[dl]_-{\delta} \\
 H^0 (q^{-2}A_{n}(m) \otimes_{n-1} A_{n}(m), d_N) \ar[r] &   H^0\ (A_{n+1}(m),d_N)  \ar[dl] \\
 H^0\left(\left(A_{n}(m)  \oplus q^{2m-4n}\lambda^2  A_{n}(m) [1] \right) \otimes \bZ[\xi] ,d_N\right) \ar[r]^-{\delta} & H^{-1} (q^{-2}A_{n}(m) \otimes_{n-1} A_{n}(m), d_N) \ar[dl] \\
\dots &
}
\]
of $H(A_n(m), d_N)$-bimodules.
Then, since $H(A_n(m), d_N)$ is concentrated in homological degree~$0$ and isomorphic to  $\nh_n^{m+N}$, we obtain
\[
\xymatrix@C-4pc{
 & \dots \ar[dl] \\
0 \ar[r] &   H^1\left(\left(A_{n}(m)  \oplus q^{2m-4n}\lambda^2   A_{n}(m) [1]\right)\otimes \bZ[\xi],d_N\right) \ar[dl]_-{\delta} \\
q^{-2} \nh_n^{m+N} \otimes_{\nh_{n-1}^{m+N}} \nh_n^{n+m}   \ar[r] &  \nh_{n+1}^{m+N}  \ar[dl] \\
 H^0\left(\left(A_{n}(m)  \oplus q^{2m-4n}\lambda^2  A_{n}(m) [1] \right)\otimes \bZ[\xi] ,d_N\right) \ar[r]^-{\delta} & 0 \ar[dl] \\
\dots &
}
\]
of $\nh_n^{m+N}$-bimodules.

\begin{lem}
There are isomorphisms of $\nh_n ^{m+N}$-bimodules for $N+m-2n \ge 0$,
\[
H^i\left(\left(A_{n}(m) \oplus  q^{2m-4n} \lambda^2   A_{n}(m)[1]\right) \otimes \bZ[\xi],d_N\right) \cong 
\begin{cases}
\bigoplus\limits_{\{N+m-2n\}} \nh_n^{m+N}, &\text{ if } i=0, \\
\mspace{29mu}0, &\text{ if } i \neq 0,
\end{cases}
\]
and for $N+m-2n \le 0$,
\[
H^i\left(\left(A_{n}(m) \oplus q^{2m-4n} \lambda^2 A_{n}(m)[1]\right)\otimes \bZ[\xi],d_N\right) \cong 
\begin{cases}
\mspace{29mu}0, &\text{ if } i\neq 1,\\
 \bigoplus\limits_{\{2n-m-N\}}\lambda^2 q^{2m-4n}   \nh_n^{m+N}, & \text{ if } i = 1, \\
\end{cases}
\]
where $\oplus_{ \{k\}} M = \bigoplus_{i=0}^{k-1} q^{ 2i} M$.
\end{lem}

\begin{proof}
 These homologies are easily computed by looking at the image of the independent generator elements $0 \oplus \xi^i$ and  $\xi^j \oplus 0$, $i,j \ge 0$. We have $d_N(\xi^j \oplus 0) = 0$. For $N+m-2n \le 0$, we compute $d_N(0 \oplus \xi^i) = 0$ whenever $i < 2n-m-N$, and $d_N(0 \oplus \xi^{2n-m-N}) = 1 \oplus 0$. For $N+m-2n \ge 0$, we observe that $d_N(0 \oplus \xi^i)$ gives a monic polynomial with respect to $\xi$, up to sign, and leading term given by $(-1)^{N+m+n}\xi^{N+m-2n+i}$.
\end{proof}

After shifting by the degree of the connecting homomorphism,
we recover the direct sum decompositions of the cyclotomic nilHecke algebra that categorify the $\slt$-commutator in the categorification of the irreducible, finite-dimensional $U_q(\slt)$-module $V(N+m)$ (see \cite[Theorem~5.2]{KK}):
\begin{align*}
 &\nh_{n+1}^{m+N} \cong \left(q^{-2} \nh_n^{m+N} \otimes_{n-1} \nh_n^{m+N} \right) \oplus_{\{N+m-2n\}} \nh_n^{m+N}, &\text{ if } N+m-2n \ge 0, \\
&q^{-2} \nh_n^{m+N} \otimes_{n-1} \nh_n^{m+N} \cong  \nh_{n+1}^{m+N}  \oplus_{\{2n-m-N\}} q^{2m-4n+2N} \nh_n^{m+N}, &\text{ if } N+m-2n \le 0.
\end{align*}

\subsubsection{Categorification of $V(N)$} 
Until the end of the section, we work with $\bQ$ as ground ring.

\smallskip

For a dg-algebra $(A,d)$, let $\cD^c( A, d)$ denote the triangulated, full subcategory of compact objects in the derived category $\cD(A,d)$ of $(A,d)$, and $\cD^f (A, d)$ be the triangulated, full subcategory of $\cD(A,d)$ consisting in objects which are quasi-isomorphic to finite-dimensional $(A,d)$-modules
(see \cite{keller} for explanations about those notions, see also \cite{eliasqi} for a nice exposition of similar notions used in the context of categorification). 

\smallskip

The dg-algebra $(A_n(m), d_N)$ is \emph{formal}, i.e. it is quasi-isomorphic to its homology equipped with a trivial differential. Indeed, we can consider the surjective map $(A_n(m), d_N) \rightarrow (\nh_n^N, 0)$ which sends floating dots to zero and projects onto the quotient by the cyclotomic ideal. Let $Kom\left(\nh^{N+m}\fgpmod \right)$ be the homotopy category of bounded complexes of finitely generated, projective left, $\nh^{N+m}$-modules. Similarly, $Kom\left(\nh^{N+m}\fdmod \right)$ is given by bounded complexes of finite-dimensional modules.  By standard arguments (see \cite{keller}), we get the following:

\begin{thm}\label{thm:derivedequiv}
There are 
 equivalences 
 of triangulated categories  
\[
\cD^c (A(m), d_N) \cong \cD^c(\nh^{N+m}, 0) \cong Kom\left(\nh^{N+m}\fgpmod \right),
\]
and
 equivalences of triangulated categories
\[
\cD^f (A(m), d_N) \cong \cD^f(\nh^{N+m}, 0) \cong Kom\left(\nh^{N+m}\fdmod \right). 
\]
\end{thm}

Recall from~\cite{KK} that the category of finitely generated, projective left $\nh^{N+m}$-modules (resp. category of finite-dimensional, left $\nh^{N+m}$-modules), when working with ground ring $\bQ$, categorify the irreducible highest weight module $V(N+m)$ (resp. the dual weight module $V(N+m)^*$) of Luztig integral $U_{\bZ[q,q^{-1}]}(\slt)$ (see \cite{lusztig}). 
As a direct consequence of this fact together with Theorem~\ref{thm:derivedequiv} we have:

\begin{cor}
There are isomorphisms of $U_{\bZ[q,q^{-1}]}(\slt)$-representations
\[
K_0\left(\cD^c\left( (A(m), d_N)\amod\right)\right) \cong V(N+m),
\]
and
\[
K_0\left(\cD^f\left( (A(m), d_N)\amod\right)\right) \cong V(N+m)^*.
\]
\end{cor}

\begin{rem}
In fact, the categorical action of $U_{\bZ[q,q^{-1}]}(\slt)$ on $Kom\left(\nh^{N+m}\fgpmod \right)$ (or equivalently on $Kom\left(\nh^{N+m}\fgpmod \right)$) coincide with the one obtain by deriving the induction/restriction functors given by the inclusion $(A_n(m), d_N) \hookrightarrow  (A_{n+1}(m), d_N)$ that adds a vertical strand at the right. Hence we can think of these as being `equivalent 2-representations up to homotopy'.
\end{rem}

\subsubsection{Equivariant-cohomology} \label{ssec:equivdifferential}
As observed in~\cite[\S4]{AEHL}, there are other interesting differentials that can be defined on $A_n$. 
Let $\Sigma$ denote the multiset of roots of the polynomial
\[
P(x) = \sum_{j=0}^{N+m} \kappa_j x^{N+m-j},
\]
where $\kappa_1, \dots, \kappa_{N+m}$ are generic parameters of degree $\deg(\kappa_i) = (2i,0)$
and $\kappa_0 = 1$. The deformed differential is defined on $A^\Sigma_n(m) = A_n(m) \otimes \bQ[\und \kappa_{N+m}]$ by
\begin{align*}
d_N^\Sigma(T_i) &=0, & d_N^\Sigma(x_i) &= 0,& d_N^\Sigma(\omega_i^a) = (-1)^{N+a+i} \sum_{j=0}^{N+a-i}  \kappa_j \mathcal{h}_{N+a+i-j}(x_1, \dots, x_i).
\end{align*}

\begin{prop}\emph{(\cite[Corollary~4.9]{AEHL})}
This defines a differential  on $A^\Sigma_n(m)$ of $(q,\lambda)$-degree $(2N, -2)$.
\end{prop}

By~\cite[Theorem~4.10]{AEHL} there is a quasi-isomorphism between 
\[
(A^\Sigma_n(m), d_N^\Sigma) \xrightarrow{\cong} \nh_n^\Sigma = \nh_n \otimes \bQ[\und \kappa_{N+m}]/I_{N+m}^\Sigma
\]
 where $I_{N+m}^\Sigma$ is the two-sided
ideal generated by $\sum_{j=0}^{N+m} \kappa_j x_1^{N+m-j}$.

\begin{rem}In view of~\cite[Lecture~6]{fulton}, 
the center $Z(\nh_n^\Sigma)$ is isomorphic to the $GL(N)$-equivariant cohomology  $H^*_{GL(N)}(Gr(n,N))$.
\end{rem}

Making use of Lemma~\ref{lem:projvalues}, we define a differential $d_N^\Sigma$ on the bimodule
\[
\left(A^\Sigma_{n}(m) \otimes \bQ[\xi]\right) \oplus \left(q^{2m-4n}\lambda^2  A^\Sigma_{n}(m)[1] \otimes \bQ [\xi]\right)
\]
by setting $d_N^\Sigma(\xi^i \oplus 0) = 0$, and 
\[
d_N^\Sigma(0 \oplus \xi^i) = (-1)^{N+m+n} \sum_{j=0}^{N+m} \kappa_j \sum_{p = n }^{N+m+i-j-n} \mathcal{h}_{p - n}(\und x_n, \xi) \mathcal{h}_{N+m+i-j-n-p}(\und x_n).
\]
This yields a short exact sequence of dg-bimodules
\begin{align*}
0\rightarrow (q^{-2}A^\Sigma_{n}(m) \otimes_{n-1}& A^\Sigma_{n}(m), d^\Sigma_N) \rightarrow (A^\Sigma_{n+1}(m),d^\Sigma_N) \\
 &\rightarrow \left(\left(A^\Sigma_{n}(m) \otimes \bZ[\xi]\right) \oplus \left(q^{2m-4n}\lambda^2  A^\Sigma_{n}(m)[1] \otimes \bZ [\xi]\right) ,d^\Sigma_N\right) \rightarrow 0,
\end{align*}
which in turn implies direct sum decompositions
\begin{align*}
 &\nh_{n+1}^{\Sigma} \cong \left(q^{-2} \nh_n^{\Sigma} \otimes_{n-1} \nh_n^{\Sigma}\right) \oplus_{\{N+m-2n\}} \nh_n^{\Sigma}, &\text{ if } N+m-2n \ge 0, \\
&q^{-2} \nh_n^{\Sigma} \otimes_{n-1} \nh_n^{\Sigma} \cong  \nh_{n+1}^{\Sigma}  \oplus_{\{2n-m-N\}} q^{2m-4n+2N} \nh_n^{\Sigma}, &\text{ if } N+m-2n \le 0.
\end{align*}

Therefore, we get

\begin{thm}
There is an equivalence of triangulated categories
\begin{gather*}
\cD^c(A^\Sigma(m), d_N^\Sigma) \cong Kom(\nh^\Sigma\fgpmod)
\intertext{and}
K_0(\cD^c(A^\Sigma(m), d_N^\Sigma)) \cong K_0(\nh^\Sigma\fgpmod) \cong V(N+m) .
\end{gather*}
\end{thm}

\begin{rem}
In~\cite{AEHL} it is proved that $d_N^\Sigma$ restricts to $R^{S_n}$ (which is clear from the perspective of bubbled algebras) and $ H(R^{S_n}, d_N^\Sigma) \cong H^*_{GL(N)}(G_{n;N}, \bQ)$ (see~\S\ref{ssec:equivcoh}).
We know by Proposition~\ref{prop:RsnOmega} that $R^{S_n} \cong \Omega_n$. Hence, we can induce a differential $d_N^\Sigma$ on $\Omega_n \otimes \bQ[\und \kappa_{N}]$. From the differential $d_N^\Sigma$ on $A_n(-1)$, we get an induced differential $d_N^\Sigma$ on $\Omega_{n,n+1}$, sending a floating dot to its corresponding bubble, so that $H(\Omega_{n,n+1}, d_N^\Sigma) \cong H^*_{GL(N)}(G_{n,n+1;N}, \bQ)$. This way we recover the categorification of $V(N)$ in~\cite{L2} using equivariant cohomology from the (geometric) categorification of $M(\lambda q^{-1})$ constructed by the authors in~\cite{naissevaz1}.
\end{rem}

\subsection{Categorical $\slt$-action on $(0,\infty)\text{-}\BA_n$}\label{ssec:catbubble}

The inclusion of algebras
\[
(0,\infty)\text{-}\BA_n \subset (0,\infty)\text{-}\BA_{n+1}
\]
yields induction and restriction functors allowing us to define
\begin{align*}
&\tilde\F_n : (0,\infty)\text{-}\BA_n(m)\smod \rightarrow (0,\infty)\text{-}\BA_{n+1}(m) = \Ind_n^{n+1},  \\
&\tilde \E_n : (0,\infty)\text{-}\BA_{n+1}(m)\smod \rightarrow (0,\infty)\text{-}\BA_{n}(m) =  q^{2n-m} \lambda^{-1} \Res_n^{n+1}. 
\end{align*}

Using similar arguments as above, one can show that this defines a categorical $\slt$-action on $(0,\infty)\text{-}\BA(m)\smod$ (the main ingredient being that we can bring all $Y$'s and $\varpi$'s to the left, so that we can have analogs of Lemma~\ref{lem:newbasisdecomp} and Theorem~\ref{thm:sesAnm}).

\begin{thm}
There is a natural short exact sequence:
\[
0 \rightarrow \tilde\F_{n-1}\tilde\E_{n-1} \rightarrow \tilde\E_{n}\tilde\F_{n} \rightarrow q^{m-2n}\lambda \Q_{n+1}\oplus q^{2n-m}\lambda^{-1} \Pi \Q_{n+1} \rightarrow 0.
\]
\end{thm}

From this we deduce as in Remark~\ref{rem:psmodlfg}:

\begin{thm}
The split (topological) Grothendieck group of $(0,\infty)\text{-}\BA(m)\prmods$ is an $U_q(\slt)$-module and
\[
\boldsymbol K_0((0,\infty)\text{-}\BA(m)\prmods) \cong M(\lambda q^m).
\]
\end{thm}

With the aim of imitating \S\ref{sec:differential}, replacing $A_n(m)$ by $(0,\infty)\text{-}\BA(m)$ and $\nh_n^{N+m}$ by  $\nh_n^\Sigma$, we observe the following:

\begin{prop}\label{prop:isonhsNbnh}
There is an isomorphism $\nh_n^\Sigma \cong (0,N+m)\text{-}\bnh_n$ where $\kappa_j$ is identified with $(-1)^j Y_{j,0}$.
\end{prop}

\begin{proof}
First, we observe that by~\eqref{eq:bubblerelationY} in $\bnh_n$
\[
Y_{N+m+k,0} = \sum_{r=0}^{N+m+k} \mathcal{e}_r(\und x_n) Y_{N+m+k-r, n}.
\]
Since we kill all $Y_{i,n}$ for $i > N+m-n$ and $\mathcal{e}_r(\und x_n) = 0$ for $r > n$, we get $Y_{N+m+k,0} = 0$ if and only if $k > 0$. This means we can define a surjective map $ (0,N+m)\text{-}\bnh_n \rightarrow \nh_n^\Sigma$ sending $(-1)^j Y_{j,0}$ to $\kappa_j$ for $j \le N+m$. Then, we observe that
\[
\sum_{j=0}^{N+m} (-1)^j Y_{j,0} x_1^{N+m-j}  = Y_{N+m,1}  = \sum_{r=0}^{N+m} \mathcal{e}_{r}(x_2, \dots, x_n) Y_{N+m-r,n} = 0,
\]
since $ \mathcal{e}_{r}(x_2, \dots, x_n) = 0$ if $r > n-1$ and $Y_{i,n} = 0$ if $i > N+m-n$.
\end{proof}

As before, we define a differential $d_N : (0,\infty)\text{-}\BA_n(m) \rightarrow (0,\infty)\text{-}\BA_n(m)$ for $N+m \ge 0$ by
\begin{align*}
d_N(T_i) &=0, & d_N(\omega_i^a) &= (-1)^{N+a+i} \sum_{j=0}^{N+a-i}  \mathcal (-1)^j   Y_{j,0} \mathcal{h}_{N+a+i-j}(x_1, \dots, x_i) = Y_{N+a-i,i},\\
 d_N(x_i) &= 0,& d_N(\varpi_{0}^a) &=  Y_{N+a,0}.
\end{align*}
This is well defined since, by the same arguments as in Proposition~\ref{prop:invert-S0i}, we can bring all $\varpi_{i}^a$ to the left, and both $Y_{N+a,0}$ and $\varpi_{0}^a$ (super)commutes with everything.  

\begin{rem}
We stress that in general $d_N(\varpi_{i}^a) \neq X_{N+a-i,i}$ and $d_N(\varpi_{i}^a) \neq Y_{N+a-i,i}$ but 
\begin{align*}
d_N(\varpi_{i}^a) &= \sum_{\ell = 0}^{i} \mathcal{e}_{\ell}(x_1, \dots, x_i) d_N(\varpi_0^{i+a+\ell}) \\
&=  \sum_{\ell = 0}^{i} X_{\ell,i} Y_{N+a+\ell,0}.
\end{align*}
\end{rem}

Moreover, since the minimal label for $\varpi_0^a$ is $m+1$ we have $d_N(\varpi_{0}^a) = Y_{N+a,0}$ with $a \ge m+1$. Also,
$d_N(\omega_1^{m+1}) = Y_{N+m,1} = \sum_{j=0}^{N+m} (-1)^j Y_{j,0} x_1^{N+m-j}$  does not involve any $Y_{i,0}$ for $i \ge N+m$,  and thus
 we get
\[
H((0,\infty)\text{-}\BA_n(m), d_N)  \cong (0,N+m)\text{-}\bnh_n \cong \nh_n^\Sigma.
\]
Then, it yields a quasi-isomorphism
\[
((0,\infty)\text{-}\BA_n(m), d_N) \cong   (\nh_n^\Sigma, 0),
\]
so that again
\begin{gather*}
\cD^c((0,\infty)\text{-}\BA(m), d_N)  \cong Kom(\nh^\Sigma\fgpmod),
\intertext{and}
K_0(\cD^c((0,\infty)\text{-}\BA(m), d_N)) \cong V(N+m) .
\end{gather*}

\begin{rem}
It is also possible to first define two differentials $d_N'$ and $d_N^\Sigma$ on $(0,\infty)\text{-}\BA_n(m)$, both sending all generators to zero except for $d_N'(\varpi_{0}^a) =  Y_{N+a,0}$, and for $d_N^\Sigma(\omega_i^a) = Y_{N+a-i,i}$. This is well defined  since $\omega_i^a$ behaves like $Y_{N+a-i,i}$. Then, $H((0,\infty)\text{-}\BA_n(m), d'_N) \cong A_n^\Sigma(m)$ and $d_N = d_N' + d_N^\Sigma$. Since $d_N'$ commutes with $d_N^\Sigma$, we get an induced differential $d_N^\Sigma$ on $H((0,\infty)\text{-}\BA_n(m), d'_N)$, coinciding with the one defined on $A_n^\Sigma(m)$ in~\S\ref{ssec:equivdifferential}. In particular,  we recover \cite[Corollary~4.9]{AEHL}. Moreover, since $H((0,\infty)\text{-}\BA_n(m), d'_N)$ is concentrated in homological degree zero with respect to the degree given by counting the number of $\varpi_i^a$'s, the spectral sequence going from $ H((0,\infty)\text{-}\BA_n(m), d'_N)$ to the total homology $H((0,\infty)\text{-}\BA_n(m), d'_N + d_N^\Sigma)$ converges at the second page, so that
\begin{align*}
H(A_n^\Sigma(m), d_N^\Sigma) &\cong H(H((0,\infty)\text{-}\BA_n(m), d_N'), d_N^\Sigma)  \\
&\cong H((0,\infty)\text{-}\BA_n(m), d'_N + d_N^\Sigma) \\
& \cong   (0,N+m)\text{-}\bnh_n \cong \nh_n^\Sigma.
\end{align*}
and we recover \cite[Theorem~4.10]{AEHL}.
\end{rem}

\subsection{Cyclotomic quotients}\label{ssec:cycquot}

For $N\in\bN_0$, let $A^N_n$ denote the quotient 
$A^N_n=A_n / (x_1^N)$ and put
\[
A^N = \bigoplus\limits_{n\geq 0}A_n^N ,
\]
which is a finite-dimensional, bigraded superalgebra thanks to \cite[Proposition~2.8]{hoffnunglauda} (in particular $A_n^N \cong 0$ for $n > N$). 
In the following we let $K_0(B)$ denote the split Grothendieck group of the category $B\fgpmod$
of finitely generated, projective, (super)modules over a finite-dimensional, bigraded (super)algebra $B$. 

Let $I_n$ be the minimal 2-sided ideal of $A_n^N$ containing $\omega_1,\dotsc,\omega_n$  
(it is generated by the $\omega_n^a$'s cf. Proposition~\ref{prop:invert-S0i}). 
Since the $\omega$'s are ``exterior elements'', $I_n$ is nilpotent. Hence, the superalgebras $A_n^N/I_n$ and $\nh_n^N$ are isomorphic (the latter seen as
a superalgebra concentrated in parity 0), we have an isomorphism of $\bZ$-modules
$K_0(A_n^N/I_n)\cong K_0(\nh_n^N)$ for all $n$. 
However, the methods from~\S\ref{ssec:resind-cataction},
do not give a categorical $\slt$-action without a significative modification of the functors $\F$ and $\E$.

\subsection{Idempotented half 2-Kac-Moody algebra for $\slt$}

For $m\in\bZ$ and for $(n_1,n_2,\dotsc ,n_k)\in\bN_0^{k}$ a composition of $n$ let 
$(m_1,m_2,\dotsc ,m_k)\in\bZ^{k}$ be defined by the rule $m_1=m$ and for $i>1$, $m_{i+1}=m_i-n_i$. 
Define
\[
A_{n_1,n_2,\dotsc, n_k}(m) = A_{n_1}(m_1) \otimes A_{n_2}(m_2)\otimes \dotsm \otimes A_{n_k}(m_k).
\]
We have inclusions of bigraded superalgebras
\[
\psi_{\und{n},\und{m}}\colon A_{n_1,n_2,\dotsc, n_k}(m) \to A_{n}(m) . 
\]

Let $k=2$  and define the functors
\begin{align*}
  \Ind_{(n_1,n_2),m}^{n_1+n_2,m} &\colon A_{n_1,n_2}(m)\smod \to A_{n_1+n_2}(m)\smod ,
  \\[1.5ex]
  \Res_{(n_1,n_2),m}^{n_1+n_2,m} &\colon A_{n_1+n_2}(m)\smod \to A_{n_1,n_2}(m)\smod .
\end{align*}

\n Similarly as before, we define functors of induction and restriction as sums of the above
functors:
\begin{align*}
  \Ind_{n_1,n_2}^{n_1+n_2} &= \bigoplus\limits_{m\in\bZ}\Ind_{(n_1,n_2),m}^{n_1+n_2,m}
  \colon A_{n_1,n_2}\smod \to A_{n_1+n_2}\smod ,
  \\[1.5ex]
  \Res_{n_1,n_2}^{n_1+n_2} &= \bigoplus\limits_{m\in\bZ}\Res_{(n_1,n_2),m}^{n_1+n_2,m}
  \colon A_{n_1+n_2}\smod \to A_{n_1,n_2}\smod .
\end{align*}

As in the case of the nilHecke algebra, the functors of induction above define products on $K'_0(A)$ 
(as in~\S\ref{ssec:cycquot}) 
and the
functors of restriction above define coproducts. These products and
coproducts make $K'_0(A)$ 
into twisted 
bialgebra
as in the  case of the nilHecke algebras~\cite[Props. 3.1-3.2]{KL1}. 

Let ${U}^-_{\bZ[q,q^{-1}]}(\slt)$ be an integral version of the half quantum group $\slt$ corresponding with Beilinson--Lusztig--MacPherson's~\cite{BLM} idempotented quantum $\slt$,  that is the subalgebra generated by the $1_{n-2} F 1_n$'s in $U_{\bZ[q,q^{-1}]}(\slt)$.

\begin{prop}
The split Grothendieck groups $K'_0(A)$ 
  is isomorphic with ${U}^-_{\bZ[q,q^{-1}]}(\slt)$.
\end{prop}

%
%

%
%
\section{2-Verma modules}\label{sec:verma2rep}


Let $\Bbbk$ be a field of characteristic $0$.
 Let $c$ be either an integer or a formal parameter and define $\varepsilon_c$ 
to be zero if $c\in\bZ$ and to be 1 otherwise. Let $\Lambda_{c,m} = c + m  - 2\bN$. 
Recall that a \emph{Quillen exact category}~\cite[\S2]{quillen} is a full subcategory of an abelian category, closed under extensions. We can view it as an additive category equipped with a class of short exact sequences given by ker-coker pairs respecting some axioms (e.g. direct sums yield short exact sequences). Moreover, we say an additive category is locally additive, cone complete~\cite[\S5]{naissevaz1} if it is stricly bigraded (i.e.  $q^a\lambda^b C  \not\cong C$ for all $C \in \cC$ and $(a,b)\in \bZ\times\bZ$) and if it admits all locally finite, cone bounded coproducts, and those coincide with the locally finite products. Now recall the definition of $2$-Verma module for quantum $\slt$ from~\cite[\S6.3]{naissevaz1}.

\begin{defn}  
A  \emph{2-Verma module} for $\slt$ with highest weight $c+m$
consists of a bigraded, $\Bbbk$-linear, idempotent complete (strict) 2-category $\tM$ admitting a parity 2-functor
  $\Pi : \tM \rightarrow \tM$, where:
\begin{itemize}
\item The objects of $\tM$ are indexed by weights $\mu\in c + \bZ$.
\item There are identity 1-morphisms $\un_{\mu}$ for each $\mu \in  c + \bZ$,
  as well as 1-morphisms $\F\un_\mu\colon\mu \to \mu -2$ in $\tM$ and their grading shift.
  We also assume that $\F\un_\mu$ has a right adjoint and define the 1-morphism $\E\un_\mu\colon \mu-2 \to \mu$ as
a grading shift of a right adjoint of $\F\un_\mu$, 
\[
\un_\mu\E  = q^{c-\mu +2} \lambda^{-\varepsilon_c} (\F\un_\mu)_R  .
\]
\item The $\Hom$-spaces between objects are locally additive, cone complete, Quillen exact categories.
\end{itemize}
On this data we impose the following conditions:
\begin{enumerate}
\item The identity 1-morphism $\un_{\mu}$ of the object $\mu$ is isomorphic to the zero 1-morphism 
  if $\mu\notin\Lambda_{c,m}$.
\item The enriched $2\text{-}\Hom$, $\HOM_{\tM}(\un_{\mu}, \un_{\mu})$, is cone bounded for all $\mu$. 
\item \label{co:EF} There is an exact sequence 
\[
0\xra{\quad} 
\F\E\un_\mu
\xra{\quad} 
\E\F\un_\mu
\xra{\quad} 
q^{-c+\mu} \lambda^{\varepsilon_c} \Q_\mu  \oplus\ q^{c-\mu}\lambda^{-\varepsilon_c}\Pi\Q_\mu 
\xra{\quad} 0 , 
\]
where $\Q_\mu := \bigoplus_{k \ge 0} q^{2k+1} \Pi \un_\mu$.
\item  For each $k \in \bN_0$, $\F^k\un_{\mu}$ carries a faithful action of the enlarged nilHecke algebra $A_k(\mu-c)$.
\end{enumerate}
\end{defn}

Following Rouquier~\cite{R1}, we now restrict to the case of 2-Verma modules which are subcategories of the strict $2$-categories of all bigraded, additive, $\Bbbk$-linear supercategories, with $1$-morphisms being (additive) $\Bbbk$-linear functors, and  $2$-morphisms being grading preserving natural transformations. For such a 2-Verma module, we write $\mathcal{M} = \bigoplus_j \mathcal{M}_{c+j}$ with $\mathcal{M}_{c+j}$ being the category corresponding to the object $c+j$ in it. We will call this a  \emph{$\Bbbk$-linear 2-Verma module}. When $\mathcal{M}$ is abelian, we will say it is an \emph{abelian 2-Verma module}.

\smallskip

A nice property of $\Bbbk$-linear 2-Verma module is that, since $\F$ and $\E$ are functors, the short exact sequence~\eqref{co:EF} yields a ker-coker pair of natural transformations.
Therefore, evaluating the functors involved on an object of $\mathcal{M}$ gives a ker-coker pair of morphisms in $\mathcal{M}$.   In particular, for a projective object $P$ we obtain a short exact sequence of functors
\begin{align*}
0\xra{\quad} 
\Hom(P, \F\E\un_\mu -)
\xra{\quad} 
\Hom&(P, \E\F\un_\mu -) \\
&\xra{\quad} 
\Hom\left(P, \left(q^{-c+\mu} \lambda^{\varepsilon_c} \Q_\mu  \oplus\ q^{c-\mu}\lambda^{-\varepsilon_c}\Pi\Q_\mu\right) - \right)
\xra{\quad} 0 .
\end{align*}

Form the
2-category $\tM'(m)$ whose objects 
are the categories $ A_k(m)\lfmods$, the 1-morphisms are cone bounded, locally finite direct sums of shifts of compositions and summands of functors from $\{\E_k, \F_k, \Q_k, \id_k \}$, and the 
2-morphisms are (grading preserving) natural transformations of functors.
We define $\tM(m)$ as the completion under extensions of $\tM'(m)$ in the abelian 2-category of abelian categories. Then, $\tM(m)$ is an example of abelian, $\Bbbk$-linear 2-Verma module which categorify the Verma module $M(\lambda q^{m})$.
As a matter of fact, this 2-category is equivalent to the completion under extensions of the 2-category induced by the $ \Omega_k^m\lfmods$ from~ \cite{naissevaz1} (this can be seen as a consequence of the Morita equivalence between $\Omega_k^m$ and $A_k(m)$).

\begin{lem}\label{lem:homdim}
Let $\mathcal{M}$ be a $\Bbbk$-linear 2-Verma module with highest weight $c + m$. Suppose $\mathbb{M} \in \mathcal{M}_{c+m}$ is projective. Then we have
\[
\gdim \HOM_{c+m-2k}(\F^k\mathbb{M}, \F^k\mathbb{M}) = \gdim  \left(A_k(m)\right) \gdim \left(\END_{c+m}(\mathbb{M})\right),
\]
where $\gdim$ is the $\bZ \times \bZ \times \bZ/2\bZ$-graded dimension over $\Bbbk$.
\end{lem}

\begin{proof}
For the sake of simplicity we suppress the subscript $c+m-2k$ from the equations.
Using the adjunction hypothesis on $\E$ and $\F$ we get
\[
\HOM(\F^k\mathbb{M}, \F^k\mathbb{M}) \cong  \HOM(\F^{k-1}\mathbb{M}, q^{m-2k}\lambda^\varepsilon_c  \E\F^k\mathbb{M} ),
\]
and thus by applying it $k$ times we get
\[
\HOM(\F^k\mathbb{M}, \F^k\mathbb{M}) \cong  \HOM(\mathbb{M},  q^{\sum_{\ell = 0}^{k-1} m-2(k-\ell)}\lambda^{\varepsilon_c k} \E^k\F^k\mathbb{M} ).
\]
Since $\mathbb{M}$ is projective,  the functor $\HOM(\mathbb{M}, -)$ is exact and the short exact sequence ~\eqref{co:EF} yields a short exact sequence in the $\HOM$'s.
Recall also from~\cite[\S 5]{naissevaz1} that we have $\Hom(X, \bigoplus_i Y_i) \cong \prod_{i} \Hom(X, Y_i)$ in a locally additive category. Thus, we get that
\[
\gdim \HOM(X, Q Y) = \pi q(1+q^{2}+\dots) \gdim \HOM(X, Y).
\]
 Also, we observe that 
\begin{align*}
\gdim A_k(m) &= \gdim \HOM_{A(m)}(A_k(m), A_k(m))\\
& = q^{\sum_{\ell=1}^{k} m-2\ell } \lambda^{\epsilon_ck} \gdim \HOM_{A(m)}(A_0(m), \E_0 \dotsm \E_{k-1} \F_{k-1}\dotsm \F_0 A_0(m)),
\end{align*}
and that we can apply the short exact sequence from Corollary~\ref{cor:EF-rel} on the right, since $A_0(m)$ is projective. Then the statement follows from the fact that the short exact sequences in~\eqref{co:EF} and and in Corollary~\ref{cor:EF-rel} are similar.
\end{proof}

We will now show that the $\Bbbk$-linear 2-Verma module constructed from the category of graded projective modules over $A_n$
is the essentially unique $\Bbbk$-linear $2$-Verma module.

\smallskip

For an object $\mathbb{M}$ in a $\Bbbk$-linear category and $\mathcal C$ another $\Bbbk$-linear category, denote by $\mathcal C \otimes_\Bbbk  \mathbb M$ the category having the same objects as  $\mathcal C$ but with hom-spaces being $\Hom_{\mathcal C}(-,-) \otimes_\Bbbk \End(\mathbb M)$.

For $\mathcal{M}$ a  $\Bbbk$-linear $2$-Verma module with highest weight $c+m$ and  $\mathbb{M} \in \mathcal{M}_{c+m}$, we define the coproduct preserving, $\Bbbk$-linear functor 
\[
\left(\bigoplus_{k \ge 0} \F^{k} \mathbb{M}\right) \otimes_{A(m)} (-)  : A(m)\prmods \otimes_\Bbbk \mathbb M \rightarrow \mathcal{M},
\]
which sends the free $A_k(m)$-module $A_k(m)$ to $\F^k \mathbb M$. Since $\mathfrak{M}$ is idempotent complete, the image of the idempotent $e_k \in \End(\F^k \un_{m+c})$ (see~\S\ref{ssec:idempots}) gives a functor $\F^{(k)}\un_{m+c}$. Then, the indecomposable projective $A_k(m)$-module $P_{(k)}(m)$ is sent to $\F^{(k)} \mathbb M$. A morphism $f \otimes_\Bbbk g \in \End(A_k(m) \otimes_\Bbbk \mathbb M) \cong A_k(m) \otimes_\Bbbk  \End(\mathbb M)$ is sent to $ f \bullet \F^k(g)$ where $f \bullet -$ is the action of $A_k(m)$ on $\F^k\un_{c+m}$.

\begin{prop}
Let $\mathcal{M}$ be a $\Bbbk$-linear $2$-Verma module with highest weight $c+m$. 
For each projective object $\mathbb{M} \in \mathcal{M}_{c+m}$, the functor
\[
\left(\bigoplus_{k \ge 0} \F^{k} \mathbb{M}\right) \otimes_{A(m)}  (-)  : A(m)\prmods \otimes_\Bbbk \mathbb M \rightarrow \mathcal{M},
\]
 is fully faithful.
\end{prop}

\begin{proof}
The image of the functor is completely determined by the image of the indecomposable projectives $P_{(k)}(m)$ for various $k$, which are themselves determined by the image of the free modules $A_k(m)$, which have $\F^k\mathbb M$ as images. Each of the $A_k(m)$'s and $\F^k \mathbb M$'s live in a different weight space, so that the hom-spaces in $ A(m)\prmods \otimes_\Bbbk \mathbb M$ and in the image of the functor are both determined by $\END(A_k(m)) \cong A_k(m) \otimes_\Bbbk \END(\mathbb M)$ and $\END(\F^k \mathbb M)$. Then, by Lemma~\ref{lem:homdim}, together with the faithful action of $A_k(n)$ on $\F^k\un_{c+m}$, we get that $\END_{c+m-2k}(\F^k\mathbb{M}) \cong A_k(m) \otimes_k \END(\mathbb{M})$. Therefore, the functor is fully faithful.
\end{proof}

Clearly, if $\mathcal{M}$ is idempotent complete and all objects of $\mathcal{M}_{c+m-2k}$ are direct summands of $\F^k \mathbb{M}$ for some $k \ge 0$, then the functor becomes an equivalence. Similarly, we get:

\begin{prop}
Let $\mathcal{M}$ be an abelian 2-Verma module with projective object $\mathbb{M} \in \mathcal{M}_{c + m}$. Then, there is a fully faithful functor 
\[
\left(\bigoplus_{k \ge 0} \F^{k} \mathbb{M}\right) \otimes_{A(m)}  (-)  : A(m)\lfmods \otimes \mathbb M \rightarrow \mathcal{M}.
\]
\end{prop}

As before, if $\mathcal{M}_{c+m-2k}$ corresponds with the abelian category generated by $\F^k(\mathbb{M})$,
 then the functor becomes an equivalence.

%
%



\vspace*{1cm}


\end{document}